\documentclass{amsart}
\usepackage{graphicx}
\usepackage{amsmath}
\usepackage{amssymb}
\usepackage{amsthm}

\newtheorem{theorem}{Theorem}[section]
\newtheorem{lem}[theorem]{Lemma}

\theoremstyle{Corollary}
\newtheorem{cor}[theorem]{Corollary}
\newtheorem{prop}[theorem]{Proposition}

\numberwithin{equation}{section}

\begin{document}

\title[Webster scalar curvature flow I]{The Webster scalar curvature flow on CR sphere. Part I}

\author{Pak Tung Ho}
\address{Department of Mathematics, Sogang University, Seoul
121-742, Korea}

\email{ptho@sogang.ac.kr, paktungho@yahoo.com.hk}


\date{November 30, 2012.}


\begin{abstract}
This is the first of two papers, in which we prove some properties of the
Webster scalar curvature flow. More precisely, we establish the long-time existence,
$L^p$ convergence and
the blow-up analysis for the solution
of the flow. As a by-product, we prove the convergence of
the CR Yamabe flow on the CR sphere. The results in this paper will be
used to prove a result of prescribing Webster scalar curvature on the CR sphere, which is the main result of
the second paper.
\end{abstract}

\maketitle

\section{Introduction}

Suppose $(M,g_0)$ is a compact $n$-dimensional Riemannian manifold without boundary, where $n\geq 3$.
As a generalization of Uniformization Theorem for surfaces, one wants to find
a metric $g$ conformal to $g_0$ such that its scalar curvature $R_g$ is constant. This is the so-called Yamabe problem, which was introduced by Yamabe in \cite{Yamabe}. Trudinger \cite{Trudinger} solved it when $R_g\leq 0$. For the case $R_g>0$, Aubin \cite{Aubin0} solved it when $n\ge 6$ and $M$ is not locally conformally flat. Schoen \cite{Schoen} solved the remaining cases, namely, when $3\leq n\leq 5$ or $M$ is locally conformally flat.
See the survey article \cite{Lee&Parker} by Lee and Parker
for more details.

A different approach has been introduced to solve the Yamabe problem. Hamitlon \cite{Hamilton} introduced the Yamabe flow, which is defined by
$$\frac{\partial}{\partial t}g(t)=-(R_{g(t)}-\overline{R}_{g(t)})g(t),\hspace{2mm}g(t)|_{t=0}=g_0,$$
where $\overline{R}_{g(t)}$ is the average of the scalar curvature $R_{g(t)}$ of $g(t)$.
The Yamabe flow was considered by Chow \cite{Chow}, Ye \cite{Ye},
Schwetlick and Struwe \cite{Schwetlick&Struwe}. Finally Brendle
\cite{Brendle4,Brendle5} showed that the Yamabe flow exists for all
time and converges to a metric of constant scalar curvature.

As a generalization of the Yamabe problem, one wants to know what function can be realized as the scalar curvature of some conformal metric. More precisely,
given a function $f$ on a compact $n$-dimensional Riemannian manifold $(M,g_0)$, can we find a metric $g$ conformal to $g_0$ such that $R_g=f$? This prescribing  scalar curvature problem has been studied extensively, even for the case when $M$ is a surface, see \cite{Ho3,Kazdan&Warner,Kazdan&Warner1}. Especially, the problem has attracted a lot of attention when $(M,g_0)$ is the $n$-dimensional standard sphere $S^n$, which is the so-called Nirenberg's problem. See
\cite{Chang&Gursky&Yang,Chang&Yang5,Chang&Yang3,Chang&Yang4,Struwe}. Especially,
Chen and Xu \cite{Chen&Xu} defined the scalar curvature flow, which is a natural
generation of Yamabe flow above, as follows:
\begin{equation*}
\frac{\partial}{\partial t}g(t)=(\alpha(t) f-R_{g(t)})g(t),
\end{equation*}
 where
$\alpha(t)$ is given by
$$\alpha(t)\int_{S^n}fdV_{g(t)}=\int_{S^n}R_{g(t)}dV_{g(t)}.$$
Using the scalar curvature flow, Chen and Xu \cite{Chen&Xu} proved the following:
\begin{theorem}[Chen-Xu]
Suppose that $f$ is a smooth positive function on the $n$-dimensional sphere $S^{n}$ with only non-degenerate critical points
with Morse indices $ind(f,x)$ and such
that $\Delta_{g_0}f (x)\neq 0$ at any such point $x$. Let
$$m_i =\#\{x\in S^{n}: \nabla_{g_0}f(x) = 0, \Delta_{g_0}f(x)< 0, ind(f,x) = n-i\}.$$
Further, suppose $\delta_n = 2^{\frac{2}{n}}$
if $3\leq n\leq 4$ or
$=2^{\frac{2}{n-2}}$
 for $n\geq 5$.
If there is no solution with coefficient $k_i\geq 0$ to the system of equations
\begin{equation*}
m_0 = 1 + k_0, m_i = k_{i-1} + k_i\mbox{ for }1\leq i\leq n, k_{n} = 0,
\end{equation*}
and
$f$ satisfies
\begin{equation*}
\max_{S^{n}}f/\min_{S^{n}}f<\delta_n,
\end{equation*}
then $f$ can be realized as the scalar curvature of some metric conformal to the standard metric
$g_0$.
\end{theorem}

The Yamabe problem can also be formulated in the context of CR manifold.
Suppose that $(M,\theta_0)$ is a compact strongly pseudoconvex CR manifold of real dimensional $2n+1$ with a given contact form $\theta_0$. The CR Yamabe problem is to find a contact form $\theta$ conformal to $\theta_0$ such that
its Webster scalar curvature $R_\theta$ is constant. This was introduced by Jerison and Lee in \cite{Jerison&Lee3}, and was solved by Jerison and Lee
for the case when $n\geq 2$ and $M$ is not locally CR equivalent to the sphere $S^{2n+1}$ in \cite{Jerison&Lee1,Jerison&Lee2,Jerison&Lee3}. The remaining case, namely, when $n=1$ or $M$ is locally CR equivalent to the sphere $S^{2n+1}$,  was solved by Gamara and Yacoub in \cite{Gamara2,Gamara1}.

As an analogue to Yamabe flow, one can consider the CR Yamabe flow defined by
$$\frac{\partial}{\partial t}\theta(t)=-(R_{\theta(t)}-\overline{R}_{\theta(t)})\theta(t),\hspace{2mm}\theta(t)|_{t=0}=\theta_0,$$
where $\overline{R}_{\theta(t)}$ is the average of the Webster scalar curvature $R_{\theta(t)}$ of $\theta(t)$.
Cheng and Chang \cite{Chang&Cheng} proved the short time existence of the CR Yamabe flow. Zhang \cite{Zhang} proved the long time existence and convergence for the case $R_{\theta_0}<0$. For the case $R_{\theta_0}>0$, Chang, Chiu, and Wu \cite{Chang&Chiu&Wu} proved the convergence when $M$ is spherical and $n=1$ and $\theta_0$ is torsion-free.
For general $n$, the author proved in \cite{Ho1} the long time existence for the case when $R_{\theta_0}>0$, and the convergence when $M$ is
the CR sphere.

As a generalization of CR Yamabe problem, one can consider prescribing CR Webster scalar curvature problem: given a function $f$
on  a CR manifold $(M,\theta_0)$, we want to find a contact form $\theta$ conformal to $\theta_0$ such that its Webster scalar curvature
$R_{\theta}=f$. This has been studied in \cite{Chtioui&Elmehdi&Gamara,Felli&Uguzzoni,Gamara3,Ho1,Malchiodi&Uguzzoni,Riahia&Gamara,Salem&Gamara}.
As an analogy of Nirenberg's problem, we study the problem of prescribing  Webster scalar curvature on the CR sphere $(S^{2n+1},\theta_0)$.
Motivated by the work of Chen and Xu in \cite{Chen&Xu}, we define the Webster scalar curvature flow:
\begin{equation*}
\frac{\partial }{\partial t}\theta(t)=(\alpha(t) f- R_{\theta(t)})\theta(t),
\end{equation*}
where $R_{\theta(t)}$ is the Webster scalar curvature of the contact form
$\theta(t)$, and $\alpha(t)$ is given by
\begin{equation*}
\alpha(t)\int_{S^{2n+1}}f dV_{\theta(t)}=\int_{S^{2n+1}}R_{\theta(t)} dV_{\theta(t)}.
\end{equation*}
We assume that $f$ is a smooth positive Morse function on the CR sphere $S^{2n+1}$ with only non-degenerate critical points in the sense that
if $f'(x)=0$, then the sub-Laplacian $\Delta_{\theta_0}f(x)\neq 0$. Here $f'(x)$ denotes the gradient of $f$ with respect to the standard Riemannian metric on
$S^{2n+1}$. By using the Webster scalar curvature flow, we will prove
the following theorem:

\begin{theorem}\label{thm1.1}
Suppose that $n\geq 2$ and $f$ is a smooth positive function on $S^{2n+1}$ with only non-degenerate critical points
with Morse indices $ind(f,x)$ and such
that $\Delta_{\theta_0}f (x)\neq 0$ at any such point $x$. Let
$$m_i =\#\{x\in S^{2n+1}: f'(x) = 0, \Delta_{\theta_0}f(x)< 0, ind(f,x) = 2n+1-i\}.$$
If there is no solution with coefficient $k_i\geq 0$ to the system of equations
\begin{equation*}
m_0 = 1 + k_0, m_i = k_{i-1} + k_i\mbox{ for }1\leq i\leq 2n+1, k_{2n+1} = 0,
\end{equation*}
and
$f$ satisfies the simple bubble condition, namely
\begin{equation}\tag{sbc}\label{sbc}
\max_{S^{2n+1}}f/\min_{S^{2n+1}}f<2^{\frac{1}{n}},
\end{equation}
then $f$ can be realized as the Webster scalar curvature of some contact form conformal to
$\theta_0$.
\end{theorem}

In this paper we first
establish some results which will be used to prove Theorem \ref{thm1.1}, and
the proof of Theorem \ref{thm1.1} will be postponed to the forthcoming paper due to its length.
The paper is organized as follows. In section \ref{section2}, we prove some properties of the Webster scalar
curvature flow, including the uniform lower bound of
the Webster scalar curvature for all time $t\geq 0$ and the long time existence.
In section \ref{section3}, we establish the convergence in the sense of $S_1^2$ and $L^p$, where $p\geq 1$,
with respect to the contact form at time $t$ as $t\rightarrow\infty$.
Section \ref{section4} is devoted to the analysis of the blow-ups of solutions. First we prove the CR analogue
of the compactness theorem obtained by Schwetlick and Struwe in \cite{Schwetlick&Struwe}, and we apply it to
show that one of the two cases must occur: either the flow itself converges in $S_2^p$ for some $p>2n+2$,
or the corresponding normalized flow $v(t)$, which will be introduced there, converges.
In section \ref{section6}, we prove the convergence of the CR Yamabe flow on the CR sphere by using the
techniques we developed in this paper. See Theorem \ref{thm6.3}.
In particular, this recovers the result of the author in \cite{Ho2}.
In the Appendix, we prove the CR analogue of the Aubin's improvement of the
Sobolev inequality. More precisely, we proved that the Folland-Stein inequality
can be improved by assuming that the conformal factor satisfies the balancing condition, i.e. the
center of mass is equal to zero. This is of independent interest and
is used in the compactness argument in section \ref{section4}.

This project began when I was invited by Prof. Paul Yang to visit Princeton University in the summer of 2011.
I would like to thank Prof. Paul Yang, who encouraged me to work on this project and
explained to me his work on prescribing scalar curvature.
I would like to thank Prof. Xingwang Xu, who answered many of my questions about his paper.
I am grateful to Prof. Sai-Kee Yeung for his encouragement and valuable advices.
I would like to thank Prof.  Ngaiming Mok for the invitation to the University of Hong Kong during the summer of 2012, where part of this
work was done. I would also like to thank Prof. Howard Jacobowitz and Prof. L\'{a}szl\'{o} Lempert for
helpful comments and discussions. This work was supported by the National Research Foundation of Korea (NRF) grant funded
by the Korea government (MEST) (No.2012R1A1A1004274).

\section{The flow}\label{section2}

\subsection{The flow and its basic properties}

Let  $\theta_0$ be the standard contact form on the sphere
$S^{2n+1}=\{x=(x_1,...,x_{n+1}):|x|^2=1\}\subset\mathbb{C}^{n+1}$, i.e.
$$\theta_{0}=\sqrt{-1}(\overline{\partial}-\partial)|x|^2
=\sqrt{-1}\sum_{j=1}^{n+1}(x_{j}d\overline{x}_{j}-\overline{x}_{j}dx_{j}).$$
Then $(S^{2n+1},\theta_0)$ is a compact strictly pseudoconvex CR
manifold of real dimension $2n+1$. We define a new contact form
$\theta$ which is conformal to $\theta_0$ as follows:
\begin{equation*}
\theta=u^{\frac{2}{n}}\theta_0, u>0.
\end{equation*}
The volume form with respect to $\theta$ is given by
$dV_\theta:=\theta\wedge d\theta^n$, and the volume of $S^{2n+1}$
with respect to $\theta$ is given by
Vol$(S^{2n+1},\theta)=\displaystyle\int_{S^{2n+1}}dV_\theta
=\int_{S^{2n+1}}u^{2+\frac{2}{n}}dV_{\theta_0}$.
The sub-Laplacian $\Delta_\theta$ with respect to $\theta$ is
defined as
$$\int_{S^{2n+1}}(\Delta_\theta u)\phi\,dV_\theta
=-\int_{S^{2n+1}}\langle \nabla_\theta u,\nabla_\theta \phi\rangle_\theta\,dV_\theta,$$
for any smooth function $\phi$. Here $\langle\mbox{ },\mbox{
}\rangle_\theta$ is the Levi form of $\theta$ and $\nabla_\theta$ is
the sub-gradient of $\theta$. The Folland-Stein space
$S_1^2(S^{2n+1},\theta_0)$ is the completion of $C^1(S^{2n+1})$ with
respect to the norm
\begin{eqnarray*}
\|u\|_{S_1^2(S^{2n+1},\theta_0)}=\left(\int_{S^{2n+1}}(|\nabla_{\theta_0}
u|^2_{\theta_0}+u^2)dV_{\theta_0} \right)^{\frac{1}{2}}.
\end{eqnarray*}
(For more properties about the Folland-Stein space, see
\cite{Folland&Stein}.) If $\Delta_{\theta}$ is the sub-Laplacian
with respect to the contact form $\theta$, then it can be shown that
\begin{equation}\label{2.0}
\Delta_{\theta}\phi=u^{-(1+\frac{2}{n})} \big(u\Delta_{\theta_0}
\phi+2\langle\nabla_{\theta_0} u,\nabla_{\theta_0}
\phi\rangle_{\theta_0}\big)
\end{equation}
for any smooth function $\phi$.
We refer the reader to \cite{Dragomir} for other definitions which
are not included here.

Now suppose $f$ is a smooth positive function on $S^{2n+1}$. We set
$$0<m=\displaystyle\min_{S^{2n+1}}f\leq f\leq \max_{S^{2n+1}}f=M.$$
Let $u_0\in C^\infty(S^{2n+1})$ such that
\begin{equation}\label{2.01}
\int_{S^{2n+1}}u_0^{2+\frac{2}{n}}dV_{\theta_0}=\int_{S^{2n+1}}dV_{\theta_0},
\end{equation}
we define the Webster scalar curvature flow as the evolution of the contact form
$\theta=\theta(t)$, $t\geq 0$ as follows:
\begin{equation}\label{2.1}
\frac{\partial }{\partial t}\theta=(\alpha f- R_{\theta})\theta,
\hspace{4mm} \theta\big|_{t=0}=u_0^{\frac{2}{n}}\theta_0,
\end{equation}
where $R_\theta$ is the Webster scalar curvature (sometimes it is
called pseudo-Hermitian scalar curvature) of the contact form
$\theta$ and $\alpha=\alpha(t)$ is given by
\begin{equation}\label{2.2}
\alpha\int_{S^{2n+1}}f dV_{\theta}=\int_{S^{2n+1}}R_\theta dV_{\theta}.
\end{equation}
If we write $\theta=u^{\frac{2}{n}}\theta_0$,
then (\ref{2.1}) is equivalent to
the following evolution equation of the conformal factor $u$:
\begin{equation}\label{2.3}
\frac{\partial u}{\partial t}=\frac{n}{2}(\alpha f- R_{\theta})u,
\hspace{4mm} u|_{t=0}=u_0.
\end{equation}
Since $\theta=u^{\frac{2}{n}}\theta_0$, the Webster scalar curvature
$R_\theta$ of $\theta$ satisfies the following CR Yamabe equation
\begin{equation}\label{2.4}
R_\theta=u^{-(1+\frac{2}{n})}\left(-(2+\frac{2}{n})\Delta_{\theta_0}u+R_{\theta_0}u\right),
\end{equation}
where $R_{\theta_0}=n(n+1)/2$ is the Webster scalar curvature of
$\theta_0$.

\begin{prop}\label{prop2.1}
The volume of $S^{2n+1}$does not change under the flow. That is,
\emph{Vol}$(S^{2n+1},\theta)=$\emph{Vol}$(S^{2n+1},\theta_0)$ for $t\geq 0$.
\end{prop}
\begin{proof}
By (\ref{2.2}) and (\ref{2.3}), we have
\begin{equation*}
\begin{split}
\frac{d}{dt}\mbox{Vol}(S^{2n+1},\theta)&=\frac{d}{dt}\left(\int_{S^{2n+1}}dV_\theta\right)
=\frac{d}{dt}\left(\int_{S^{2n+1}}u^{2+\frac{2}{n}}dV_{\theta_0}\right)\\
&=(2+\frac{2}{n})\int_{S^{2n+1}}u^{1+\frac{2}{n}}\frac{\partial u}{\partial t}dV_{\theta_0}
=(n+1)\int_{S^{2n+1}}(\alpha f- R_{\theta})dV_{\theta}=0.
\end{split}
\end{equation*}
This together with (\ref{2.01}) implies
Vol$(S^{2n+1},\theta)=$Vol$(S^{2n+1},\theta)\big|_{t=0}
=$Vol$(S^{2n+1},\theta_0)$ for $t\geq 0$.
\end{proof}

We define
\begin{equation}\label{2.5}
E(u)=\int_{S^{2n+1}}\left((2+\frac{2}{n})|\nabla_{\theta_0}u|^2_{\theta_0}
+R_{\theta_0}u^2\right)dV_{\theta_0}=\int_{S^{2n+1}}R_\theta dV_\theta
\end{equation}
where the last equality follows from (\ref{2.4}). By (\ref{2.2}) and (\ref{2.5}), we have
\begin{equation}\label{2.6}
\alpha=\frac{E(u)}{\int_{S^{2n+1}}f u^{2+\frac{2}{n}}dV_{\theta_0}}.
\end{equation}
We also define
\begin{equation}\label{2.7}
E_f(u)=\frac{E(u)}{(\int_{S^{2n+1}}fu^{2+\frac{2}{n}}dV_{\theta_0})^{\frac{n}{n+1}}}.
\end{equation}

\begin{prop}\label{prop2.2}
The functional $E_f$ is non-increasing along the flow. Indeed,
$$\frac{d}{dt}E_f(u)=-n\int_{S^{2n+1}}(\alpha f-R_\theta)^2u^{2+\frac{2}{n}}dV_{\theta_0}\Big/
\left(\int_{S^{2n+1}}fu^{2+\frac{2}{n}}dV_{\theta_0}\right)^{\frac{n}{n+1}}\leq 0.$$
\end{prop}
\begin{proof}
It follows from (\ref{2.3}), (\ref{2.4}) and (\ref{2.5}) that
\begin{equation}\label{2.8}
\begin{split}
\frac{d}{dt}E(u)=2\int_{S^{2n+1}}\Big(-(2+\frac{2}{n})\Delta_{\theta_0}u
+R_{\theta_0}u\Big)\frac{\partial u}{\partial t}dV_{\theta_0}
=n\int_{S^{2n+1}}(\alpha f-R_\theta)R_\theta dV_{\theta}.
\end{split}
\end{equation}
Therefore, by (\ref{2.3}) and (\ref{2.8}) we obtain
\begin{equation*}
\begin{split}
\frac{d}{dt}E_f(u)&
=\frac{d}{dt}\left(\frac{E(u)}{(\int_{S^{2n+1}}fu^{2+\frac{2}{n}}dV_{\theta_0})^{\frac{n}{n+1}}}\right)\\
&=\frac{d}{dt}E(u)\Big/\left(\int_{S^{2n+1}}fu^{2+\frac{2}{n}}dV_{\theta_0}\right)^{\frac{n}{n+1}}\\
&\hspace{2mm}-
\frac{n}{n+1}E(u)
\int_{S^{2n+1}}(2+\frac{2}{n})fu^{1+\frac{2}{n}}\frac{\partial u}{\partial t}dV_{\theta_0}\Big/
\left(\int_{S^{2n+1}}fu^{2+\frac{2}{n}}dV_{\theta_0}\right)^{\frac{n}{n+1}+1}\\
&=n\int_{S^{2n+1}}(\alpha f-R_\theta)R_\theta dV_{\theta}\Big/
\left(\int_{S^{2n+1}}fu^{2+\frac{2}{n}}dV_{\theta_0}\right)^{\frac{n}{n+1}}\\
&\hspace{2mm}-nE(u)\int_{S^{2n+1}}f(\alpha f-R_\theta)dV_{\theta}\Big/
\left(\int_{S^{2n+1}}fu^{2+\frac{2}{n}}dV_{\theta_0}\right)^{\frac{n}{n+1}+1}\\
&=-n\int_{S^{2n+1}}(\alpha f-R_\theta)^2u^{2+\frac{2}{n}}dV_{\theta_0}\Big/
\left(\int_{S^{2n+1}}fu^{2+\frac{2}{n}}dV_{\theta_0}\right)^{\frac{n}{n+1}},
\end{split}
\end{equation*}
where the last equality follows from (\ref{2.6}).
\end{proof}

By Proposition \ref{prop2.2}, we have for all $t\geq 0$, we have
\begin{equation*}
\int_0^t\frac{\int_{S^{2n+1}}(\alpha
f-R_\theta)^2u^{2+\frac{2}{n}}dV_{\theta_0}}{
(\int_{S^{2n+1}}fu^{2+\frac{2}{n}}dV_{\theta_0})^{\frac{n}{n+1}}}dt=\frac{1}{n}\Big(E_f(u)(0)-E_f(u)(t)\Big)\leq
\frac{1}{n}E_f(u)(0)<\infty
\end{equation*}
since $E_f(u)(t)\geq 0$ by (\ref{2.5}) and (\ref{2.7}). By
Proposition \ref{prop2.1},
\begin{equation}\label{2.9}
m\mbox{Vol}(S^{2n+1},\theta_0)\leq\displaystyle\int_{S^{2n+1}}fu^{2+\frac{2}{n}}dV_{\theta_0}\leq
M\mbox{Vol}(S^{2n+1},\theta_0)\mbox{ for all }t\geq 0,
\end{equation}
hence we have
\begin{equation}\label{2.10}
\int_0^\infty\int_{S^{2n+1}}(\alpha
f-R_\theta)^2u^{2+\frac{2}{n}}dV_{\theta_0}dt\leq C_0
\end{equation}
where $C_0$ is a constant depend only on $f$ and the initial data
$u_0$. Hence, there exists a sequence $\{t_j\}_{j=1}^\infty$ with
$t_j\rightarrow\infty$ such that
\begin{equation}\label{2.11}
\int_{S^{2n+1}}\big(\alpha(t_j)
f-R_{\theta(t_j)}\big)^2u^{2+\frac{2}{n}}(t_j)dV_{\theta_0}\rightarrow
0\mbox{ as }j\rightarrow\infty.
\end{equation}

\subsection{Uniform lower bound of the Webster scalar curvature.}

We have the following:
\begin{lem}\label{lem2.3}
For any $u\in S^2_1(S^{2n+1},\theta_0)$, we have
$$R_{\theta_0}\mbox{\emph{Vol}}(S^{2n+1},\theta_0)^{\frac{1}{n+1}}=
Y(S^{2n+1},\theta_0)\leq\frac{E(u)}{(\int_{S^{2n+1}}u^{2+\frac{2}{n}}dV_{\theta_0})^{\frac{n}{n+1}}}.$$
Here the CR Yamabe invariant of
$(S^{2n+1},\theta)$  defined as
$$
Y(S^{2n+1},\theta)=\inf_{0\not\equiv u\in
S^2_1(S^{2n+1})}\frac{\int_{S^{2n+1}}\left((2+\frac{2}{n})|\nabla_{\theta}
u|^2_{\theta}+R_{\theta} u^2\right)dV_{\theta}}
{\left(\int_{S^{2n+1}}u^{2+\frac{2}{n}}dV_{\theta}\right)^{\frac{n}{n+1}}}
$$
is a conformal invariant.
\end{lem}

This is a result of Jerison and Lee in \cite{Jerison&Lee1} and we
refer the reader to \cite{Jerison&Lee1} for its proof. See also
\cite{Frank&Lieb}.

Therefore, we have the following bounds for $E(u)$:
\begin{equation}\label{2.12a}
R_{\theta_0}\mbox{Vol}(S^{2n+1},\theta_0)\leq E(u)\leq E_f(u(0))\left((\max_{S^{2n+1}}f)\mbox{Vol}(S^{2n+1},\theta_0)\right)^{\frac{n}{n+1}}.
\end{equation}
To see this, we note that by (\ref{2.7}), Proposition \ref{prop2.1} and Proposition \ref{prop2.2} we have
\begin{equation*}
\begin{split}
E(u)=E_f(u)\left(\int_{S^{2n+1}}fu^{2+\frac{2}{n}}dV_{\theta_0}\right)^{\frac{n}{n+1}}
&\leq E_f(u(0))\left((\max_{S^{2n+1}}f)\int_{S^{2n+1}}u^{2+\frac{2}{n}}dV_{\theta_0}\right)^{\frac{n}{n+1}}\\
&=E_f(u(0))\left((\max_{S^{2n+1}}f)\mbox{Vol}(S^{2n+1},\theta_0)\right)^{\frac{n}{n+1}}.
\end{split}
\end{equation*}
On the other hand, it follows Proposition \ref{prop2.1} and Lemma \ref{lem2.3} that
\begin{equation*}
\begin{split}
E(u)\geq Y(S^{2n+1},\theta)\left(\int_{S^{2n+1}}u^{2+\frac{2}{n}}dV_{\theta_0}\right)^{\frac{n}{n+1}}=R_{\theta_0}\mbox{Vol}(S^{2n+1},\theta_0).
\end{split}
\end{equation*}

\begin{lem}\label{lem2.4}
There exists positive constants $\alpha_1$ and $\alpha_2$, depending
only on $f$ and the initial data, such that
$0<\alpha_1\leq\alpha\leq\alpha_2$ for all $t\geq 0$.
\end{lem}
\begin{proof}
By (\ref{2.6}),  (\ref{2.12a}) and Proposition \ref{prop2.1}, we have
\begin{equation*}
\alpha=\frac{E(u)}{\int_{S^{2n+1}}f
u^{2+\frac{2}{n}}dV_{\theta_0}}\leq
\frac{E_f(u(0))\left((\max_{S^{2n+1}}f)\mbox{Vol}(S^{2n+1},\theta_0)\right)^{\frac{n}{n+1}}}{(\min_{S^{2n+1}}f)\mbox{Vol}(S^{2n+1},\theta_0)}:=\alpha_2
\end{equation*}
for all $t\geq 0$. On the other hand, by (\ref{2.12a}) and Proposition \ref{prop2.1}, we have
\begin{equation*}
\begin{split}
\alpha=\frac{E(u)}{\int_{S^{2n+1}}f
u^{2+\frac{2}{n}}dV_{\theta_0}}&\geq\frac{R_{\theta_0}\mbox{Vol}(S^{2n+1},\theta_0)}
{(\max_{S^{2n+1}}f)\mbox{Vol}(S^{2n+1},\theta_0)}:=\alpha_1>0
\end{split}
\end{equation*}
for all $t\geq 0$.
\end{proof}

The following lemma is also proved in \cite{Ho1}:

\begin{lem}\label{lem2.5}
Under the flow $(\ref{2.3})$, the Webster scalar curvature
$R_\theta$ with respect to $\theta$ satisfies the following
evolution equation:
\begin{equation*}
\frac{\partial R_\theta}{\partial t}=-(n+1)\Delta_\theta (\alpha
f-R_\theta)+(R_\theta-\alpha f)R_\theta,
\end{equation*}
where $\Delta_\theta$ is the sub-Laplacian with respect to the
contact form $\theta$.
\end{lem}
\begin{proof}
If follows from (\ref{2.3}) and (\ref{2.4}) that
\begin{equation*}
\begin{split}
\frac{\partial R_\theta}{\partial t}&=\frac{\partial}{\partial
t}\left[u^{-(1+\frac{2}{n})}\Big(-(2+\frac{2}{n})\Delta_{\theta_0}u+R_{\theta_0}u\Big)\right]\\
&=-(1+\frac{2}{n})u^{-(2+\frac{2}{n})}\frac{\partial u}{\partial
t}\left(-(2+\frac{2}{n})\Delta_{\theta_0}u+R_{\theta_0}u\right)\\
&\hspace{4mm}+u^{-(1+\frac{2}{n})}\left(-(2+\frac{2}{n})\Delta_{\theta_0}(\frac{\partial
u}{\partial t})+R_{\theta_0}\frac{\partial u}{\partial t}\right)\\
&=(1+\frac{n}{2})(R_\theta-\alpha f)R_\theta\\
&\hspace{4mm}+u^{-(1+\frac{2}{n})}\left[-(n+1)\Delta_{\theta_0}((\alpha
f-R_\theta)u) +\frac{n}{2}R_{\theta_0}(\alpha f-R_\theta)u\right].
\end{split}
\end{equation*}
The second term of the last expression can be written as
\begin{equation*}
\begin{split}
&u^{-(1+\frac{2}{n})}\Big[-(n+1)\Big((\alpha
f-R_\theta)\Delta_{\theta_0}u+u\Delta_{\theta_0}(\alpha f-R_\theta)
\\
&\hspace{3cm}+2\langle \nabla_{\theta_0}(\alpha
f-R_\theta),\nabla_{\theta_0}u\rangle_{\theta_0}\Big)
+\frac{n}{2}R_{\theta_0}(\alpha f-R_\theta)u\Big]\\
=&\,\frac{n}{2}(\alpha f-R_\theta)u^{-(1+\frac{2}{n})}
\left(-(2+\frac{2}{n})\Delta_{\theta_0}u+R_{\theta_0}u\right)\\
&-(n+1)u^{-(1+\frac{2}{n})}\Big(u\Delta_{\theta_0}(\alpha
f-R_\theta)+2\langle \nabla_{\theta_0}(\alpha
f-R_\theta),\nabla_{\theta_0}u\rangle_{\theta_0}
\Big)\\
=&\,\frac{n}{2}(\alpha f-R_\theta)R_\theta-(n+1)\Delta_\theta
(\alpha f-R_\theta),
\end{split}
\end{equation*}
where the last equality follows from (\ref{2.0}) and (\ref{2.4}).
Combining all these, the assertion follows.
\end{proof}

The following formula follows from $(\alpha f)_t=\alpha' f$ and
Lemma \ref{lem2.5}:
\begin{equation}\label{2.12}
(\alpha f-R_\theta)_t=(n+1)\Delta_\theta (\alpha f-R_\theta)+(\alpha
f-R_\theta)R_\theta+\alpha' f.
\end{equation}

\begin{lem}\label{lem2.6}
There exists a constant $\alpha_0$ such that $\alpha'\leq\alpha_0$
for all $t>0$.
\end{lem}
\begin{proof}
It follows from (\ref{2.5}) and (\ref{2.6}) that
\begin{equation*}
\begin{split}
&\alpha'\int_{S^{2n+1}}f u^{2+\frac{2}{n}}dV_{\theta_0}
+\alpha(2+\frac{2}{n})\int_{S^{2n+1}}f
u^{1+\frac{2}{n}}\frac{\partial u}{\partial
t}dV_{\theta_0}\\
&=\frac{d}{dt}E(u)
=2\int_{S^{2n+1}}\left(-(2+\frac{2}{n})\Delta_{\theta_0}u+R_{\theta_0}u\right)\frac{\partial
u}{\partial t}dV_{\theta_0}.
\end{split}
\end{equation*}
By using (\ref{2.3}), (\ref{2.4}) and (\ref{2.6}), it can be
rewritten as
\begin{equation}\label{2.21}
\frac{\alpha'}{\alpha}E(u)+(n+1)\int_{S^{2n+1}}\alpha f (\alpha
f-R_{\theta})dV_{\theta}=n\int_{S^{2n+1}}(\alpha f-R_\theta)R_\theta
dV_{\theta},
\end{equation}
which implies that
\begin{equation}\label{2.22}
\alpha'=\frac{\alpha}{E(u)}\left[-n\int_{S^{2n+1}}(\alpha
f-R_\theta)^2dV_{\theta}-\int_{S^{2n+1}}\alpha f (\alpha
f-R_{\theta})dV_{\theta}\right].
\end{equation}
Applying Young's inequality, which says that
$ab\leq\displaystyle\frac{\epsilon}{2}b^2+\frac{1}{2\epsilon}a^2$
for any $\epsilon>0$, with $a=\alpha f$, $b=-(\alpha f-R_\theta)$
and $\epsilon=2n$ and integrating it over $S^{2n+1}$ with respect to
$\theta$, we obtain
$$-\int_{S^{2n+1}}\alpha f (\alpha
f-R_{\theta})dV_{\theta}\leq n\int_{S^{2n+1}}(\alpha
f-R_\theta)^2dV_{\theta}+\frac{\alpha^2}{4n}\int_{S^{2n+1}}f^2dV_{\theta}.$$
Hence, we have the following estimate:
\begin{equation*}
\begin{split}
\alpha'\leq\frac{\alpha^3}{4n E(u)}\int_{S^{2n+1}}f^2dV_{\theta}
&\leq\frac{\alpha_2^3M^2\mbox{Vol}(S^{2n+1},\theta_0)}{4n
E_f(u(0))\left(M\,\mbox{Vol}(S^{2n+1},\theta_0)\right)^{\frac{n}{n+1}}}:=\alpha_0,
\end{split}
\end{equation*}
where we have used (\ref{2.12a}), Proposition \ref{prop2.1}, and Lemma
\ref{lem2.4}.
\end{proof}

\begin{lem}\label{lem2.7}
The Webster scalar curvature of $\theta$ satisfies the following:
$$R_\theta-\alpha f\geq\min\Big\{R_{\theta_0}
-\alpha_2M,-\frac{1}{\alpha_1m}(\alpha_0 M+\alpha_2^2M^2)\Big\}
:=\gamma$$ for all $t\geq 0$.
\end{lem}
\begin{proof}
Define $L=\displaystyle\frac{\partial}{\partial
t}-(n+1)\Delta_\theta+\alpha f$. It follows from Lemma \ref{lem2.5} that
\begin{equation*}
LR_\theta=\frac{\partial}{\partial t}R_\theta-(n+1)\Delta_\theta
R_\theta+\alpha f R_\theta=-(n+1)\Delta_\theta(\alpha f)+R_\theta^2.
\end{equation*}
Hence, if one set $w(t)=\alpha f+\gamma$ for $t\geq 0$, then we have
\begin{equation*}
\begin{split}
Lw&=\alpha' f-(n+1)\Delta_\theta (\alpha f)+\alpha f(\alpha
f+\gamma)\\
&\leq \alpha_0
M+\alpha_2^2M^2+\alpha_1m\gamma-(n+1)\Delta_\theta(\alpha f)\\
&\leq -(n+1)\Delta_\theta(\alpha f)\leq LR_\theta,
\end{split}
\end{equation*}
where we have used the fact that $\gamma<0$,  Lemma \ref{lem2.4} and
\ref{lem2.6}. Note that $w(0)\leq\alpha_2 M+\gamma\leq R_{\theta_0}$
by Lemma \ref{lem2.4}. By maximum principle, $R_\theta-\alpha
f\geq\gamma$.
\end{proof}

\subsection{Long time existence}\label{section2.3}

\begin{lem}\label{lem2.8}
Given any $T>0$, there exist constants $c=c(T)$, $C=C(T)$ such that
$$c\leq u(x,t)\leq C$$
for any $(x,t)\in S^{2n+1}\times[0,T]$.
\end{lem}
\begin{proof}
By (\ref{2.3}) and Lemma \ref{lem2.7}, we deduce
$$\frac{\partial u}{\partial t}=\frac{n}{2}(\alpha f-
R_{\theta})u\leq-\frac{n}{2}\gamma u$$ for all $t$. This implies
that $u(t)\leq e^{-\frac{n}{2}\gamma t}\leq e^{-\frac{n}{2}\gamma
T}:=C(T)$ for all $0\leq t\leq T$, since $u(0)=u_0>0$ by (\ref{2.3}).

On the other hand, denote
by $P(x)=R_{\theta_0}+\displaystyle\sup_{0\leq t\leq
T}\sup_{S^{2n+1}}\Big[-(\alpha f+\gamma)u^{\frac{2}{n}}\Big]$. Then
by (\ref{2.4}) and Lemma \ref{lem2.7}, we get
\begin{equation*}
\begin{split}
0\leq(R_\theta-\alpha
f-\gamma)u^{1+\frac{2}{n}}&=-(2+\frac{2}{n})\Delta_{\theta_0}u+R_{\theta_0}u-(\alpha
f+\gamma)u^{1+\frac{2}{n}}
\\&\leq-(2+\frac{2}{n})\Delta_{\theta_0}u+Pu.
\end{split}
\end{equation*}
Applying  Proposition A.1 in \cite{Ho2}, we can conclude that there
exists a constant $c(T)$ such that
$\displaystyle\inf_{S^{2n+1}}u(t)\geq c(T)$ for all $0\leq t\leq T$.
\end{proof}

For $p\geq 1$, let
$$F_p(t)=\int_{S^{2n+1}}|\alpha f-R_\theta|^{p}dV_{\theta}.$$
Then for any $p\geq 2$, we have
\begin{equation}\label{2.13}
\begin{split}
\frac{d}{dt}&F_p(t)=\frac{d}{dt}\left(\int_{S^{2n+1}}|\alpha f-R_\theta|^{p}dV_{\theta}\right)\\
=&\,p\int_{S^{2n+1}}|\alpha f-R_\theta|^{p-2}(\alpha
f-R_\theta)\frac{\partial}{\partial t}(\alpha
f-R_\theta)dV_{\theta}+\int_{S^{2n+1}}|\alpha
f-R_\theta|^{p}\frac{\partial}{\partial t}(dV_{\theta})\\
=&\,p\int_{S^{2n+1}}|\alpha f-R_\theta|^{p-2}(\alpha f-R_\theta)
\big[(n+1)\Delta_\theta (\alpha f-R_\theta)+(\alpha
f-R_\theta)R_\theta\big]dV_{\theta}
\\
&+p\,\alpha'\int_{S^{2n+1}}f|\alpha f-R_\theta|^{p-2}(\alpha
f-R_\theta)dV_{\theta}
+(n+1)\int_{S^{2n+1}}|\alpha f-R_\theta|^{p}(\alpha f-R_\theta)dV_{\theta}\\
=&\,-(n+1)p(p-1)\int_{S^{2n+1}}|\alpha
f-R_\theta|^{p-2}|\nabla_\theta(\alpha
f-R_\theta)|^2_\theta dV_{\theta}\\
&+p\int_{S^{2n+1}} R_\theta|\alpha f-R_\theta|^{p}dV_{\theta}
+p\,\alpha'\int_{S^{2n+1}}f|\alpha f-R_\theta|^{p-2}(\alpha f-R_\theta)dV_{\theta}\\
&+(n+1)\int_{S^{2n+1}}|\alpha f-R_\theta|^{p}(\alpha
f-R_\theta)dV_{\theta},
\end{split}
\end{equation}
where we have used (\ref{2.3}) and (\ref{2.12}).

\begin{lem}\label{lem2.9}
For $p>n+1$, there holds
\begin{equation}\label{2.14}
\frac{d}{dt}F_p(t)+\left(\int_{S^{2n+1}}|\alpha
f-R_\theta|^{\frac{p(n+1)}{n}}dV_\theta\right)^{\frac{n}{n+1}}\leq
CF_p(t)+CF_p(t)^{\frac{p-n}{p-n-1}}.
\end{equation}
\end{lem}
\begin{proof}
It follows from (\ref{2.13}) that
\begin{equation*}
\begin{split}
\frac{d}{dt}F_p(t)
&=-\frac{2n(p-1)}{p}\int_{S^{2n+1}}\left((2+\frac{2}{n})\big|\nabla_\theta|\alpha
f-R_\theta|^{\frac{p}{2}}\big|^2_\theta+R_\theta|\alpha
f-R_\theta|^{p}\right)dV_\theta\\
&\hspace{2mm}+\big(p+\frac{2n(p-1)}{p}\big)\int_{S^{2n+1}}R_\theta|\alpha
f-R_\theta|^{p}dV_\theta\\
&\hspace{2mm}+(n+1)\int_{S^{2n+1}}|\alpha f-R_\theta|^{p}(\alpha
f-R_\theta)dV_{\theta}+p\,\alpha'\int_{S^{2n+1}}f|\alpha
f-R_\theta|^{p-2}(\alpha f-R_\theta)dV_{\theta}\\
&\leq-\frac{2n(p-1)Y(S^{2n+1},\theta_0)}{p}\left(\int_{S^{2n+1}}|\alpha
f-R_\theta|^{\frac{p(n+1)}{n}}dV_\theta\right)^{\frac{n}{n+1}}\\
&\hspace{2mm}+CF_p(t)+CF_{p+1}(t)
+\alpha'\int_{S^{2n+1}}f|\alpha
f-R_\theta|^{p-2}(\alpha f-R_\theta)dV_{\theta},
\end{split}
\end{equation*}
where we have used Lemma \ref{lem2.3}.
By (\ref{2.12a}), (\ref{2.22}), and H\"{o}lder's inequality, we have
\begin{equation*}
\begin{split}
&\alpha'\int_{S^{2n+1}}f|\alpha
f-R_\theta|^{p-2}(\alpha f-R_\theta)dV_{\theta}\\
&=\frac{\alpha}{E(u)}\left[-n\int_{S^{2n+1}}(\alpha
f-R_\theta)^2dV_{\theta}-\int_{S^{2n+1}}\alpha f (\alpha
f-R_{\theta})dV_{\theta}\right]\left(\int_{S^{2n+1}}f|\alpha
f-R_\theta|^{p-2}(\alpha f-R_\theta)dV_{\theta}\right)\\
&\leq C\left(\int_{S^{2n+1}}|\alpha
f-R_\theta|^2dV_{\theta}\right)
\left(\int_{S^{2n+1}}|\alpha
f-R_\theta|^{p-1}dV_{\theta}\right)\\
&\leq C\,\mbox{Vol}(S^{2n+1},\theta)^{\frac{p-2}{p}}\left(\int_{S^{2n+1}}|\alpha
f-R_\theta|^pdV_{\theta}\right)^{\frac{2}{p}}\mbox{Vol}(S^{2n+1},\theta)^{\frac{1}{p}}\left(\int_{S^{2n+1}}|\alpha
f-R_\theta|^pdV_{\theta}\right)^{\frac{p-1}{p}}\\
&=CF_{p}(t)^{\frac{p+1}{p}}\leq \left\{
        \begin{array}{ll}
          CF_p(t)^{\frac{p-n}{p-n-1}}, & \hbox{if $F_p(t)\geq 1$;} \\
          CF_p(t), & \hbox{if $F_p(t)<1$.}
        \end{array}
      \right.
\end{split}
\end{equation*}
On the other hand, for any $0<\epsilon<1$, by
H\"{o}lder's inequality, we have
\begin{equation*}
\begin{split}
&\int_{S^{2n+1}}|\alpha f-R_\theta|^{p+1}dV_\theta\\
&\leq \left(\int_{S^{2n+1}}|\alpha
f-R_\theta|^{\frac{p(n+1)}{n}}dV_\theta\right)^{\frac{n}{p}}\left(\int_{S^{2n+1}}|\alpha
f-R_\theta|^{p}dV_\theta\right)^{\frac{p-n}{p}}\\
&\leq \epsilon\left(\int_{S^{2n+1}}|\alpha
f-R_\theta|^{\frac{p(n+1)}{n}}dV_\theta\right)^{\frac{n}{n+1}}+C(\epsilon)
\left(\int_{S^{2n+1}}|\alpha
f-R_\theta|^{p}dV_\theta\right)^{\frac{p-n}{p-n-1}}
\end{split}
\end{equation*}
where we have used Young's inequality, which says that
$\displaystyle ab\leq\epsilon a^{\frac{p}{n+1}}+C(\epsilon)
b^{\frac{p}{p-n-1}}$ for any $a, b\geq 0$. Thus, if we choose
$\epsilon$ small enough,
one obtains the result.
\end{proof}

For $T>0$, let $$\delta=\sup_{0\leq t\leq T}\|\alpha
f+\gamma\|_{C^0(S^{2n+1})}+1$$ where $\gamma$ is the constant given
in Lemma \ref{lem2.7}. Then $R_\theta+\delta\geq R_\theta-(\alpha
f+\gamma)+1\geq 1$ by Lemma \ref{lem2.7}.

\begin{lem}\label{lem2.10}
For $p>2$, there holds
\begin{equation}\label{2.15}
\begin{split}
\frac{d}{dt}&\left(\int_{S^{2n+1}}(R_\theta+\delta)^{p}dV_{\theta}\right)
=-\frac{4(n+1)(p-1)}{p}\int_{S^{2n+1}}|\nabla_\theta(R_\theta+\delta)^{\frac{p}{2}}|^2_{\theta}dV_{\theta}\\
&+(n+1)p(p-1)\int_{S^{2n+1}}(R_\theta+\delta)^{p-2}\langle\nabla_\theta(R+\delta),\nabla_\theta(\alpha
f+\delta)\rangle_\theta
dV_\theta\\
&-(n+1-p)\int_{S^{2n+1}}(R_\theta+\delta)^{p}(R_\theta-\alpha
f)dV_{\theta}\\
&-p\,\delta\int_{S^{2n+1}}\big[(R_\theta+\delta)^{p-1}-(\alpha
f+\delta)^{p-1}\big](R_\theta-\alpha f)dV_{\theta}\\
&-p\,\delta\int_{S^{2n+1}}(\alpha f+\delta)^{p-1}(R_\theta-\alpha
f)dV_{\theta}.
\end{split}
\end{equation}
\end{lem}
\begin{proof}
Note that
\begin{equation*}
\begin{split}
\frac{d}{dt}\left(\int_{S^{2n+1}}(R_\theta+\delta)^{p}dV_{\theta}\right)
&=p\int_{S^{2n+1}}(R_\theta+\delta)^{p-1}\frac{\partial
R_\theta}{\partial t}dV_{\theta}+
\int_{S^{2n+1}}(R_\theta+\delta)^{p}\frac{\partial}{\partial t}(dV_{\theta})\\
&=p\int_{S^{2n+1}}(R_\theta+\delta)^{p-1}\big[-(n+1)\Delta_\theta
(\alpha f-R_\theta)+(R_\theta-\alpha f)R_\theta\big]dV_{\theta}\\
&\hspace{4mm}+(n+1)\int_{S^{2n+1}}(R_\theta+\delta)^{p}(\alpha
f-R_\theta)dV_{\theta}.
\end{split}
\end{equation*}
Here we have used (\ref{2.3}) and Lemma \ref{lem2.5}. The last
expression can be written as
\begin{equation*}
\begin{split}
&(n+1)p\int_{S^{2n+1}}(R_\theta+\delta)^{p-1}\Delta_\theta
[(R_\theta+\delta)-(\alpha
f+\delta)]dV_{\theta}\\
&+p\int_{S^{2n+1}}(R_\theta+\delta)^{p-1}(R_\theta-\alpha
f)(R_\theta+\delta-\delta)
dV_{\theta}+(n+1)\int_{S^{2n+1}}(R_\theta+\delta)^{p}(\alpha
f-R_\theta)dV_{\theta}\\
=&-\frac{4(n+1)(p-1)}{p}\int_{S^{2n+1}}|\nabla_\theta(R_\theta+\delta)^{\frac{p}{2}}|^2_{\theta}dV_{\theta}\\
&+(n+1)p(p-1)\int_{S^{2n+1}}(R_\theta+\delta)^{p-2}\langle\nabla_\theta(R+\delta),\nabla_\theta(\alpha
f+\delta)\rangle_\theta
dV_\theta\\
&-(n+1-p)\int_{S^{2n+1}}(R_\theta+\delta)^{p}(R_\theta-\alpha
f)dV_{\theta}-p\,\delta\int_{S^{2n+1}}(R_\theta+\delta)^{p-1}(R_\theta-\alpha
f)dV_{\theta}.
\end{split}
\end{equation*}
Note that the last integral can be written as
\begin{equation*}
\begin{split}
\int_{S^{2n+1}}(R_\theta+\delta)^{p-1}(R_\theta-\alpha
f)dV_{\theta}
=&\int_{S^{2n+1}}\big[(R_\theta+\delta)^{p-1}-(\alpha
f+\delta)^{p-1}\big](R_\theta-\alpha
f)dV_{\theta}\\&+\int_{S^{2n+1}}(\alpha
f+\delta)^{p-1}(R_\theta-\alpha f)dV_{\theta}.
\end{split}
\end{equation*}
Combining all these, we  prove the assertion.
\end{proof}

\begin{lem}\label{lem2.11}
For any fixed $T>0$, there exists constant $C=C(T)$ such that
$\displaystyle\int_{S^{2n+1}}|\alpha
f-R_\theta|^{p}dV_\theta\leq C(T)$ for all $0\leq
t\leq T$ and $n+1<p<\displaystyle\frac{(n+1)^2}{n}$.
\end{lem}
\begin{proof}
Take $p=n+1$ in (\ref{2.15}), we deduce that
\begin{equation*}
\begin{split}
\frac{d}{dt}\left(\int_{S^{2n+1}}(R_\theta+\delta)^{n+1}dV_{\theta}\right)
\leq&-4n\int_{S^{2n+1}}|\nabla_\theta(R_\theta+\delta)^{\frac{n+1}{2}}|^2_{\theta}dV_{\theta}\\
&+n(n+1)^2\int_{S^{2n+1}}(R_\theta+\delta)^{n-1}\langle\nabla_\theta(R+\delta),\nabla_\theta(\alpha
f+\delta)\rangle_\theta
dV_\theta\\
&-(n+1)\delta\int_{S^{2n+1}}(\alpha f+\delta)^{n}(R_\theta-\alpha
f)dV_{\theta}.
\end{split}
\end{equation*}
Note that
$$-(n+1)\delta\int_{S^{2n+1}}(\alpha f+\delta)^{n}(R_\theta-\alpha
f)dV_{\theta}\leq-(n+1)\delta\gamma\displaystyle\int_{S^{2n+1}}(\alpha
f+\delta)^{n}dV_{\theta}\leq C$$  where we have used Proposition
\ref{prop2.1}, Lemma \ref{lem2.4}, and Lemma \ref{lem2.7}. Note also that
for every $0<\epsilon<1$, by Young's inequality and H\"{o}lder's
inequality and Lemma \ref{lem2.8}, one obtains
\begin{equation*}
\begin{split}
&\left|\int_{S^{2n+1}}(R_\theta+\delta)^{n-1}\langle\nabla_\theta(R+\delta),\nabla_\theta(\alpha
f+\delta)\rangle_\theta dV_\theta\right|\\
&\leq\epsilon\int_{S^{2n+1}}\big|\nabla_\theta(R_\theta+\delta)^{\frac{n+1}{2}}\big|^2_\theta
dV_\theta+C(\epsilon)\int_{S^{2n+1}}\big|\nabla_\theta(\alpha
f+\delta)\big|^2_\theta
(R_\theta+\delta)^{n-1}dV_\theta\\
&\leq\epsilon\int_{S^{2n+1}}\big|\nabla_\theta(R_\theta+\delta)^{\frac{n+1}{2}}\big|^2_\theta
dV_\theta+C\left(\int_{S^{2n+1}}
(R_\theta+\delta)^{n+1}dV_\theta\right)^{\frac{n-1}{n+1}}.
\end{split}
\end{equation*}
If we choose $\epsilon=\displaystyle\frac{2}{(n+1)^2}$ and let
$y(t)=\displaystyle\int_{S^{2n+1}}
(R_\theta+\delta)^{n+1}dV_\theta$, then we have
\begin{equation}\label{2.16}
\frac{d}{dt}y(t)+2n\int_{S^{2n+1}}\big|\nabla_\theta(R_\theta+\delta)^{\frac{n+1}{2}}\big|^2_\theta
dV_\theta\leq Cy(t)^{\frac{n-1}{n+1}}+C.
\end{equation}

We claim that $y(t)\leq C(T)$ for all $0\leq t\leq T$. When $n=1$,
it follows from (\ref{2.16}) that $\displaystyle\frac{dy}{dt}\leq
C$, which implies that $y(t)\leq C(T)y(0)$ for $0\leq t\leq T$. When
$n>1$, by (\ref{2.16}) we have
$$\frac{d}{dt}\Big(y(t)^{\frac{2}{n+1}}\Big)=y(t)^{-\frac{n-1}{n+1}}\frac{d}{dt}y(t)
\leq C+Cy(t)^{-\frac{n-1}{n+1}}\leq C,$$ where the last inequality
follows from $y(t)=\displaystyle\int_{S^{2n+1}}
(R_\theta+\delta)^{n+1}dV_\theta\geq \int_{S^{2n+1}}
dV_\theta=\mbox{Vol}(S^{2n+1},\theta_0)$ since $R_\theta+\delta\geq
1$. Hence, $y(t)^{\frac{2}{n+1}}\leq C(T)$ for $0\leq t\leq T$. This
proves the claim.

Now by (\ref{2.16}) and the claim, we have
$$\max_{0\leq t\leq T}\int_{S^{2n+1}}(R_\theta+\delta)^{n+1}dV_\theta=\max_{0\leq t\leq T}y(t)\leq C(T)$$
and
$$\int_0^T\int_{S^{2n+1}}\big|\nabla_\theta(R_\theta+\delta)^{\frac{n+1}{2}}\big|^2_\theta
dV_\theta dt\leq C(T).$$ By Lemma \ref{lem2.8}, we have
$$\max_{0\leq t\leq T}\int_{S^{2n+1}}(R_\theta+\delta)^{n+1}dV_{\theta_0}\leq C(T)$$
and
$$\int_0^T\int_{S^{2n+1}}\big|\nabla_{\theta_0}(R_\theta+\delta)^{\frac{n+1}{2}}\big|^2_{\theta_0}
dV_{\theta_0} dt\leq C(T).$$
By these estimates and by Lemma \ref{lem2.3} and Lemma \ref{lem2.8}, we have
\begin{equation*}
\begin{split}
\int_0^T\left(\int_{S^{2n+1}}(R_\theta+\delta)^{\frac{n+1}{2}
\cdot(2+\frac{2}{n})}dV_{\theta}\right)^{\frac{n}{n+1}}dt
&\leq C(T)
\int_0^T\left(\int_{S^{2n+1}}(R_\theta+\delta)^{\frac{n+1}{2}
\cdot(2+\frac{2}{n})}dV_{\theta_0}\right)^{\frac{n}{n+1}}dt\\
&\leq C(T)\int_0^TE((R_\theta+\delta)^{\frac{n+1}{2}})dt\leq C(T).
\end{split}
\end{equation*}
This together with Proposition \ref{prop2.1} and Lemma \ref{lem2.4} implies that
\begin{equation}\label{2.17}
\begin{split}
&\int_0^T\left(\int_{S^{2n+1}}|\alpha f-R_\theta|^{\frac{(n+1)^2}{n}}dV_{\theta}\right)^{\frac{n}{n+1}}dt\\
&\leq \int_0^T\left(\int_{S^{2n+1}}(R_\theta+\delta)^{\frac{(n+1)^2}{n}}dV_{\theta}\right)^{\frac{n}{n+1}}dt+
\int_0^T\left(\int_{S^{2n+1}}(\alpha f+\delta)^{\frac{(n+1)^2}{n}}dV_{\theta}\right)^{\frac{n}{n+1}}dt
\leq C(T).
\end{split}
\end{equation}
Applying Lemma \ref{lem2.9} with $p=\displaystyle\frac{(n+1)^2}{n}$, we obtain
$$\frac{d}{dt}\log\left(\int_{S^{2n+1}}|\alpha f-R_\theta|^{\frac{(n+1)^2}{n}}dV_{\theta}\right)\leq C
+C\left(\int_{S^{2n+1}}(\alpha f+\delta)^{\frac{(n+1)^2}{n}}
dV_{\theta}\right)^{\frac{n}{n+1}}.$$
Integrating it from $0$ to $T$ and using the estimate (\ref{2.17}), we can conclude that
$$\int_{S^{2n+1}}|\alpha f-R_\theta|^{\frac{(n+1)^2}{n}}dV_{\theta}\leq C(T).$$
Now the assertion follows from this, H\"{o}lder's inequality and
Proposition \ref{prop2.1}.
\end{proof}

\begin{lem}\label{lem2.12}
For $0<\lambda<\displaystyle\frac{2}{n+1}$ and any fixed $T>0$, there exists a constant $C=C(T)>0$ such that
$$|u(x_1,t_1)-u(x_2,t_2)|\leq C\big((t_1-t_2)^\frac{\lambda}{2}+d_{S^{2n+1}}(x_1,x_2)^\lambda\big)$$
for all  $x_1, x_2\in S^{2n+1}$ and all $t_1, t_2\geq 0$ satisfying
$0<t_1-t_2<1$. Here $d_{S^{2n+1}}$ is the Carnot-Carath\'{e}odory distance
with respect to the contact form $\theta_0$.
\end{lem}
\begin{proof}
Choose $\lambda=2-\displaystyle\frac{2n+2}{p}$ with $n+1<p<\displaystyle\frac{(n+1)^2}{n}$.
For $0\leq t\leq T$ we have
\begin{equation}\label{2.18}
\begin{split}
\int_{S^{2n+1}}\big|-(2+\frac{2}{n})\Delta_{\theta_0}u+R_{\theta_0}
u\big|^pdV_{\theta_0}&=\int_{S^{2n+1}}|R_{\theta}u^{1+\frac{2}{n}}|^pdV_{\theta_0}\\
&\leq C\int_{S^{2n+1}}|R_{\theta}|^pdV_{\theta_0}\\
&\leq C\Big(\int_{S^{2n+1}}|\alpha f-R_{\theta}|^pdV_{\theta_0}+\int_{S^{2n+1}}|\alpha
f|^pdV_{\theta_0}\Big)\\
&\leq
C\Big(C+\big(\alpha_2\max_{S^{2n+1}}f\big)^p\,\mbox{Vol}({S^{2n+1}},\theta_0)\Big),
\end{split}
\end{equation}
where the first equality follows from (\ref{2.4}), the first
inequality follows from Lemma \ref{lem2.8}, and
the final inequality follows from
Proposition \ref{prop2.1}, Lemma \ref{lem2.4} and Lemma \ref{lem2.11}.
On the other hand, for $0\leq t\leq T$,
\begin{equation}\label{2.19}
\int_{S^{2n+1}}\Big|\frac{\partial u}{\partial t}\Big|^pdV_{\theta_0}
=\big(\frac{n}{2}\big)^p\int_{S^{2n+1}}|(\alpha f-R_\theta)u|^pdV_{\theta_0}
\leq C\int_{S^{2n+1}}|\alpha f-R_\theta|^pdV_{\theta}\leq C,
\end{equation}
where the first equality follows from (\ref{2.3}), the first
inequality follows from Lemma \ref{lem2.8}, and
the last inequality follows from Lemma \ref{lem2.11}.

Then (\ref{2.18}) implies that
$$|u(x_1,t)-u(x_2,t)|\leq Cd_{S^{2n+1}}(x_1,x_2)^{\lambda}$$
for all $x_1, x_2\in {S^{2n+1}}$ and $0\leq t\leq T$. Now using (\ref{2.19}), we
obtain
\begin{equation*}
\begin{split}
&|u(x,t_1)-u(x,t_2)|\leq
C(t_1-t_2)^{-(n+1)}\int_{B_{\sqrt{t_1-t_2}}(x)}|u(x,t_1)-u(x,t_2)|dV_{\theta_0}\\
&\leq
C(t_1-t_2)^{-(n+1)}\int_{B_{\sqrt{t_1-t_2}}(x)}|u(t_1)-u(t_2)|dV_{\theta_0}+C(t_1-t_2)^{\frac{\alpha}{2}}\\
&\leq C(t_1-t_2)^{-n}\sup_{t_2\leq t\leq
t_1}\int_{B_{\sqrt{t_1-t_2}}(x)}\left|\frac{\partial}{\partial
t}u(t)\right|dV_{\theta_0}+C(t_1-t_2)^{\frac{\beta}{2}}\\
&\leq C(t_1-t_2)^{\frac{\beta}{2}}\sup_{t_2\leq t\leq
t_1}\left(\int_{S^{2n+1}}\left|\frac{\partial}{\partial
t}u(t)\right|^pdV_{\theta_0}\right)^{\frac{1}{p}}+C(t_1-t_2)^{\frac{\beta}{2}}\\
&\leq C(t_1-t_2)^{\frac{\beta}{2}}
\end{split}
\end{equation*}
for all $x\in S^{2n+1}$ and all $0\leq t_1,t_2\leq T$ satisfying $0<t_1-t_2<1$.
This proves the assertion.
\end{proof}

In view of Lemma \ref{lem2.12}, it is easy to see that all
derivatives of $u(t)$ are uniformly bounded in every finite interval
$[0,T]$. Indeed, we can apply Theorem 1.1 in \cite{Bramanti}, which
says: let $X_1,X_2,\dots,X_q$ be a system of real smooth vector
fields satisfying H\"{o}rmander's condition in a bounded domain
$\Omega$ of $\Bbb R^n$. Let $A=\{a_{ij}(t,x)\}^q_{i,j=1}$ be a
symmetric, uniformly positive-definite matrix of real functions
defined in a domain $U\subset\Bbb R\times\Omega$. For operator of
the form
$$H=\partial_t-\sum^q_{i,j=1}a_{ij}(t,x)X_iX_j-\sum^q_{i=1}b_i(t,x)X_i-c(t,x)$$
we have a priori estimate of Schauder type in parabolic
H\"{o}rmander H\"{o}lder spaces $C^{k,\beta}_P(U)$. Namely, for
$a_{ij},b_i,c\in C^{k,\beta}_P(U)$ and $U'\Subset U$, we have
\begin{eqnarray}\label{2.20}
\| u\|_{C^{k+2,\beta}_P(U')}\le C\{\| Hu\|_{C^{k,\beta}_P(U)}+\|
u\|_{L^\infty(U)}\}.
\end{eqnarray}
Here, we have  (see P.193-194 in \cite{Bramanti})
\begin{equation}\label{2.24}
\begin{split}
C^{k,\beta}_P(U)&=\{u:U\rightarrow\mathbb{R}:
\|u\|_{C^{k,\beta}_P(U)}<\infty\},\\
\|u\|_{C^{k,\beta}_P(U)}&=\sum_{|I|+2h\leq
k}\left\|\partial_t^hX^Iu\right\|_{C^\beta_P(U)},\\
\|u\|_{C^{\beta}_P(U)}&=|u|_{C^\beta_P(U)}+\|u\|_{L^\infty(U)},\\
|u|_{C^{\beta}_P(U)}&=\sup\left\{\frac{|u(t,x)-u(s,y)|}{d_P((t,x),(s,y))^{\beta}}:(t,x),(s,y)\in
U, (t,x)\neq(s,y)\right\},
\end{split}
\end{equation}
where $d_P$ is the parabolic Carnot-Carath\'{e}odory distance (see
P. 189 in \cite{Bramanti}) which is given by
$$d_P((t_1,x_1),(t_2,x_2))=\sqrt{d(x_1,x_2)^2+|t_1-t_2|}.$$
Here $d$ is the Carnot-Carath\'{e}odory distance in $\Omega$.
Moreover, for any multiindex $I=(i_1,i_2,...,i_s)$, with $1\leq
i_j\leq q$, $X^Iu=X_{i_1}X_{i_2}\cdots X_{i_s}u$.

It follows from Lemma \ref{lem2.12} that $u(t,x)\in
C_P^{0,\lambda}([0,T]\times S^{2n+1})$. Therefore, with the
estimates (\ref{2.20}), Lemma \ref{lem2.8} and \ref{2.12}, we can
now use the similar standard regularity theory for weakly parabolic
equation to show that all higher order derivatives of $u(t)$ are
uniformly bounded on $[0,T]$. This shows the long time existence of
the flow (\ref{2.1}) and (\ref{2.3}).

\section{$S^2_1$ and $L^p$ convergence}\label{section3}

We first consider $F_2(t)$. It follows from (\ref{2.13}) with $p=2$
that
\begin{equation}\label{3.1}
\begin{split}
&\frac{1}{2}\frac{d}{dt}F_2(t)=\frac{d}{dt}\left(\frac{1}{2}\int_{S^{2n+1}}(\alpha
f-R_\theta)^{2}dV_{\theta}\right)\\
&=\alpha'\int_{S^{2n+1}}f(\alpha
f-R_\theta)dV_{\theta}+\int_{S^{2n+1}}(\alpha
f-R_\theta)^2R_\theta dV_{\theta}\\
&\hspace{2mm}-(n+1)\int_{S^{2n+1}}|\nabla_\theta(\alpha
f-R_\theta)|^2_\theta
dV_{\theta}+\frac{n+1}{2}\int_{S^{2n+1}}(\alpha
f-R_\theta)^3dV_{\theta}.
\end{split}
\end{equation}
Now if we substitute (\ref{2.21}) into the first term on the right
hand side of (\ref{3.1}), we obtain
\begin{equation}\label{3.2}
\begin{split}
\frac{1}{2}\frac{d}{dt}F_2(t)
&=-\frac{1}{n+1}\left(\frac{\alpha'}{\alpha}\right)^2E(u)
+\frac{n}{n+1}\frac{\alpha'}{\alpha}\int_{S^{2n+1}}(\alpha
f-R_\theta)R_\theta dV_{\theta}\\
&\hspace{2mm}+\int_{S^{2n+1}}(\alpha f-R_\theta)^2R_\theta
dV_{\theta}-(n+1)\int_{S^{2n+1}}|\nabla_\theta(\alpha
f-R_\theta)|^2_\theta
dV_{\theta}\\
&\hspace{2mm}+\frac{n+1}{2}\int_{S^{2n+1}}(\alpha
f-R_\theta)^3dV_{\theta}.
\end{split}
\end{equation}

First we show that the Webster scalar curvature $R_\theta$ converges
to $\alpha f$ in the $L^2$ sense.

\begin{lem}\label{lem3.1}
For a positive smooth solution $u$ of $(\ref{2.3})$, there holds
$$F_2(t)=\int_{S^{2n+1}}(\alpha
f-R_\theta)^{2}dV_{\theta}\rightarrow 0, \mbox{ as
}t\rightarrow\infty.$$
\end{lem}
\begin{proof}
By (\ref{2.22}), (\ref{2.12a}), H\"{o}lder's inequality, Proposition \ref{prop2.1},
Lemma \ref{lem2.4}, we have
\begin{equation}\label{3.3}
\begin{split}
|\alpha'|&\leq\frac{\alpha}{E(u)}\left[n\int_{S^{2n+1}}(\alpha
f-R_\theta)^2dV_{\theta}+\int_{S^{2n+1}}\alpha f |\alpha
f-R_{\theta}|dV_{\theta}\right]\\
&\leq\frac{\alpha_2}{R_{\theta_0}\mbox{Vol}(S^{2n+1},\theta_0)}\Big[nF_2(t)+\alpha_2M\left(\mbox{Vol}(S^{2n+1},\theta_0)F_2(t)\right)^{\frac{1}{2}}\Big].
\end{split}
\end{equation}
Thus it follows from (\ref{3.1}) that
\begin{equation}\label{3.5}
\begin{split}
&\frac{1}{2}\frac{d}{dt}F_2(t)+(n+1)\int_{S^{2n+1}}|\nabla_\theta(\alpha
f-R_\theta)|^2_\theta
dV_{\theta}\\
&=\alpha'\int_{S^{2n+1}}f(\alpha
f-R_\theta)dV_{\theta}+\alpha\int_{S^{2n+1}}f(\alpha
f-R_\theta)^2dV_{\theta}+\frac{n-1}{2}\int_{S^{2n+1}}(\alpha
f-R_\theta)^3dV_{\theta}\\
&\leq C|\alpha'|\left(\int_{S^{2n+1}}(\alpha
f-R_\theta)^2dV_{\theta}\right)^{\frac{1}{2}}+\alpha\int_{S^{2n+1}}f(\alpha
f-R_\theta)^2dV_{\theta}-\frac{(n-1)\gamma}{2}\int_{S^{2n+1}}(\alpha
f-R_\theta)^2dV_{\theta}\\
&\leq C\int_{S^{2n+1}}(\alpha
f-R_\theta)^2dV_{\theta}\left[1+\left(\int_{S^{2n+1}}(\alpha
f-R_\theta)^2dV_{\theta}\right)^{\frac{1}{2}}\right]
\end{split}
\end{equation}
where we have used H\"{o}lder's inequality  and Lemma \ref{lem2.6}
in the first inequality, and the last inequality follows from
(\ref{3.3}) and Lemma \ref{lem2.4}.

Hence, if we set
$v(t)=\displaystyle\int_0^{F_2(t)}\frac{ds}{1+s^{1/2}}$, then
$$\frac{dv(t)}{dt}=\frac{1}{1+\sqrt{F_2(t)}}\frac{d}{dt}F_2(t)\leq
CF_2(t).$$ Thus we obtain, for all $t\geq t_j$
$$v(t)\leq v(t_j)+C\int_{t_j}^\infty F_2(t)dt,$$
where $t_j$ is a sequence defined by (\ref{2.11}). Note that
$v(t_j)=\displaystyle\int_0^{F_2(t_j)}\frac{ds}{1+s^{1/2}}\leq
F(t_j)\rightarrow 0$ as $j\rightarrow\infty$ by (\ref{2.11}). Note
also that $\displaystyle\int_{t_j}^\infty F_2(t)dt\rightarrow 0$ as
$j\rightarrow\infty$ by (\ref{2.10}). Therefore, $v(t)\rightarrow 0$
as $t\rightarrow\infty$. On the other hand, by definition of $v(t)$,
we have
\begin{equation}\label{3.4}
v(t)=\int_0^{F_2(t)}\frac{ds}{1+s^{1/2}}\geq\frac{F_2(t)}{1+F_2(t)^{\frac{1}{2}}}
=F_2(t)^{\frac{1}{2}}-\frac{F_2(t)^{\frac{1}{2}}}{1+F_2(t)^{\frac{1}{2}}},
\end{equation}
or equivalently,
$$F_2(t)^{\frac{1}{2}}\leq
v(t)+\frac{F_2(t)^{\frac{1}{2}}}{1+F_2(t)^{\frac{1}{2}}}\leq
v(t)+1.$$ This implies that $F_2(t)$ is bounded. By (\ref{3.4}), we
have $F_2(t)\leq(1+F_2(t)^{\frac{1}{2}})v(t)\leq Cv(t)\rightarrow 0$
as $t\rightarrow\infty$. This proves the assertion.
\end{proof}

For $p\geq 1$, let
$$G_p(t)=\int_{S^{2n+1}}|\nabla_\theta(R_\theta-\alpha f)|_\theta^pdV_\theta.$$
Integrating (\ref{3.5}) from $0$ to $\infty$ yields
\begin{equation}\label{3.6}
\int_0^\infty
G_2(t)dt=\int_0^\infty\int_{S^{2n+1}}|\nabla_\theta(\alpha
f-R_\theta)|^2_\theta dV_{\theta}dt\leq C+C\int_0^\infty
F_2(t)dt<\infty
\end{equation}
 by (\ref{2.10}) and Lemma \ref{lem3.1}.
By (\ref{3.1}), we have
\begin{equation*}
\begin{split}
&\frac{n-1}{2}\int_{S^{2n+1}}(\alpha f-R_\theta)^3dV_{\theta}
\\
&=\frac{1}{2}\frac{d}{dt}F_2(t)-\alpha\int_{S^{2n+1}}f(\alpha
f-R_\theta)^2 dV_{\theta}+(n+1)G_2(t)-\alpha'\int_{S^{2n+1}}f(\alpha
f-R_\theta)dV_{\theta}\\
&\leq\frac{1}{2}\frac{d}{dt}F_2(t)
+(n+1)G_2(t)+|\alpha'|\Big(\max_{S^{2n+1}}f\Big)F_1(t)\\
&\leq\frac{1}{2}\frac{d}{dt}F_2(t)+(n+1)G_2(t)\\
&\hspace{4mm}+\frac{2\alpha_2}{n(n+1)\mbox{Vol}(S^{2n+1},\theta_0)}
\Big[nF_2(t)+\alpha_2M\left(\mbox{Vol}(S^{2n+1},\theta_0)F_2(t)\right)^{\frac{1}{2}}\Big]
\Big(\max_{S^{2n+1}}f\Big)\left(\mbox{Vol}(S^{2n+1},\theta_0)F_2(t)\right)^{\frac{1}{2}}\\
&\leq \frac{1}{2}\frac{d}{dt}F_2(t)+CF_2(t)+(n+1)G_2(t),
\end{split}
\end{equation*}
where we have used
(\ref{3.3}), Proposition \ref{prop2.1} and H\"{o}lder's inequality
in the second inequality, and Lemma \ref{lem3.1} in the last
inequality. Similarly, we can show that
\begin{equation*}
\begin{split}
&\frac{n-1}{2}\int_{S^{2n+1}}(\alpha f-R_\theta)^3dV_{\theta}
\\
&=\frac{1}{2}\frac{d}{dt}F_2(t)-\alpha\int_{S^{2n+1}}f(\alpha
f-R_\theta)^2 dV_{\theta}+(n+1)G_2(t)-\alpha'\int_{S^{2n+1}}f(\alpha
f-R_\theta)dV_{\theta}\\
&\geq\frac{1}{2}\frac{d}{dt}F_2(t)
-|\alpha|\Big(\max_{S^{2n+1}}f\Big)F_2(t)-|\alpha'|\Big(\max_{S^{2n+1}}f\Big)F_1(t)\geq
\frac{1}{2}\frac{d}{dt}F_2(t)-CF_2(t).
\end{split}
\end{equation*}
Integrating these two inequalities and using (\ref{2.10}),
(\ref{3.6}) and Lemma \ref{lem3.1}, we conclude that
\begin{equation}\label{3.8}
\left|\int_0^\infty\int_{S^{2n+1}}(\alpha
f-R_\theta)^{3}dV_{\theta}dt\right|<\infty.
\end{equation}

\begin{lem}\label{lem3.2}
For any $p<\infty$, there holds $F_p(t)\rightarrow 0$ as
$t\rightarrow\infty$.
\end{lem}
\begin{proof}
We divide the proof into two steps.

\textit{Step 1.} We are going to show that there exists $p_0>n+1$ and
$v_0\in(0,1]$ such that
\begin{equation}\label{3.9}
\int_0^\infty\left(\int_{S^{2n+1}}|\alpha
f-R_\theta|^{p_0}dV_\theta\right)^{v_0}dt<\infty.
\end{equation}
To do this, we
will show that the following estimates are true for all positive integer
$k<\displaystyle\frac{n+1}{2}$:
\begin{eqnarray}
&&\label{3.10}\int_0^\infty\int_{S^{2n+1}}|\alpha
f-R_\theta|^{2k}dV_{\theta}dt<\infty,\\
&&\label{3.11}\int_{S^{2n+1}}|\alpha
f-R_\theta|^{2k}dV_{\theta}\leq C,\\
&&\label{3.12}\int_0^\infty\int_{S^{2n+1}}(\alpha
f-R_\theta)^{2(k-1)}|\nabla_\theta(\alpha f-R_\theta)|^2_\theta
dV_{\theta}dt<\infty,\\
&&\label{3.13}\left|\int_0^\infty\int_{S^{2n+1}}(\alpha
f-R_\theta)^{2k+1}dV_{\theta}dt\right|<\infty,
\end{eqnarray}
where $C$ is a positive constant independent of $t$.

For $k=1$, (\ref{3.10})-(\ref{3.13}) follow from (\ref{2.10}),
(\ref{3.6}), (\ref{3.8}) and Lemma \ref{lem3.1}.

Now suppose that the estimates (\ref{3.10})-(\ref{3.13}) are true for
$k$ with $k<(n+1)/2$. By (\ref{3.10}) and (\ref{3.12}) and the sharp
Folland-Stein inequality (see \cite{Folland&Stein} or Theorem 3.13 in \cite{Dragomir}), we obtain
\begin{equation}\label{3.14}
\int_0^\infty\left(\int_{S^{2n+1}}|\alpha
f-R_\theta|^{\left(\frac{2n+2}{n}\right)k}dV_{\theta}\right)^{\frac{n}{n+1}}dt\leq
C.
\end{equation}
In the following three cases, we intend to show either this
procedure terminates by suitable choice of $p_0>n+1$ and
$v_0\in(0,1]$, or the estimates (\ref{3.10})-(\ref{3.13}) are true
for $k+1$.

\textit{Case (a).} If $k>n/2$, we can set
$p_0=k(2n+2)/n>n+1$ and $v_0=n/(n+1)$, then the iteration terminates
by (\ref{3.14}).

\textit{Case (b).} If $k<n/2$, by (\ref{2.13}) with $p=2k+1$, we have
\begin{equation}\label{3.15}
\begin{split}
&\frac{d}{dt}\left(\int_{S^{2n+1}}(\alpha f-R_\theta)^{2k+1}dV_{\theta}\right)\\
&
+2k(2k+1)(n+1)\int_{S^{2n+1}}(\alpha
f-R_\theta)^{2k-1}|\nabla_\theta(\alpha
f-R_\theta)|^2_\theta dV_{\theta}\\
=&\,(2k+1)\int_{S^{2n+1}} \alpha f(\alpha f-R_\theta)^{2k+1}dV_{\theta}
+(2k+1)\alpha'\int_{S^{2n+1}}f(\alpha f-R_\theta)^{2k}dV_{\theta}\\
&+(n-2k)\int_{S^{2n+1}}(\alpha f-R_\theta)^{2k+2}dV_{\theta}.
\end{split}
\end{equation}
Note that by  H\"{o}lder's inequality, Young's inequality, and Lemma \ref{lem2.4}, we have
\begin{equation}\label{3.16}
\begin{split}
&(2k+1)\left|\int_{S^{2n+1}} \alpha f(\alpha f-R_\theta)^{2k+1}dV_{\theta}\right|\\
&\leq(2k+1)\left(\int_{S^{2n+1}} (\alpha f-R_\theta)^{2k+2}dV_{\theta}\right)^{\frac{1}{2}}
\left(\int_{S^{2n+1}}\alpha^2 f^2(\alpha f-R_\theta)^{2k}dV_{\theta}\right)^{\frac{1}{2}}\\
&\leq\frac{n-2k}{2}\int_{S^{2n+1}} (\alpha f-R_\theta)^{2k+2}dV_{\theta}
+C\int_{S^{2n+1}} (\alpha f-R_\theta)^{2k}dV_{\theta}.
\end{split}
\end{equation}
By (\ref{3.3}) and Lemma \ref{lem3.1}, we have
\begin{equation}\label{3.17}
|\alpha'|\leq C
\end{equation}
where $C$ is a constant independent of $t$. Combining (\ref{3.15}),
(\ref{3.16}), (\ref{3.17}) and Lemma \ref{lem2.7}, we obtain
\begin{equation*}
\begin{split}
&\frac{n-2k}{2}\int_{S^{2n+1}}(\alpha f-R_\theta)^{2k+2}dV_{\theta}\\
\leq&\,\frac{d}{dt}\left(\int_{S^{2n+1}}(\alpha f-R_\theta)^{2k+1}dV_{\theta}\right)
+C\int_{S^{2n+1}}(\alpha f-R_\theta)^{2k}dV_{\theta}\\
&-2k(2k+1)(n+1)\gamma\int_{S^{2n+1}}(\alpha
f-R_\theta)^{2k-2}|\nabla_\theta(\alpha
f-R_\theta)|^2_\theta dV_{\theta}.
\end{split}
\end{equation*}
Integrating it from $0$ to $t$, we get
\begin{equation*}
\begin{split}
&\frac{n-2k}{2}\int_0^t\int_{S^{2n+1}}(\alpha f-R_\theta)^{2k+2}dV_{\theta}dt\\
\leq&\,\int_{S^{2n+1}}(\alpha f-R_\theta)^{2k+1}(t)dV_{\theta(t)}
-\int_{S^{2n+1}}(\alpha f-R_\theta)^{2k+1}(0)dV_{\theta(0)}\\
&+C\int_0^t\int_{S^{2n+1}}|\alpha f-R_\theta|^{2k}dV_{\theta}dt\\
&-2k(2k+1)(n+1)\gamma\int_0^t\int_{S^{2n+1}}|\alpha
f-R_\theta|^{2(k-1)}|\nabla_\theta(\alpha
f-R_\theta)|^2_\theta dV_{\theta}dt
\\
\leq &\,-\gamma\int_{S^{2n+1}}(\alpha f-R_\theta)^{2k}(t)dV_{\theta(t)}+C\leq C
\end{split}
\end{equation*}
where we have used Lemma \ref{lem2.7} and (\ref{3.10})-(\ref{3.13}) for $k$. By letting $t\rightarrow\infty$,
we have shown that (\ref{3.10})
holds for $k+1$.

\textit{Case (b)(i).} If $2k+2>n+1$, choose $p_0=2k+2$ and $v_0=1$ in (\ref{3.9}),
then the iteration terminates.

\textit{Case (b)(ii).} If $2k+2=n+1$, by (\ref{2.13}) with $p=2k+\displaystyle\frac{4}{3}$, we have
\begin{equation}\label{3.18}
\begin{split}
&\frac{d}{dt}F_{2k+\frac{4}{3}}(t)+C\int_{S^{2n+1}}|\alpha
f-R_\theta|^{2k-\frac{2}{3}}|\nabla_\theta(\alpha
f-R_\theta)|^2_\theta dV_{\theta}\\
&=(2k+\frac{4}{3})\int_{S^{2n+1}} \alpha f|\alpha f-R_\theta|^{2k+\frac{4}{3}}dV_{\theta}\\
&\hspace{4mm}
+(2k+\frac{4}{3})\alpha'\int_{S^{2n+1}}f|\alpha f-R_\theta|^{2k-\frac{2}{3}}(\alpha f-R_\theta)dV_{\theta}\\
&\hspace{4mm}+\frac{2}{3}\int_{S^{2n+1}}|\alpha f-R_\theta|^{2k+\frac{4}{3}}(\alpha f-R_\theta)dV_{\theta}.
\end{split}
\end{equation}
By Lemma \ref{lem2.7} and Young's inequality
\begin{equation}\label{3.19}
F_{2k+\frac{4}{3}}(t)\leq \frac{1}{3}F_{2k}(t)+\frac{2}{3}F_{2k+2}(t),
\end{equation}
we have
\begin{equation*}
\int_{S^{2n+1}}|\alpha f-R_\theta|^{2k+\frac{4}{3}}(\alpha
f-R_\theta)dV_{\theta}\leq-\gamma F_{2k+\frac{4}{3}}(t)
\leq -\gamma\left( \frac{1}{3}F_{2k}(t)+\frac{2}{3}F_{2k+2}(t)\right).
\end{equation*}
By (\ref{3.17}) and Young's inequality
$
F_{2k+\frac{1}{3}}(t)\leq \frac{5}{6}F_{2k}(t)+\frac{1}{6}F_{2k+2}(t),
$
we have
\begin{equation*}
\left|\alpha'\int_{S^{2n+1}}f|\alpha f-R_\theta|^{2k-\frac{2}{3}}(\alpha f-R_\theta)dV_{\theta}\right|
\leq  |\alpha'|\Big(\max_{S^{2n+1}}f\Big)F_{2k+\frac{1}{3}}(t)
\leq C[F_{2k}(t)+F_{2k+2}(t)].
\end{equation*}
Combining all these inequality,
we get by (\ref{3.18})
\begin{equation}\label{3.20}
\begin{split}
&\frac{d}{dt}\left(\int_{S^{2n+1}}(\alpha f-R_\theta)^{2k+\frac{4}{3}}dV_{\theta}\right)
+C\int_{S^{2n+1}}|\alpha
f-R_\theta|^{2k-\frac{2}{3}}|\nabla_\theta(\alpha
f-R_\theta)|^2_\theta dV_{\theta}\\
&\leq C[F_{2k}(t)+F_{2k+2}(t)].
\end{split}
\end{equation}
Due to (\ref{3.10}) for $k$ and $k+1$, integrating (\ref{3.20}) from $0$ to $\infty$ with respect to $t$ yields
\begin{equation}\label{3.21}
\int_0^\infty\int_{S^{2n+1}}|\alpha
f-R_\theta|^{2k-\frac{2}{3}}|\nabla_\theta(\alpha
f-R_\theta)|^2_\theta dV_{\theta}dt\leq C\hspace{2mm}\mbox{ and }\hspace{2mm}F_{2k+\frac{4}{3}}(t)\leq C.
\end{equation}
Also, due to (\ref{3.10}) for $k$ and $k+1$, integrating
(\ref{3.19}) from $0$ to $\infty$ with respect to
$t$ yields
\begin{equation}\label{3.22}
\int_0^\infty\int_{S^{2n+1}}|\alpha
f-R_\theta|^{2k+\frac{4}{3}}dV_{\theta}dt\leq C.
\end{equation}
Combining (\ref{3.21}) and (\ref{3.22}) and using the sharp Folland-Stein inequality again, we obtain
$$\int_0^\infty\left(\int_{S^{2n+1}}|\alpha f-R_\theta|^{(k+\frac{2}{3})(2+\frac{2}{n})}
dV_{\theta}\right)^{\frac{n}{n+1}}dt\leq C.$$
Now we can choose $p_0=\displaystyle(k+\frac{2}{3})(2+\frac{2}{n})>\frac{n+1}{2}$
and $v_0=\displaystyle\frac{n}{n+1}$ and then the iteration terminates.

\textit{Case (b)(iii).} If $2k+2<n+1$, by (\ref{2.13}) again with $p=2k+2$, we have
\begin{equation}\label{3.23}
\begin{split}
&\frac{d}{dt}F_{2k+2}(t)+C\int_{S^{2n+1}}|\alpha
f-R_\theta|^{2k}|\nabla_\theta(\alpha
f-R_\theta)|^2_\theta dV_{\theta}\\
&=(2k+2)\int_{S^{2n+1}} \alpha f|\alpha f-R_\theta|^{2k+2}dV_{\theta}
\\&\hspace{4mm}+(2k+2)\alpha'\int_{S^{2n+1}}f|\alpha f-R_\theta|^{2k}(\alpha f-R_\theta)dV_{\theta}\\
&\hspace{4mm}+(n-2k-1)\int_{S^{2n+1}}|\alpha f-R_\theta|^{2k+2}(\alpha f-R_\theta)dV_{\theta}\\
&\leq CF_{2k+2}(t)
+CF_2(t)-\gamma(n-2k-1)F_{2k+2}(t)
\end{split}
\end{equation}
where we have used (\ref{3.17}), Lemma \ref{lem2.7}, and the fact that
\begin{equation*}
|\alpha f-R_\theta|^{2k+1}\leq \left\{
                                 \begin{array}{ll}
                                   |\alpha f-R_\theta|^{2k+2}, & \hbox{if $|\alpha f-R_\theta|\geq 1$;} \\
                                   |\alpha f-R_\theta|^{2}, & \hbox{if  $|\alpha f-R_\theta|<1$.}
                                 \end{array}
                               \right.
\end{equation*}
By  (\ref{2.10}) and (\ref{3.10})  for $k+1$,
 integrating (\ref{3.23}) from $0$ to $\infty$ yields
\begin{equation}\label{3.24}
\int_0^\infty\int_{S^{2n+1}}|\alpha
f-R_\theta|^{2k}|\nabla_\theta(\alpha
f-R_\theta)|^2_\theta dV_{\theta}dt\leq C\hspace{2mm}\mbox{ and }\hspace{2mm}F_{2k+2}(t)\leq C.
\end{equation}
This implies that (\ref{3.11}) and (\ref{3.12}) are true for $k+1$.
By (\ref{3.23}) again, we have
\begin{equation*}
\begin{split}
&(n-2k-1)\int_{S^{2n+1}}|\alpha f-R_\theta|^{2k+2}(\alpha
f-R_\theta)dV_{\theta}\\
=&\,\frac{d}{dt}F_{2k+2}(t)+C\int_{S^{2n+1}}|\alpha
f-R_\theta|^{2k}|\nabla_\theta(\alpha
f-R_\theta)|^2_\theta dV_{\theta}\\
&-(2k+2)\int_{S^{2n+1}} \alpha f|\alpha
f-R_\theta|^{2k+2}dV_{\theta}-(2k+2)\alpha'\int_{S^{2n+1}}f|\alpha f-R_\theta|^{2k}(\alpha
f-R_\theta)dV_{\theta}\\
\leq&\,\frac{d}{dt}F_{2k+2}(t)+C\int_{S^{2n+1}}|\alpha
f-R_\theta|^{2k}|\nabla_\theta(\alpha f-R_\theta)|^2_\theta
dV_{\theta}+CF_{2k+2}(t)+CF_2(t).
\end{split}
\end{equation*}
Integrating
it from 0 to $\infty$, we obtain by (\ref{2.10}), (\ref{3.24}) and by
(\ref{3.10}) for $k+1$
$$\left|\int_0^\infty\int_{S^{2n+1}}(\alpha
f-R_\theta)^{2k+3}dV_{\theta}dt\right|<\infty,$$ which implies that
(\ref{3.13}) is true for $k+1$. Thus, the iteration procedure can be
continued.

\textit{Case (c).} If $k=\displaystyle\frac{n}{2}$, we consider
(\ref{2.13}) with $p=2k+\displaystyle\frac{1}{3}$. It yields
\begin{equation}\label{3.26}
\begin{split}
&\frac{d}{dt}F_{2k+\frac{1}{3}}(t)+C\int_{S^{2n+1}}|\alpha
f-R_\theta|^{2k-\frac{5}{3}}|\nabla_\theta(\alpha
f-R_\theta)|^2_\theta dV_{\theta}\\
\leq&\,\frac{6k+1}{3}\int_{S^{2n+1}} \alpha f|\alpha
f-R_\theta|^{\frac{6k+1}{3}}dV_{\theta}
+\frac{6k+1}{3}\alpha'\int_{S^{2n+1}}f|\alpha f-R_\theta|^{\frac{6k-5}{3}}(\alpha f-R_\theta)dV_{\theta}\\
&+\frac{2}{3}\int_{S^{2n+1}}|\alpha
f-R_\theta|^{\frac{6k+1}{3}}(\alpha f-R_\theta)dV_{\theta}.
\end{split}
\end{equation}
By Lemma \ref{lem2.7}, we have
\begin{equation}\label{3.26.5}
\int_{S^{2n+1}}|\alpha f-R_\theta|^{\frac{6k+1}{3}}(\alpha
f-R_\theta)dV_{\theta}\leq-\gamma F_{\frac{6k+1}{3}}(t).
\end{equation}
 By Lemma
\ref{lem2.4}, we have
$$\int_{S^{2n+1}} \alpha f|\alpha
f-R_\theta|^{\frac{6k+1}{3}}dV_{\theta}\leq C
F_{\frac{6k+1}{3}}(t).$$
By H\"{o}lder's inequality and Young's
inequality, we get
\begin{equation}\label{3.27}
\begin{split}
F_{\frac{6k+1}{3}}(t)&\leq F_{k(\frac{2n+2}{n})}(t)^{\frac{1}{3}}
F_{2k}(t)^{\frac{2}{3}}\\
&\leq\frac{n+1}{3n}F_{k(\frac{2n+2}{n})}(t)^{\frac{n}{n+1}}
+\frac{2n-1}{3n}F_{2k}(t)^{\frac{2n}{2n-1}}\\
&\leq CF_{k(\frac{2n+2}{n})}(t)^{\frac{n}{n+1}}+CF_{2k}(t)
\end{split}
\end{equation}
where we have used (\ref{3.11}) and the fact that
$k=\displaystyle\frac{n}{2}$. On the other hand, by (\ref{3.3}) and Lemma \ref{lem3.1}, we can estimate
\begin{equation*}
\begin{split}
\left|\alpha'\int_{S^{2n+1}}f|\alpha f-R_\theta|^{\frac{6k-5}{3}}(\alpha
f-R_\theta)dV_{\theta}\right|
&\leq C|\alpha'|F_{\frac{6k-2}{3}}(t)\\
&\leq C(F_2(t)+F_2(t)^{\frac{1}{2}})F_{\frac{6k-2}{3}}(t)\\
&\leq C(F_2(t)+F_2(t)^{\frac{1}{2}})(F_{2k}(t)+F_1(t))\\
&\leq C(F_2(t)+F_2(t)^{\frac{1}{2}})(F_{2k}(t)+F_2(t)^{\frac{1}{2}})\\
&\leq CF_{2k}(t)+CF_2(t),
\end{split}
\end{equation*}
where we have used the inequality
\begin{equation*}
|\alpha f-R_\theta|^{\frac{6k-2}{3}}\leq \left\{
                                 \begin{array}{ll}
                                   |\alpha f-R_\theta|^{2k}, & \hbox{if $|\alpha f-R_\theta|\geq 1$;} \\
                                   |\alpha f-R_\theta|, & \hbox{if  $|\alpha f-R_\theta|<1$.}
                                 \end{array}
                               \right.
\end{equation*}
 Combining  the above estimates, the right hand side of
(\ref{3.26}) can be bounded by
\begin{equation*}
\begin{split}
&\frac{d}{dt}F_{2k+\frac{1}{3}}(t)+C\int_{S^{2n+1}}|\alpha
f-R_\theta|^{2k-\frac{5}{3}}|\nabla_\theta(\alpha
f-R_\theta)|^2_\theta dV_{\theta}\\
&\leq CF_{k(\frac{2n+2}{n})}(t)^{\frac{n}{n+1}}+CF_{2k}(t)+CF_2(t).
\end{split}
\end{equation*}
Integrating it from 0 to $\infty$, we obtain
\begin{equation}\label{3.28}
\int_0^\infty\int_{S^{2n+1}}|\alpha
f-R_\theta|^{2k-\frac{5}{3}}|\nabla_\theta(\alpha
f-R_\theta)|^2_\theta dV_{\theta}dt\leq C\mbox{ and
}F_{2k+\frac{1}{3}}(t)\leq C
\end{equation}
by (\ref{2.10}), (\ref{3.10}) and  (\ref{3.14}). On the other hand, integrating (\ref{3.27}) from 0 to $\infty$, we obtain
\begin{equation}\label{3.29}
\int_0^\infty\int_{S^{2n+1}}|\alpha f-R_\theta|^{2k+\frac{1}{3}}
dV_{\theta}dt\leq C
\end{equation}
by (\ref{3.10}) and (\ref{3.14}).
 Using (\ref{3.28}), (\ref{3.29})
and the sharp Folland-Stein inequality again, we obtain
$$\int_0^\infty\left(\int_{S^{2n+1}}|\alpha f-R_\theta|^{(k+\frac{1}{6})(2+\frac{2}{n})}
dV_{\theta}\right)^{\frac{n}{n+1}}dt<\infty.$$ Hence we can choose
$p_0=\displaystyle(k+\frac{1}{6})(2+\frac{2}{n})=n+1+\frac{n+1}{3n}$
since $k=\displaystyle\frac{n}{2}$ and
$v_0=\displaystyle\frac{n}{n+1}$ in (\ref{3.9}), and the iteration
terminates.

In sum, after finitely many steps, we obtain (\ref{3.9}) with some
suitable $p_0>n+1$ and $v_0\in (0,1]$.

\textit{Step 2.}  We will show that for any $p>n+1$, there holds
$\displaystyle\lim_{t\rightarrow\infty}F_p(t)=0$. Set
$p_k=p_0\displaystyle(\frac{n+1}{n})^k$, $k\in\mathbb{N}$, where
$p_0, v_0$ are given in Step 1. Next we argue by induction on $k$.

Assume for $p=p_k$ and some $v_k\in(0,1]$, there holds
\begin{equation}\label{3.30}
\int_0^\infty F_{p_k}(t)^{v_k}dt<\infty.
\end{equation}
We are going to show that
$\displaystyle\lim_{t\rightarrow\infty}F_{p_k}(t)=0$ and to
establish (\ref{3.30}) with $p_{k+1}$ and
$v_k=\displaystyle\frac{n}{n+1}$. Using (\ref{2.14}) for
$p=p_k>n+1$, we have
\begin{equation}\label{3.31}
\frac{d}{dt}F_{p_k}(t)+F_{p_{k+1}}(t)^{v_{k+1}}\leq
CF_{p_k}(t)^{v_k}\left(F_{p_k}(t)^{1-v_k}+F_{p_k}(t)^{\frac{p_k-n}{p_k-n-1}-v_k}\right).
\end{equation}
From (\ref{3.30}), there exists a sequence $t_j$ with
$t_j\rightarrow\infty$ as $j\rightarrow\infty$ such that
\begin{equation}\label{3.32}
F_{p_k}(t_j)\rightarrow 0\hspace{2mm}\mbox{ as }j\rightarrow\infty.
\end{equation}
For brevity, let $F(t)=F_{p_k}(t)$ and $v=v_k$ with
$\displaystyle\int_0^\infty F(t)^vdt<\infty$ by (\ref{3.30}). Then
(\ref{3.31}) implies that
\begin{equation}\label{3.33}
\frac{d}{dt}F(t)\leq
CF(t)^{v}\left(F(t)^{\beta_1}+F(t)^{\beta_2}\right)
\end{equation}
where $0\leq \beta_1=1-v_k<1$ and
$\beta_2=\displaystyle\frac{p_k-n}{p_k-n-1}-v_k=\beta_1+\frac{1}{p_k-n-1}>\beta_1\geq
0$. Define
$$H(t)=\int_0^{F(t)}\frac{ds}{s^{\beta_1}+s^{\beta_2}}.$$
By (\ref{3.32}), we have
\begin{equation}\label{3.34}
H(t_j)\leq
\int_0^{F(t_j)}\frac{ds}{s^{\beta_1}}=\frac{1}{1-\beta_1}F(t_j)^{1-\beta_1}\rightarrow
0 \hspace{2mm}\mbox{ as }j\rightarrow\infty.
\end{equation}
 By (\ref{3.33}), we
also have
$$\frac{d}{dt}H(t)=\frac{1}{F(t)^{\beta_1}+F(t)^{\beta_2}}\cdot\frac{d}{dt}F(t)\leq CF(t)^v.$$
Integrating it from $t_j$ to $t$ yields
$$H(t)\leq H(t_j)+C\int_{t_j}^\infty F(t)^vdt,$$
where the right hand side tends to 0 as $j\rightarrow\infty$ by
(\ref{3.30}) and (\ref{3.34}). This implies that
\begin{equation}\label{3.35}
\lim_{t\rightarrow\infty}H(t)=0.
\end{equation}
We claim that there exists a positive constant $C_0$ such that
$F(t)\leq C_0$ for all $t$. Otherwise, there would exist a sequence $t_k$
with $t_k\rightarrow\infty$ as $k\rightarrow\infty$, such that
$F(t_k)>1$ for all $k\in\mathbb{N}$, which would imply
$$H(t_k)=\int_0^{F(t_k)}\frac{ds}{s^{\beta_1}+s^{\beta_2}}\geq\int_0^{1}\frac{ds}{s^{\beta_1}+s^{\beta_2}}>0,
\mbox{ for all }k\in\mathbb{N},$$ which contradicts (\ref{3.35}).
Hence, by definition of $H(t)$, we have
$$H(t)\geq\frac{F(t)}{F(t)^{\beta_1}+F(t)^{\beta_2}}\geq \frac{F(t)}{C_0^{\beta_1}+C_0^{\beta_2}}=
CF(t).$$ Combining this with (\ref{3.35}), we have
\begin{equation}\label{3.36}
\lim_{t\rightarrow\infty}F(t)=0.
\end{equation}
By (\ref{3.31}) and (\ref{3.36}), we have $\displaystyle\frac{d}{dt}F(t)+F_{p_{k+1}}(t)^{v_{k+1}}\leq
CF(t)^{v}$. Integrating it from $0$ to $\infty$ and using (\ref{3.30}) and (\ref{3.36}), we obtain
$$\int_0^\infty F_{p_{k+1}}(t)^{v_{k+1}}dt<\infty.$$
Thus, the induction step is complete.
\end{proof}

\begin{lem}\label{lem3.3}
There holds $G_2(t)\rightarrow 0$ as $t\rightarrow\infty$.
\end{lem}
\begin{proof}
By (\ref{3.6}), there exists a sequence $t_j$ with
$t_j\rightarrow\infty$ as $j\rightarrow\infty$ such that
\begin{equation}\label{3.37}
G_2(t_j)\rightarrow 0\mbox{ as
}j\rightarrow\infty.
\end{equation}
Since $\theta=u^{\frac{2}{n}}\theta_0$, we have (see (2.4) in \cite{Ho2} for example)
$$\int_{S^{2n+1}}\langle\nabla_\theta\,\psi_1,\nabla_\theta\,\psi_2\rangle_\theta dV_{\theta}
=\int_{S^{2n+1}}u^2\langle\nabla_{\theta_0}\psi_1,\nabla_{\theta_0}\psi_2\rangle_{\theta_0} dV_{\theta_0},$$
for any function $\psi_1 $ and $\psi_2$,
which together with (\ref{2.3}) and Lemma \ref{lem2.5} implies that
\begin{equation*}
\begin{split}
\frac{d}{dt}G_2(t)
&=2\int_{S^{2n+1}}|\nabla_{\theta_0}(\alpha
f-R_\theta)|^2_{\theta_0} u\frac{\partial u}{\partial t}dV_{\theta_0}+
2\int_{S^{2n+1}}\langle\nabla_\theta(\alpha
f-R_\theta),\nabla_\theta(\alpha'
f-\frac{\partial R_\theta}{\partial t})\rangle_\theta dV_{\theta}\\
&=n\int_{S^{2n+1}}u^2|\nabla_{\theta_0}(\alpha
f-R_\theta)|^2_{\theta_0} (\alpha
f-R_\theta)dV_{\theta_0}+2
\alpha'\int_{S^{2n+1}}\langle\nabla_\theta(\alpha
f-R_\theta),\nabla_\theta f\rangle_\theta dV_{\theta}\\
&\hspace{4mm}-2\int_{S^{2n+1}}\Delta_\theta(\alpha
f-R_\theta)\Big[(n+1)\Delta_\theta (\alpha
f-R_\theta)+(\alpha f-R_\theta)R_\theta\Big]dV_{\theta}\\
&:=I+II+III.
\end{split}
\end{equation*}
Note that by Lemma \ref{lem2.7} we have
$$I\leq-n\gamma\int_{S^{2n+1}}u^2|\nabla_{\theta_0}(\alpha
f-R_\theta)|^2_{\theta_0} dV_{\theta_0}=-n\gamma\,G_2(t).$$
Note also that $|\alpha'|\leq CF_2(t)+CF_2(t)^{\frac{1}{2}}\leq CF_2(t)^{\frac{1}{2}}$
by (\ref{3.3}) and the fact that $F_2(t)\leq C$. This together
with H\"{o}lder's inequality and Young's inequality
implies that
\begin{equation*}
\begin{split}
|II|
&\leq 2\frac{|\alpha'|}{2}\left(\int_{S^{2n+1}}|\nabla_\theta(\alpha
f-R_\theta|_\theta^2 dV_{\theta}\right)^{\frac{1}{2}}\left(\int_{S^{2n+1}}|\nabla_\theta f|^2_\theta
dV_{\theta}\right)^{\frac{1}{2}}\\
&\leq CF_2(t)^{\frac{1}{2}}G_2(t)^{\frac{1}{2}}\leq CF_2(t)+CG_2(t),
\end{split}
\end{equation*}
since
$$\int_{S^{2n+1}}|\nabla_\theta f|^2_\theta dV_{\theta}
=\int_{S^{2n+1}}|\nabla_{\theta_0} f|^2_{\theta_0}u^2 dV_{\theta_0}
\leq C\left(\int_{S^{2n+1}}u^{\frac{2n+2}{n}}dV_{\theta_0}\right)^{\frac{n}{n+1}}\leq C$$
by Proposition \ref{prop2.1}.
On the other hand, for any $\epsilon>0$, by Lemma \ref{lem2.4} and Young's inequality, we have
\begin{equation*}
\begin{split}
&\frac{III}{2}+(n+1)\int_{S^{2n+1}}|\Delta_\theta(\alpha
f-R_\theta)|^2dV_{\theta}
\\
&=-\int_{S^{2n+1}}(\alpha f-R_\theta)R_\theta\Delta_\theta(\alpha
f-R_\theta)dV_{\theta}\\
&=\int_{S^{2n+1}}(\alpha f-R_\theta)^2\Delta_\theta(\alpha
f-R_\theta)dV_{\theta}-\alpha\int_{S^{2n+1}}f(\alpha f-R_\theta)\Delta_\theta(\alpha
f-R_\theta)dV_{\theta}\\
&\leq \epsilon\int_{S^{2n+1}}|\Delta_\theta(\alpha
f-R_\theta)|^2dV_{\theta}+C\int_{S^{2n+1}}(\alpha f-R_\theta)^4dV_{\theta}
+C\int_{S^{2n+1}}(\alpha f-R_\theta)^2dV_{\theta}.
\end{split}
\end{equation*}
Hence, if we choose $\epsilon=n+1$ and by H\"{o}lder's inequality, we get
\begin{equation*}
\begin{split}
III\leq C\big[F_4(t)+F_2(t)\big]
&\leq C\big[F_{2n+2}(t)^{\frac{1}{n+1}}F_{\frac{2n+2}{n}}(t)^{\frac{n}{n+1}}+F_2(t)\big]\\
&\leq
C\big[F_{\frac{2n+2}{n}}(t)^{\frac{n}{n+1}}+F_2(t)\big]
\end{split}
\end{equation*}
since $F_{2n+2}(t)\leq C$ by Lemma \ref{lem3.2}. Combining
all the estimates above, we obtain
$$\frac{d}{dt}G_2(t)\leq C\big[G_2(t)+F_2(t)+F_{\frac{2n+2}{n}}(t)^{\frac{n}{n+1}}\big].$$
Integrating it from $t_j$ to $t$, we get
$$G_2(t)\leq G_2(t_j)+C\int_{t_j}^t\big[G_2(t)+F_2(t)
+F_{\frac{2n+2}{n}}(t)^{\frac{n}{n+1}}\big]dt.$$
Note that the right hand side tends to 0 when $j\rightarrow\infty$,
thanks to (\ref{2.10}), (\ref{3.6}),
(\ref{3.14}) and (\ref{3.37}). This proves the assertion.
\end{proof}

\section{Blow-up analysis}\label{section4}

The Riemannian version of the following theorem was proved by
Schwetlick and Struwe (see Theorem A.1 in \cite{Schwetlick&Struwe}).

\begin{theorem}\label{thm4.1}
Let $u\in C^\infty(M)$ be a solution of
\begin{equation}\label{4.1}
-(2+\frac{2}{n})\Delta_{\theta_0}u+R_{\theta_0}u=Pu\hspace{2mm}\mbox{
on }S^{2n+1}.
\end{equation}
\emph{(i)} For any $\sigma< Y(S^{2n+1},\theta_0)$, there exists
constants $q_0>2+\displaystyle\frac{2}{n}$ and $r_0>0$ such that
whenever for some $r<r_0$ and some $x_0\in M$ there holds
$\|P\|_{L^{n+1}(B_{2r}(x_0))}\leq\sigma$, then
$$\|u\|_{L^{q_0}(B_r(x_0))}<C\|u\|_{L^{2+\frac{2}{n}}(B_{2r}(x_0))}$$
for some constant $C$ independent of $x_0$.\\
\emph{(ii)} For any $q>2+\displaystyle\frac{2}{n}$, and any $r>0$,
there exists a constant $C=C(q,r)$ such that
$$\|u\|_{L^{q}(B_{3r}(x_0))}<C\|u\|_{L^{2+\frac{2}{n}}(B_{4r}(x_0))}$$
whenever there holds
$\|P\|_{L^{n+1}(B_{4r}(x_0))}<\displaystyle\frac{2n+2}{nq}
Y(S^{2n+1},\theta_0)$.
\end{theorem}
\begin{proof}
(i) For suitable $p\geq 1$ and $\eta\in C^1_0(B_{2r}(x_0))$ with
\begin{equation}\label{4.2}
\eta=1\mbox{ in }B_{r}(x_0),\hspace{3mm}\eta=0\mbox{ in
}S^{2n+1}\setminus B_{2r}(x_0),\hspace{3mm}
|\nabla_{\theta_0}\eta|_{\theta_0}<C/r,
\end{equation} we let
$v=u^{2p-1}\eta^2$. Multiplying (\ref{4.1}) by $v$ and integrating
it over $S^{2n+1}$, we have
\begin{equation}\label{4.3}
\int_{S^{2n+1}}\left((2+\frac{2}{n})\langle\nabla_{\theta_0}u,\nabla_{\theta_0}v\rangle_{\theta_0}
+R_{\theta_0}uv\right)dV_{\theta_0}=
\int_{S^{2n+1}}Puv\:dV_{\theta_0}.
\end{equation}
Also, let $w=u^p\eta$. Then $w^2=u^{2p}\eta^2=uv$ and
\begin{equation}\label{4.4a}
\begin{split}
|\nabla_{\theta_0}w|_{\theta_0}^2&=p^2u^{2p-2}|\nabla_{\theta_0}u|_{\theta_0}^2\eta^2
+2p\,u^{2p-1}\eta\langle\nabla_{\theta_0}u,\nabla_{\theta_0}\eta\rangle_{\theta_0}
+u^{2p}|\nabla_{\theta_0}\eta|_{\theta_0}^2\\
&\leq
(p^2+\frac{p-1}{2})u^{2p-2}|\nabla_{\theta_0}u|_{\theta_0}^2\eta^2
+(p+1)u^{2p-1}\eta\langle\nabla_{\theta_0}u,\nabla_{\theta_0}\eta\rangle_{\theta_0}
+(1+\frac{p-1}{2})u^{2p}|\nabla_{\theta_0}\eta|_{\theta_0}^2\\
&=\frac{1+p}{2}\Big(\langle\nabla_{\theta_0}u,\nabla_{\theta_0}v\rangle_{\theta_0}
+u^{2p}|\nabla_{\theta_0}\eta|_{\theta_0}^2\Big).
\end{split}
\end{equation}
Thus,
\begin{equation}\label{4.4}
\begin{split}
&\frac{2}{1+p}\int_{S^{2n+1}}\left((2+\frac{2}{n})|\nabla_{\theta_0}w|_{\theta_0}^2
+R_{\theta_0}w^2\right)dV_{\theta_0}-(2+\frac{2}{n})
\int_{S^{2n+1}}u^{2p}|\nabla_{\theta_0}\eta|_{\theta_0}^2dV_{\theta_0}\\
&\leq\int_{S^{2n+1}}(2+\frac{2}{n})\langle\nabla_{\theta_0}u,\nabla_{\theta_0}v\rangle_{\theta_0}dV_{\theta_0}
+\frac{2}{1+p}\int_{S^{2n+1}}R_{\theta_0}w^2\:dV_{\theta_0}\\
&\leq\int_{S^{2n+1}}\left((2+\frac{2}{n})\langle\nabla_{\theta_0}u,\nabla_{\theta_0}v\rangle_{\theta_0}dV_{\theta_0}
+R_{\theta_0}w^2\right)dV_{\theta_0}\\
&=\int_{B_{2r}(x_0)}Pw^2\:dV_{\theta_0}\leq
\|P\|_{L^{n+1}(B_{2r}(x_0))}\|w\|_{L^{2+\frac{2}{n}}(B_{2r}(x_0))}^2\leq
\sigma\|w\|_{L^{2+\frac{2}{n}}(B_{2r}(x_0))}^2,
\end{split}
\end{equation}
where the first inequality follows from (\ref{4.4a}), and the first equality follows from  (\ref{4.3}) and the fact that
support of $w$ lies in $B_{2r}(x_0)$, and the second last
inequality follows from H\"{o}lder's inequality, and the last
inequality follows from the assumption that
$\|P\|_{L^{n+1}(B_{2r}(x_0))}\leq\sigma$.

On the other hand, by Lemma \ref{lem2.3} and the fact that support
of $w$ lies in $B_{2r}(x_0)$, we have
$$\int_{S^{2n+1}}\left((2+\frac{2}{n})|\nabla_{\theta_0}w|_{\theta_0}^2
+R_{\theta_0}w^2\right)dV_{\theta_0}\geq
Y(S^{2n+1},\theta_0)\|w\|_{L^{2+\frac{2}{n}}(B_{2r}(x_0))}^2.$$
Combining this with (\ref{4.4}), we have
\begin{equation}\label{4.5}
\left(\frac{2}{1+p}Y(S^{2n+1},\theta_0)-\sigma\right)\|w\|_{L^{2+\frac{2}{n}}(B_{2r}(x_0))}^2\leq
(2+\frac{2}{n})
\int_{S^{2n+1}}u^{2p}|\nabla_{\theta_0}\eta|_{\theta_0}^2dV_{\theta_0}.
\end{equation}
Let $\sigma=a Y(S^{2n+1},\theta_0)$, then $a<1$ by assumption. For
$1<p<\min\displaystyle\Big\{\frac{3-a}{1+a},2+\frac{2}{n}\Big\}$ and
$r<r_0=r_0(a)$, we have
\begin{equation*}
\begin{split}
\frac{(1-a)Y(S^{2n+1},\theta_0)}{4(2+\frac{2}{n})}\|u\|^{2p}_{L^{(2+\frac{2}{n})p}(B_{r}(x_0))}
&=\frac{(1-a)Y(S^{2n+1},\theta_0)}{4(2+\frac{2}{n})}
\left(\int_{B_{r}(x_0)}w^{\frac{2n+2}{n}}dV_{\theta_0}\right)^{\frac{n}{n+1}}\\
&\leq
\frac{(1-a)Y(S^{2n+1},\theta_0)}{4(2+\frac{2}{n})}\|w\|_{L^{\frac{2n+2}{n}}(B_{2r}(x_0))}^2\\
&\leq
\int_{S^{2n+1}}u^{2p}|\nabla_{\theta_0}\eta|_{\theta_0}^2dV_{\theta_0}\leq
Cr^{-2}\|u\|_{L^{\frac{2n+2}{n}}(B_{2r}(x_0))}^{2p}
\end{split}
\end{equation*}
where the second last inequality follows from (\ref{4.5}), and the
last inequality follows from (\ref{4.2}) and H\"{o}lder's
inequality. This proves the assertion.\\
(ii) For $1<p\leq\displaystyle\frac{q}{2+\frac{2}{n}}$ and $\eta\in C^1(B_{4r}(x_0))$, we
define the test function
 $v=u^{2p-1}\eta^2$. Then by the same argument as (i) we can derive
\begin{equation}\label{4.6}
\frac{Y(S^{2n+1},\theta_0)}{2+\frac{2}{n}}\left(\frac{2}{1+p}
-\frac{2+\frac{2}{n}}{q}\right)\|w\|_{L^{2+\frac{2}{n}}(S^{2n+1})}^2\leq
\int_{S^{2n+1}}u^{2p}|\nabla_{\theta_0}\eta|_{\theta_0}^2dV_{\theta_0}
\end{equation}
where $w=u^p\eta$, since $\|P\|_{L^{n+1}(B_{4r}(x_0))}<\displaystyle\frac{2n+2}{nq}
Y(S^{2n+1},\theta_0)$ by assumption. Choose $p_1=\displaystyle\Big(1,\frac{n+1}{n}\Big]$ maximal such that
$q=2p_1\displaystyle\Big(\frac{n+1}{n}\Big)^m$ for some $m\in\mathbb{N}$.
For $i=1,...,m+1$, let $p_i=p_1\displaystyle\Big(\frac{n+1}{n}\Big)^{i-1}$, $r_i=r(3+2^{1-i})$, and
$$\eta_i=1\mbox{ in }B_{r_{i+1}}(x_0),\hspace{3mm}\eta_i=0\mbox{ in
}S^{2n+1}\setminus B_{r_i}(x_0),\hspace{3mm}
|\nabla_{\theta_0}\eta_i|_{\theta_0}<C2^i/r,$$
we can apply (\ref{4.6}) to get
\begin{equation*}
\|u\|_{L^{2p_{i+1}}(B_{r_{i+1}}(x_0))}
\leq \left(\frac{C2^{2i}(q-\frac{2n+2}{n})}{r^2}\right)^{\frac{1}{2p_i}}\|u\|_{L^{2p_i}(B_{r_i}(x_0))},
\end{equation*}
which can be iterated to obtain the result.
\end{proof}

Now we can apply the previous theorem to prove
the following concentration-compactness result,
which is the CR version of the results of Schwetlick and Struwe
(see Theorem 3.1 in \cite{Schwetlick&Struwe}).

\begin{theorem}\label{thm4.2}
Let $\theta_k=u_k^{\frac{2}{n}}\theta_0$,  where
$0<u_k\in C^\infty(S^{2n+1})$ and $k\in\mathbb{Z}^+$, be a family of contact form with
\emph{Vol}$(S^{2n+1},\theta_k)=$\emph{Vol}$(S^{2n+1},\theta_0)$. If
there exists  $p_1>n+1$ and a constant $C_0$ such that
\begin{equation}\label{4.7}
\overline{R}_{\theta_k}=\frac{\int_{S^{2n+1}}R_{\theta_k}dV_{\theta_k}}{\int_{S^{2n+1}}dV_{\theta_k}}
\leq C_0\mbox{ and }
\int_{S^{2n+1}}|R_{\theta_k}-\overline{R}_{\theta_k}|^{p_1}dV_{\theta_k}\leq C_0
\end{equation}
for all $k$, then either\\
\emph{(i)} the sequence $\{u_k\}$ is uniformly bounded in $S^p_2(S^{2n+1},\theta_0)$ for all $p<p_1$, or \\
\emph{(ii)} there exists a subsequence $\{u_k\}$ (relabeled) and
finitely many points $x_1,..., x_L\in S^{2n+1}$ such that for any
$r>0$ and any $i\in\{1,..., L\}$ there holds
\begin{equation}\label{4.8}
\liminf_{k\rightarrow\infty}\left(\int_{B_r(x_i)}|R_{\theta_k}|^{n+1}dV_{\theta_k}\right)^{\frac{1}{n+1}}\geq
Y(S^{2n+1},\theta_0).
\end{equation}
Moreover, the sequence $\{u_k\}$ is bounded in $S_2^p(S^{2n+1},\theta_0)$
on any compact subset of $(S^{2n+1}\setminus\{x_1,..., x_L\},\theta_0)$.
\end{theorem}
\begin{proof}
Fix a point $x_0\in S^{2n+1}$ and assume that for some $r>0$ there holds
\begin{equation}\label{4.9}
\sup_k\left(\int_{B_r(x_0)}|R_{\theta_k}|^{n+1}dV_{\theta_k}\right)^{\frac{1}{n+1}}
\leq\sigma< Y(S^{2n+1},\theta_0).
\end{equation}
For $\theta_k=u_k^{\frac{2}{n}}\theta_0$, we have
$$-(2+\frac{2}{n})\Delta_{\theta_0}u_k+R_{\theta_0}u_k=R_{\theta_k}u_k^{1+\frac{2}{n}}=P_ku_k,$$
where $P_k=R_{\theta_k}u_k^{\frac{2}{n}}\in L^{n+1}(B_{2r}(x_0))$
because
$$\|P_k\|_{L^{n+1}(B_{2r}(x_0)}
=\left(\int_{B_{2r}(x_0)}|R_{\theta_k}|^{n+1}u_k^{2+\frac{2}{n}}dV_{\theta_0}\right)^{\frac{1}{n+1}}
\leq\sigma<Y(S^{2n+1},\theta_0)$$
by (\ref{4.9}). Clearly we can assume that $r\leq r_0(\sigma)$ as determined in Theorem \ref{thm4.1}(i).
Therefore, Theorem \ref{thm4.1}(i) and the
assumption that $\|u_k\|^{2+\frac{2}{n}}_{L^{2+\frac{2}{n}}(S^{2n+1})}=$Vol$(S^{2n+1},\theta_k)=$
Vol$(S^{2n+1},\theta_0)$
imply that $u_k$ is bounded in $L^{q_0}(B_r(x_0))$ for
some $q_0>2+\displaystyle\frac{2}{n}$. In particular,
\begin{equation}\label{4.10}
\mbox{Vol}(B_r(x_0),\theta_k)
\leq \|u_k\|_{L^{q_0}(B_r(x_0))}^{2+\frac{2}{n}}\mbox{Vol}(B_r(x_0),\theta_0)^{1-\frac{1}{q_0}}\rightarrow 0\hspace{2mm}\mbox{ as }r\rightarrow 0
\end{equation}
uniformly in $k\in\mathbb{Z}^+$. Then
\begin{equation*}
\begin{split}
&\|P_k\|_{L^{n+1}(B_{4r}(x_0))}\\
&=\left(\int_{B_{4r}(x_0)}|R_{\theta_k}|^{n+1}u_k^{2+\frac{2}{n}}dV_{\theta_0}\right)^{\frac{1}{n+1}}=
\left(\int_{B_{4r}(x_0)}|R_{\theta_k}|^{n+1}dV_{\theta_k}\right)^{\frac{1}{n+1}}\\
&\leq \left(\int_{B_{4r}(x_0)}|R_{\theta_k}-\overline{R}_{\theta_k}|^{n+1}dV_{\theta_k}\right)^{\frac{1}{n+1}}+
\overline{R}_{\theta_k}\left(\int_{B_{4r}(x_0)}dV_{\theta_k}\right)^{\frac{1}{n+1}}\\
&\leq\left(\int_{B_{4r}(x_0)}|R_{\theta_k}-\overline{R}_{\theta_k}|^{p_1}dV_{\theta_k}\right)^{\frac{1}{p_1}}
\left(\int_{B_{4r}(x_0)}dV_{\theta_k}\right)^{\frac{1}{n+1}-\frac{1}{p_1}}+
\overline{R}_{\theta_k}\left(\int_{B_{4r}(x_0)}dV_{\theta_k}\right)^{\frac{1}{n+1}}\\
&\leq C_0^{\frac{1}{p_1}} \mbox{Vol}(B_r(x_0),\theta_k)^{\frac{1}{n+1}-\frac{1}{p_1}}
+C_0\mbox{Vol}(B_r(x_0),\theta_k)^{\frac{1}{n+1}}\rightarrow 0\hspace{2mm}\mbox{ as }r\rightarrow 0,
\end{split}
\end{equation*}
where we have used H\"{o}lder's inequality, (\ref{4.7}) and (\ref{4.10}).
By replacing $r$ by a smaller radius,
we can achieve $\|P_k\|_{L^{n+1}(B_{4r}(x_0))}<\displaystyle\frac{2n+2}{nq}
Y(S^{2n+1},\theta_0)$. Therefore, we can apply Theorem \ref{thm4.1}(ii) and the assumption that
$\|u_k\|^{2+\frac{2}{n}}_{L^{2+\frac{2}{n}}(S^{2n+1})}=$Vol$(S^{2n+1},\theta_k)=$
Vol$(S^{2n+1},\theta_0)$
to conclude that $u_k$ is bounded in $L^{q}(B_{3r}(x_0))$ for any $q>2+\displaystyle\frac{2}{n}$,
 which implies that
\begin{equation}\label{4.11}
\begin{split}
-(2+\frac{2}{n})\Delta_{\theta_0}u_k&=R_{\theta_k}u_k^{1+\frac{2}{n}}-R_{\theta_0}u_k\\
&=(R_{\theta_k}-\overline{R}_{\theta_k})u_k^{1+\frac{2}{n}}
+\overline{R}_{\theta_k}u_k^{1+\frac{2}{n}}-R_{\theta_0}u_k\in L^{p}(B_{3r}(x_0))
\end{split}
\end{equation}
for all  $\displaystyle 2+\frac{2}{n}< p<p_1$. To see this, by
H\"{o}lder's inequality we have
\begin{equation*}
\begin{split}
\big\|(R_{\theta_k}-\overline{R}_{\theta_k})u_k^{1+\frac{2}{n}}\big\|^{p}_{L^{p}(B_{3r}(x_0))}
&=\int_{B_{3r}(x_0)}|R_{\theta_k}-\overline{R}_{\theta_k}|^{p}u_k^{(1+\frac{2}{n})p}dV_{\theta_0}\\
&\leq\left(\int_{B_{3r}(x_0)}|R_{\theta_k}-\overline{R}_{\theta_k}|^{p_1}u_k^{2+\frac{2}{n}}
dV_{\theta_0}\right)^{\frac{p}{p_1}}
\left(\int_{B_{3r}(x_0)}u_k^{q_0}dV_{\theta_0}\right)^{1-\frac{p}{p_1}}\\
&\leq\left(\int_{S^{2n+1}}|R_{\theta_k}-\overline{R}_{\theta_k}|^{p_1}
dV_{\theta_k}\right)^{\frac{p}{p_1}}\|u_k\|_{L^{q_0}(B_{3r}(x_0))}^{q_0(\frac{p_1-p}{p_1})}.
\end{split}
\end{equation*}
which is bounded because of (\ref{4.7}) and the fact $u_k$ is bounded in
$L^{q_0}(B_{3r}(x_0))$ where
\begin{equation*}
\begin{split}
q_0&=\big[(1+\frac{2}{n})p-(2+\frac{2}{n})\frac{p}{p_1}\big]\cdot\frac{p_1}{p_1-p}
\\
&\geq(1+\frac{2}{n})p-(2+\frac{2}{n})\frac{p}{p_1}
>(2+\frac{2}{n})\big[(1+\frac{2}{n})-(2+\frac{2}{n})\frac{1}{p_1}\big]\geq2+\frac{2}{n},
\end{split}
\end{equation*}
where we have used the assumption that $p_1>n+1$. On the other hand,
\begin{equation*}
\begin{split}
\big\|\overline{R}_{\theta_k}u_k^{1+\frac{2}{n}}\big\|_{L^{p}(B_{3r}(x_0))}
&\leq
\overline{R}_{\theta_k}\|u_k\|^{1+\frac{2}{n}}_{L^{p(1+\frac{2}{n})}(B_{3r}(x_0))}\mbox{ and }\\
\big\|R_{\theta_0}u_k\big\|_{L^{p}(B_{3r}(x_0))}&\leq
|R_{\theta_0}|\|u_k\|_{L^{p}(B_{3r}(x_0))}
\end{split}
\end{equation*}
are bounded by  (\ref{4.7}) and the fact $u_k$ is bounded in
$L^{q}(B_{3r}(x_0))$ for all $q>2+\displaystyle\frac{2}{n}$. Combining
all these, we prove (\ref{4.11}). By (\ref{4.11}) and the fact that
$u_k$ is bounded in $L^{p_1}(B_{3r}(x_0))$, we have
\begin{equation}\label{4.12}
\|u_k\|_{S^{p}_2(B_{3r}(x_0))}\leq
C\|\Delta_{\theta_0}u_k\|_{L^{p}(B_{3r}(x_0))}+C\|u_k\|_{L^{p}(B_{3r}(x_0))}\leq
C(r),
\end{equation}
where we have used Theorem 3.16 and 3.17 in \cite{Dragomir}. See
also \cite{Folland&Stein}.

Now assume that (\ref{4.9}) is satisfied for every $x\in S^{2n+1}$
and some radius $r=r(x)>0$. Since $S^{2n+1}$ is compact, the cover
$\big(B_{r(x)}(x)\big)_{x\in S^{2n+1}}$ of $S^{2n+1}$ admits a
finite subcover $B_{r_i}(x_i)$, where $r_i=r(x_i)$, $1\leq i\leq I$.
From (\ref{4.12}), we then obtain the desired uniform bound
$$\|u_k\|_{S^{p}_2(S^{2n+1})}\leq I\max_{1\leq i\leq I}C(r_i).$$

If (\ref{4.9}) does not hold for every $x$ with some $r=r(x)>0$, we
iteratively determine points $x_l, l\in\mathbb{N}$, and a
subsequence $\{u_k\}$ (relabeled) such that  condition (\ref{4.8}) is valid
for any $r>0$. This iteration terminates after
finitely many steps. Indeed, given $x_1,..., x_L$, choose
$0<r<\displaystyle\min_{i\neq j} d(x_i,x_j)/2$. Then (\ref{4.8})
yields the bound
\begin{equation*}
\begin{split}
L\cdot Y(S^{2n+1},\theta_0)^{n+1}&\leq
\sum_{i=1}^L\liminf_{k\rightarrow\infty}\int_{B_r(x_i)}|R_{\theta_k}|^{n+1}dV_{\theta_k}\\
&=
\liminf_{k\rightarrow\infty}\int_{\cup_{i=1}^LB_r(x_i)}|R_{\theta_k}|^{n+1}dV_{\theta_k}\\
&\leq \sup_{k}\int_{S^{2n+1}}|R_{\theta_k}|^{n+1}dV_{\theta_k}\\
&\leq
C\sup_{k}\left(\int_{S^{2n+1}}|R_{\theta_k}-\overline{R}_{\theta_k}|^{n+1}dV_{\theta_k}
+\overline{R}^{n+1}_{\theta_k}\right)\\
&\leq
C\sup_{k}\left(\int_{S^{2n+1}}|R_{\theta_k}-\overline{R}_{\theta_k}|^{p_1}dV_{\theta_k}
+\overline{R}^{n+1}_{\theta_k}\right)<\infty
\end{split}
\end{equation*}
where we have used (\ref{4.7}), H\"{o}lder's inequality, and
Vol$(S^{2n+1},\theta_k)=$Vol$(S^{2n+1},\theta_0)$. By a covering
argument as above we then obtain that $\{u_k\}$ is bounded in
$S_2^p(S^{2n+1})$ on any compact subset of
$(S^{2n+1}\setminus\{x_1,..., x_L\},\theta_0)$.
\end{proof}

Now we can apply Theorem \ref{4.2} to the solution of the flow
equation (\ref{2.3}).

\begin{lem}\label{lem4.3}
For any time sequence $t_k$, we denote $u_k=u(t_k)$ where $u(t)$
is the solution to the flow equation $(\ref{2.3})$. Then either\\
\emph{(i)} the sequence $\{u_k\}$ is uniformly bounded in $S^p_2(S^{2n+1},\theta_0)$
for some $p>n+1$, or \\
\emph{(ii)} there exists a subsequence $\{u_k\}$ (relabeled) and
finitely many points $x_1,..., x_L\in S^{2n+1}$ such that for any
$r>0$ and any $i\in\{1,..., L\}$ there holds
\begin{equation}\label{4.13}
\liminf_{k\rightarrow\infty}\left(\int_{B_r(x_i)}|R_{\theta(t_k)}|^{n+1}
dV_{\theta(t_k)}\right)^{\frac{1}{n+1}}\geq
Y(S^{2n+1},\theta_0).
\end{equation}
Moreover, the sequence $\{u_k\}$ is bounded in $S_2^p(S^{2n+1},\theta_0)$ on
any compact subset of $(S^{2n+1}\setminus\{x_1,...,
x_L\},\theta_0)$.
\end{lem}
\begin{proof}
We are going to apply Theorem \ref{thm4.2}. Since the flow (\ref{2.3})
keeps the volume fixed by Proposition \ref{prop2.1}, we only need to
check (\ref{4.7}) for some $p_1>n+1$. First, note that
\begin{equation}\label{4.14}
\overline{R}_{\theta(t_k)}=\frac{\int_{S^{2n+1}}R_{\theta(t_k)}
dV_{\theta(t_k)}}{\int_{S^{2n+1}}dV_{\theta(t_k)}}=
\frac{\alpha(t_k)\int_{S^{2n+1}}fdV_{\theta(t_k)}}{\int_{S^{2n+1}}
dV_{\theta(t_k)}}\leq\alpha_2\max_{S^{2n+1}}f
\end{equation}
by (\ref{2.2}) and Lemma \ref{lem2.4}. On the other hand, by
Minkowski inequality, Proposition \ref{prop2.1} and Lemma
\ref{lem2.4}, we have
\begin{equation*}
\begin{split}
&\left(\int_{S^{2n+1}}|R_{\theta(t_k)}-\overline{R}_{\theta(t_k)}|^{p_1}
dV_{\theta(t_k)}\right)^{\frac{1}{p_1}}\\
&\leq F_{p_1}(t_k)^{\frac{1}{p_1}}
+\alpha(t_k)\left(\int_{S^{2n+1}}f^{p_1}dV_{\theta(t_k)}\right)^{\frac{1}{p_1}}
+\overline{R}_{\theta(t_k)}\mbox{Vol}(S^{2n+1},\theta(t_k))^{\frac{1}{p_1}}\\
&=F_{p_1}(t_k)^{\frac{1}{p_1}} +\big(\alpha_2\max_{S^{2n+1}}f
+\overline{R}_{\theta(t_k)}\big)\mbox{Vol}(S^{2n+1},\theta_0)^{\frac{1}{p_1}}
\end{split}
\end{equation*}
which is bounded by (\ref{4.14}) and Lemma \ref{lem3.1}. Therefore,
Lemma \ref{lem4.3} follows from Theorem \ref{thm4.2}.
\end{proof}

\textit{Remark.} One important thing we would like to mention is
that concentration in the sense
of (\ref{4.13}) implies concentration of volume. To see this, by H\"{o}lder's inequality and Lemma \ref{lem2.4}, for any $r>0$ and $p>n+1$ we can estimate
\begin{equation}\label{4.14a}
\begin{split}
&\left(\int_{B_r(x_i)}|R_{\theta(t_k)}|^{n+1}dV_{\theta(t_k)}\right)^{\frac{1}{n+1}}\\
&\leq \alpha_2\Big(\max_{S^{2n+1}}f\Big)\left(\int_{B_r(x_i)}dV_{\theta(t_k)}\right)^{\frac{1}{n+1}}
+\left(\int_{B_r(x_i)}|\alpha(t_k)f-R_{\theta(t_k)}|^{n+1}dV_{\theta(t_k)}\right)^{\frac{1}{n+1}}\\
&\leq \alpha_2\Big(\max_{S^{2n+1}}f\Big)\left(\int_{B_r(x_i)}dV_{\theta(t_k)}\right)^{\frac{1}{n+1}}
+\left(\int_{B_r(x_i)}|\alpha(t_k)f-R_{\theta(t_k)}|^{p}dV_{\theta(t_k)}\right)^{\frac{1}{p}}
\left(\int_{B_r(x_i)}dV_{\theta_k}\right)^{\frac{1}{n+1}-\frac{1}{p}}.
\end{split}
\end{equation}
It follows from (\ref{4.14a}) and Lemma \ref{lem3.2} that if
 concentration of volume does not occur, then  concentration in the sense
of (\ref{4.13}) does not occur.

Let $\delta_n=\max_{S^{2n+1}}f/\min_{S^{2n+1}}f.$ By assumption (\ref{sbc}) in Theorem \ref{thm1.1}, we
have $\delta_n<2^{\frac{1}{n}}.$ Then there exists $\epsilon_0>0$
such that
$$\displaystyle\frac{\delta_n^{\frac{n}{n+1}}}{2^{\frac{1}{n+1}}}=\frac{1-\epsilon_0}{1+\epsilon_0}.$$
In particular,
$(1+\epsilon_0)\delta_n^{\frac{n}{n+1}}<2^{\frac{1}{n+1}}$.
 Set
\begin{equation}\label{4.15}
\beta=(1+\epsilon_0)Y(S^{2n+1},\theta_0)\left(\min_{S^{2n+1}}f\right)^{-\frac{n}{n+1}}.
\end{equation}

The next lemma is to estimate the number of blow-up points.

\begin{lem}\label{lem4.4}
For any $0<u_0\in C^\infty(S^{2n+1})$ with
$\displaystyle\int_{S^{2n+1}}u_0^{2+\frac{2}{n}}dV_{\theta_0}=\mbox{\emph{Vol}}(S^{2n+1},\theta_0)$
and $E_f(u_0)\leq \beta$, let $u(t)$ be the solution of the flow
$(\ref{2.3})$ with the initial data $u_0$. If $\{u(t_k)\}$ is a
sequence with $t_k\rightarrow\infty$ as $k\rightarrow\infty$ and
\begin{equation}\label{4.16}
\max_{S^{2n+1}}f<2^{\frac{1}{n}}\min_{S^{2n+1}}f,
\end{equation}
 then $L=1$.
\end{lem}
\begin{proof}
Suppose that $x_1,..., x_L\in S^{2n+1}$ are the blow-up
points. Let $r=\displaystyle\min_{i\neq j} d(x_i,x_j)/2$. For any
given $\epsilon>0$, if $k$ is sufficiently large, by (\ref{4.13}) we
have
\begin{equation}\label{4.17}
\begin{split}
L\left(Y(S^{2n+1},\theta_0)-\epsilon\right)
&\leq\sum_{i=1}^L\left(\int_{B_r(x_i)}|R_{\theta(t_k)}|^{n+1}dV_{\theta(t_k)}\right)^{\frac{1}{n+1}}\\
&\leq
L^{1-\frac{1}{n+1}}\left(\int_{S^{2n+1}}|R_{\theta(t_k)}|^{n+1}dV_{\theta(t_k)}\right)^{\frac{1}{n+1}}\\
&\leq
L^{1-\frac{1}{n+1}}\left(\int_{S^{2n+1}}|R_{\theta(t_k)}-\alpha(t_k)f|^{n+1}
dV_{\theta(t_k)}\right)^{\frac{1}{n+1}}
\\
&\hspace{4mm}+L^{1-\frac{1}{n+1}}\,\alpha(t_k)\left(\int_{S^{2n+1}}f^{n+1}
dV_{\theta(t_k)}\right)^{\frac{1}{n+1}}.
\end{split}
\end{equation}
Note that as $k\rightarrow\infty$ the first term on the right hand
side tends to 0 by Lemma \ref{lem3.2}. On the other hand, the second
term can be estimated as follows:
\begin{equation}\label{4.18}
\begin{split}
\alpha(t_k)&\left(\int_{S^{2n+1}}f^{n+1}dV_{\theta(t_k)}\right)^{\frac{1}{n+1}}\\
&=E_f(u(t_k))\left(\int_{S^{2n+1}}fdV_{\theta(t_k)}\right)^{-\frac{1}{n+1}}
\left(\int_{S^{2n+1}}f^{n+1}dV_{\theta(t_k)}\right)^{\frac{1}{n+1}}\\
&\leq E_f(u_0)\left(\max_{S^{2n+1}}f\right)^{\frac{n}{n+1}}\leq \beta\left(\max_{S^{2n+1}}f\right)^{\frac{n}{n+1}}\\
&=(1+\epsilon_0)Y(S^{2n+1},\theta_0)\left(\frac{\max_{S^{2n+1}}f}{\min_{S^{2n+1}}f}\right)^{\frac{n}{n+1}}
\end{split}
\end{equation}
by (\ref{2.6}), (\ref{2.7}), (\ref{4.15}), and Proposition
\ref{prop2.2}. Combining (\ref{4.17}) and (\ref{4.18}), and using
(\ref{4.16}), we obtain
$$L\left(Y(S^{2n+1},\theta_0)-\epsilon\right)\leq
L^{1-\frac{1}{n+1}}(1+\epsilon_0)Y(S^{2n+1},\theta_0)\delta_n^{\frac{n}{n+1}}$$
for any $\epsilon>0$. This implies that
$$L^{\frac{1}{n+1}}\leq
(1+\epsilon_0)\delta_n^{\frac{n}{n+1}}<2^{\frac{1}{n+1}},$$
which implies $L<2$. Since $L$ is a
natural number, one can easily conclude that $L=1$.
\end{proof}

\begin{lem}\label{lem4.5}
The blow-up point in Lemma \ref{lem4.4} does not depend on the
special choice of the sequence $t_k$.
\end{lem}
\begin{proof}
Suppose it were not true. Then there would exist $x_1\neq x_2$ in
$S^{2n+1}$ and two sequences $\{t_j\}$ and $\{t_k\}$ such that the
sequences $\{u(t_j)\}$ and $\{u(t_k)\}$ are blow-up at $x_1$ and
$x_2$ respectively. Then we define the new sequence $\{u(t_l)\}$
such that $u(t_{2j})=u(t_j)$ and $u(t_{2k+1})=u(t_k)$. That is, all
the even terms of $\{u(t_l)\}$ consist of the sequence $\{u(t_j)\}$
while all the odd terms consist of the sequence $\{u(t_k)\}$. Then
$\{u(t_l)\}$ would be blow-up at $x_1$ and $x_2$, which contradicts
Lemma \ref{lem4.4}.
\end{proof}

\begin{lem}\label{lem4.6}
Any sequence $u_k=u(t_k)$ with $t_k\rightarrow\infty$ is a
Palais-Smale sequence of the energy functional $E_f(u)$ if $u$ is a
solution of the flow
$(\ref{2.3})$ with a fixed initial data $u_0$.
\end{lem}
\begin{proof}
By Palais-Smale sequence we mean that $\{u_k\}$ is bounded in
$S_1^2(S^{2n+1},\theta_0)$ and $dE_f(u_k)\rightarrow 0$ as
$k\rightarrow\infty$. As $k\rightarrow\infty$,
$$e_k=E_f(u_k)\rightarrow e_\infty$$
since $E_f(u(t))$ is monotonic decreasing in time $t$ by Proposition
\ref{prop2.2}.  Since $E(u_k)$ is bounded by (\ref{2.12a}),  $\{u_k\}$ is bounded in
$S_1^2(S^{2n+1},\theta_0)$. On the other hand, for any $\varphi\in
S_1^2(S^{2n+1},\theta_0)\hookrightarrow
L^{2+\frac{2}{n}}(S^{2n+1},\theta_0)$, there
holds
\begin{equation*}
\begin{split}
&\frac{1}{2}\left(\int_{S^{2n+1}}fu_k^{2+\frac{2}{n}}dV_{\theta_0}\right)^{\frac{n}{n+1}}|\langle
DE_f(u_k),\varphi\rangle|\\
&=\left|\int_{S^{2n+1}}\left((2+\frac{2}{n})\langle\nabla_{\theta_0}u_k,\nabla_{\theta_0}\varphi\rangle_{\theta_0}
+R_{\theta_0}u_k\varphi
\right)dV_{\theta_0}
-\alpha(t_k)\int_{S^{2n+1}}fu_k^{1+\frac{2}{n}}\varphi\,dV_{\theta_0}\right|\\
&=\left|\int_{S^{2n+1}}(R_{\theta_k}-\alpha(t_k)f)u_k^{1+\frac{2}{n}}\varphi\,dV_{\theta_0}\right|\\
&\leq\left(\int_{S^{2n+1}}|R_{\theta_k}-\alpha(t_k)f|^{\frac{2n+2}{n+2}}u_k^{2+\frac{2}{n}}
dV_{\theta_0}\right)^{\frac{n+2}{2n+2}}
\left(\int_{S^{2n+1}}|\varphi|^{2+\frac{2}{n}}dV_{\theta_0}\right)^{\frac{n}{2n+2}}\rightarrow
0
\end{split}
\end{equation*}
as $k\rightarrow\infty$ by Lemma \ref{lem3.2}.
\end{proof}

For every smooth positive function $u(t)$, set
$P(t)=\displaystyle\int_{S^{2n+1}}x
u(t)^{2+\frac{2}{n}}dV_{\theta_0}$ where $x=(x_1,...,x_{n+1})\in
S^{2n+1}\subset\mathbb{C}^{n+1}$, and we define
\begin{equation}\label{cm}
\widehat{P(t)}=\displaystyle\frac{P(t)}{\|P(t)\|}\mbox{ if }\|P(t)\|\neq 0,
\mbox{ otherwise }\widehat{P(t)}=P(t).
\end{equation}
Clearly $\widehat{P(t)}\in S^{2n+1}$ smoothly depends
on the time $t$ if $u$ does. There exists a family of conformal CR
diffeomorphisms $\phi(t): S^{2n+1}\rightarrow S^{2n+1}$ such that
(see \cite{Frank&Lieb})
\begin{equation}\label{4.19}
\int_{S^{2n+1}}x\,dV_h=(0,...,0)\in\mathbb{C}^{n+1}\hspace{2mm}\mbox{ for all }t>0,
\end{equation}
where the new contact form
\begin{equation}\label{4.19a}
h=h(t)=\phi(t)^*\big(\theta(t)\big)=v(t)^{2+\frac{2}{n}}\theta_0
\end{equation}
is called the normalized contact form with
$v=v(t)=(u(t)\circ\phi(t))\big|\det(d\phi(t))\big|^{\frac{n}{2n+2}}$
and the volume form $dV_h=v(t)^{2+\frac{2}{n}}dV_{\theta_0}$. In
fact, the conformal CR diffeomorphism may be represented as
$\phi(t)=\phi_{p(t),r(t)}=\Psi\circ T_{p(t)}\circ D_{r(t)}\circ\pi$
for some $p(t) \in\mathbb{H}^n$ and $r(t)>0$. Here the CR
diffeomorphism $\pi:S^{2n+1}\setminus\{(0,...,0,-1)\}\rightarrow\mathbb{H}^n$
is given by
\begin{equation}\label{4.20}
\pi(x)=\left(\frac{x'}{1+x_{n+1}},Re\big(\sqrt{-1}\frac{1-x_{n+1}}{1+x_{n+1}}\big)\right)
, x=(x',x_{n+1})\in S^{2n+1},
\end{equation}
 where $\mathbb{H}^n$ denotes the
Heisenberg group, and
 $D_{\lambda},
T_{(z',\tau')}:\mathbb{H}^n\rightarrow\mathbb{H}^n$ are respectively
the dilation and translation on $\mathbb{H}^n$ given by
\begin{equation}\label{4.21}
D_{\lambda}(z,\tau)=(\lambda z, \lambda^2\tau)\mbox{ and }
T_{(z',\tau')}(z,\tau)=(z+z',\tau+\tau'+2Im(z'\cdot
\overline{z}))\mbox{ for }(z,\tau)\in\mathbb{H}^n.
\end{equation}
And $\Psi=\pi^{-1}$ is the inverse of $\pi$.

Meanwhile, the normalized function $v$ satisfies
\begin{equation}\label{4.22}
-(2+\frac{2}{n})\Delta_{\theta_0}v+R_{\theta_0}v=R_hv^{1+\frac{2}{n}},
\end{equation}
where $R_h=R_\theta\circ\phi(t)$ is the Webster scalar curvature of
the normalized contact form $h=h(t)$ in view of (\ref{4.19a}). Hereafter, we set
$f_\phi=f\circ\phi$.

Now we state our main result of this section.

\begin{theorem}\label{thm4.7}
For any given $u_0$ satisfying $(\ref{2.01})$ with
$E_f(u_0)\leq\beta$, consider the flow $\theta(t)$ with initial data
$u_0$. Let $\{t_k\}$ be a time sequence of the flow with
$t_k\rightarrow\infty$ as $k\rightarrow\infty$. Let $\{\theta_k\}$
be the corresponding contact forms such that
$\theta_k=u(t_k)^{\frac{2}{n}}\theta_0$. Assume that
$\|R_{\theta_k}-R_\infty\|_{L^{p_1}(S^{2n+1},\theta_k)}\rightarrow
0$ as $k\rightarrow\infty$ for some $p_1>n+1$ and a smooth function
$R_\infty>0$ defined on $S^{2n+1}$ which satisfies the simple bubble
condition (sbc):
$$\frac{\max_{S^{2n+1}}R_\infty}{\min_{S^{2n+1}}R_\infty}<2^{\frac{1}{n}}.$$
Then, up to a subsequence, either\\
\emph{(i)} $\{u_k\}$ is uniformly bounded in $S^p_2(S^{2n+1},\theta_0)$ for some $p\in(n+1,p_1)$.
 Furthermore, $u_k\rightarrow u_\infty$ in
$S^p_2(S^{2n+1},\theta_0)$ as $k\rightarrow\infty$, where $\theta_\infty=u_\infty^{\frac{2}{n}}\theta_0$
has Webster scalar curvature $R_\infty$, or\\
\emph{(ii)} let
$h_k=\phi(t_k)^*(\theta_k)=v_k^{\frac{2}{n}}\theta_0$ be the
associated sequence of the normalized contact forms satisfying
$\displaystyle\int_{S^{2n+1}}x\,dV_{h_k}=(0,...,0)\in\mathbb{C}^{n+1}$. Then, there exists
$Q\in S^{2n+1}$ such that
\begin{equation}\label{4.23}
dV_{\theta_k}\rightharpoonup\mbox{\emph{Vol}}(S^{2n+1},\theta_0)\delta_Q,\hspace{2mm}\mbox{
as }k\rightarrow\infty
\end{equation}
in the weak sense of measures. In addition, for any $\lambda\in (0,1)$, we have
\begin{equation}\label{4.24}
v_k\rightarrow 1\mbox{ in }
C^{1,\lambda}_P(S^{2n+1})\hspace{2mm}\mbox{ as }k\rightarrow\infty.
\end{equation}
Here $C^{1,\lambda}_P(S^{2n+1})$ is the parabolic H\"{o}rmander
H\"{o}lder spaces defined as in $(\ref{2.24})$.
\end{theorem}

Due to the length of the proof of Theorem \ref{thm4.7}, we will
divide the proof into several lemmas.

\begin{lem}\label{lem4.8}
Suppose case \emph{(i)} occurs. Then $f$ can be realized as the
Webster scalar curvature of some contact form conformal to
$\theta_0$.
\end{lem}
\begin{proof}
If $\{u_k\}$ is uniformly bounded in $S^p_2(S^{2n+1},\theta_0)$ for
some $p>n+1$, then up to a subsequence, there exists $u_\infty\in
S^p_2(S^{2n+1},\theta_0)$ such that $u_k\rightarrow u_\infty$ weakly
in $S^p_2(S^{2n+1},\theta_0)$ and strongly in $C^\sigma(S^{2n+1})$
for any $0<\sigma<(p-n-1)/p$  as $k\rightarrow\infty$, where
$C^\sigma(S^{2n+1})$ is the standard H\"{o}lder space (see Theorem
3.16 and 3.17 in \cite{Dragomir}, and also \cite{Folland&Stein}).
Since
$\|R_{\theta_k}-R_\infty\|_{L^{p_1}(S^{2n+1},\theta_k)}\rightarrow
0$ as $k\rightarrow\infty$ and $u_k$ satisfies
$$-(2+\frac{2}{n})\Delta_{\theta_0}u_k+R_{\theta_0}u_k
=R_{\theta_k}u_k^{1+\frac{2}{n}}\mbox{ on }S^{2n+1},$$
so $u_\infty$ weakly solves
\begin{equation}\label{4.25}
-(2+\frac{2}{n})\Delta_{\theta_0}u_\infty+R_{\theta_0}u_\infty
=R_{\infty}u_\infty^{1+\frac{2}{n}}\mbox{ on }S^{2n+1}.
\end{equation}
By Theorem 3.22 in \cite{Dragomir}, $u_\infty\in C^\infty(S^{2n+1})$
if $R_\infty\in C^\infty(S^{2n+1})$. Since
$\mbox{Vol}(S^{2n+1},\theta_k)
=\displaystyle\int_{S^{2n+1}}u_k^{2+\frac{2}{n}}dV_{\theta_0}
=\int_{S^{2n+1}}u_0^{2+\frac{2}{n}}dV_{\theta_0}>0$ by Proposition
\ref{prop2.1},
$\displaystyle\int_{S^{2n+1}}u_\infty^{2+\frac{2}{n}}dV_{\theta_0}>0$.
Since $u_k>0$, we have $u_\infty\geq 0$. That is, $u_\infty$ is
nonnegative and not identically zero on $S^{2n+1}$. Applying
Proposition A.1 in \cite{Ho2}, we get $u_\infty>0$ on $S^{2n+1}$.
Hence there exists constant $C>0$ such that
\begin{equation}\label{4.26}
C^{-1}\leq u_\infty\leq C\hspace{2mm}\mbox{ on }S^{2n+1}.
\end{equation}
Moreover, we have
\begin{equation}\label{4.27}
R_{\theta_k}\rightarrow R_\infty\hspace{2mm}\mbox{ in }L^{p_1}(S^{2n+1},\theta_0),
\mbox{ as }k\rightarrow\infty,
\end{equation}
and
\begin{equation}\label{4.28}
u_k\rightarrow u_\infty\hspace{2mm}\mbox{ in }S^{p}_2(S^{2n+1},\theta_0),
\mbox{ as }k\rightarrow\infty.
\end{equation}
Hence, by (\ref{4.26}), (\ref{4.28}), and Lemma \ref{lem3.2}, we
have
\begin{equation}\label{4.29}
R_{\theta_k}-\alpha(t_k)f\rightarrow 0\hspace{2mm}\mbox{ in }L^p(S^{2n+1},\theta_0).
\end{equation}
 On the other hand, by (\ref{4.26}) and H\"{o}lder's inequality, we have
\begin{equation*}
\begin{split}
|\alpha(t_{k})-\alpha(t_{l})|\int_{S^{2n+1}}f^pdV_{\theta_0}
\leq&\int_{S^{2n+1}}|R_{\theta_k}-\alpha(t_{k})f|^pdV_{\theta_0}
+\int_{S^{2n+1}}|R_{\theta_{l}}-\alpha(t_{l})f|^pdV_{\theta_0}\\
&+\left(\int_{S^{2n+1}}|R_{\theta_{k}}-R_\infty|^{p_1}
dV_{\theta_0}\right)^{\frac{p_1}{p}}
\left(\int_{S^{2n+1}}dV_{\theta_0}\right)^{\frac{p-p_1}{p}}\\
&
+\left(\int_{S^{2n+1}}|R_{\theta_l}-R_\infty|^{p_1}dV_{\theta_0}\right)^{\frac{p_1}{p}}
\left(\int_{S^{2n+1}}dV_{\theta_0}\right)^{\frac{p-p_1}{p}}
\end{split}
\end{equation*}
which tends to 0 as $k,l\rightarrow\infty$ by (\ref{4.27}) and
(\ref{4.29}). That is, $\{\alpha(t_k)\}$ is a Cauchy sequence, which
implies that $\alpha(t_k)\rightarrow\alpha_\infty$ as
$k\rightarrow\infty$. Combining all these, we have
$R_\infty=\alpha_\infty f$ for some $\alpha_\infty>0$. Therefore, up
to a constant multiple, $u_\infty$ is a solution we want in view of
(\ref{4.25}).
\end{proof}

We are now ready to study case (ii), i.e. study the normalized flow $v(t)$ defined in (\ref{4.19a}).
For convenience,
for each conformal CR diffeormorphism $\phi$ from $S^{2n+1}$ to itself,
we denote
$(u\circ\phi)|\det(d\phi)|^{\frac{n}{2n+2}}$ by $v$.
Note that $v$ enjoys the following properties:
\begin{equation}\label{4.30}
E(v)=E(u)\hspace{2mm}\mbox{ and }\hspace{2mm}
\int_{S^{2n+1}}v^{2+\frac{2}{n}}dV_{\theta_0}
=\int_{S^{2n+1}}u^{2+\frac{2}{n}}dV_{\theta_0}.
\end{equation}

\begin{lem}\label{lem4.9}
There exists a constant $C_0$ depending only on $n$, such that,
for the normalized conformal factor $v(t)$, we have
$$\int_{S^{2n+1}}v(t)^2dV_{\theta_0}\geq C_0>0$$
for all $t\geq 0$ with initial data in $C^\infty_f$ which is defined in the proof.
\end{lem}
\begin{proof}
It follows from  (\ref{4.30}) and Proposition \ref{prop2.2} that
\begin{equation*}
\begin{split}
E(v(t))&=E(u(t))=E_f(u(t))\left(\int_{S^{2n+1}}fu(t)^{2+\frac{2}{n}}
dV_{\theta_0}\right)^{\frac{n}{n+1}}\\
&\leq E_f(u_0)\left(\int_{S^{2n+1}}fu(t)^{2+\frac{2}{n}}dV_{\theta_0}\right)^{\frac{n}{n+1}}\\
&\leq \beta\left(\big(\max_{S^{2n+1}}f\big)\int_{S^{2n+1}}u_0^{2+\frac{2}{n}}
dV_{\theta_0}\right)^{\frac{n}{n+1}}\\
&=(1+\epsilon_0)Y(S^{2n+1},\theta_0)\mbox{Vol}(S^{2n+1},\theta_0)^{\frac{n}{n+1}}\left(\frac{\max_{S^{2n+1}}f}{\min_{S^{2n+1}}f}\right)^{\frac{n}{n+1}}\\
&=(1+\epsilon_0)R_{\theta_0}\mbox{Vol}(S^{2n+1},\theta_0)\left(\frac{\max_{S^{2n+1}}f}{\min_{S^{2n+1}}f}\right)
^{\frac{n}{n+1}}\\
&\leq(1+\epsilon_0)R_{\theta_0}\mbox{Vol}(S^{2n+1},\theta_0)\delta_n^{\frac{n}{n+1}},
\end{split}
\end{equation*}
where the initial data $u_0$ satisfies
$$u_0\in C^\infty_f:=\{u\in C_*^\infty: u>0\mbox{ and }E_f(u)\leq \beta\}$$
with $\beta$ defined as (\ref{4.15}) and
$$C^\infty_*:=\Big\{0<u\in C^\infty(S^{2n+1}): \theta=u^{\frac{2}{n}}\theta_0
\mbox{ satisfies }\int_{S^{2n+1}}u^{2+\frac{2}{n}}dV_{\theta_0}
=\int_{S^{2n+1}}dV_{\theta_0}\Big\}.$$

Choose
$\epsilon=\displaystyle\frac{\epsilon_0}{2}(1+\epsilon_0)^{-1}\delta_n^{-\frac{n}{n+1}}>0$,
then
\begin{equation*}
\begin{split}
(1+\epsilon_0)(2^{-\frac{1}{n+1}}+\epsilon)\delta_n^{\frac{n}{n+1}}
&=(1+\epsilon_0)2^{-\frac{1}{n+1}}\delta_n^{\frac{n}{n+1}}+\epsilon(1+\epsilon_0)\delta_n^{\frac{n}{n+1}}\\
&=(1-\epsilon_0)+\frac{\epsilon_0}{2}=1-\frac{\epsilon_0}{2}<1,
\end{split}
\end{equation*}
thanks to
$\displaystyle\frac{\delta_n^{\frac{n}{n+1}}}{2^{\frac{1}{n+1}}}=\frac{1-\epsilon_0}{1+\epsilon_0}$.
By Lemma \ref{lemA}, there exists a constant $C_\epsilon$ such that
$$Y(S^{2n+1},\theta_0)\left(\int_{S^{2n+1}}v^{2+\frac{2}{n}}dV_{\theta_0}\right)^{\frac{n}{n+1}}\leq (2^{-\frac{1}{n+1}}+\epsilon)(2+\frac{2}{n})
\int_{S^{2n+1}}|\nabla_{\theta_0} v|^2_{\theta_0}
dV_{\theta_0}+C_\epsilon\int_{S^{2n+1}}v^2dV_{\theta_0}.$$
By (\ref{4.30}) and Proposition \ref{prop2.1}, we have
$$
R_{\theta_0}\mbox{Vol}(S^{2n+1},\theta_0)\leq (2^{-\frac{1}{n+1}}+\epsilon)(2+\frac{2}{n})
\int_{S^{2n+1}}|\nabla_{\theta_0} v|^2_{\theta_0}
dV_{\theta_0}+C_\epsilon\int_{S^{2n+1}}v^2dV_{\theta_0}.$$
Combining the above inequalities, we obtain
\begin{equation*}
\big(C_\epsilon-(2^{-\frac{1}{n+1}}+\epsilon)R_{\theta_0}\big)\int_{S^{2n+1}}v^2dV_{\theta_0}\geq R_{\theta_0}\mbox{Vol}(S^{2n+1},\theta_0)\big[1-(1+\epsilon_0)(2^{-\frac{1}{n+1}}+\epsilon)\delta_n^{\frac{n}{n+1}}
\big]
\end{equation*}
which implies the desired estimate because of our choice of $\epsilon$ above.
\end{proof}

The following definition appeared in \cite{Chang&Chiu&Wu}, \cite{Chang&Gursky&Yang} and \cite{Chiu}: A sequence
of positive functions $v_k$ defined on $S^{2n+1}$ satisfies condition $(\ast)$ if there
is a set $\Omega_k\subset S^{2n+1}$ with $|\Omega_k|\geq C_1>0$, and a constant $C_2, \epsilon>0$ such that
$$\int_{\Omega_k}v_k^{-\epsilon}dV_{\theta_0}\leq C_2,$$
where $C_1, C_2$ are two constants independent of $k$. We have the following (see also Lemma 4.4 in \cite{Chang&Gursky&Yang}):

\begin{lem}\label{lem4.9.5}
If condition $(\ast)$ does not hold with  $\epsilon>0$ being small enough, then  $\displaystyle\int_{S^{2n+1}}v_k^\epsilon dV_{\theta_0}\rightarrow 0$
as $k\rightarrow\infty$.
\end{lem}
\begin{proof}
We follow the proof of Theorem A.2 in \cite{Schwetlick&Struwe}.
By (\ref{4.22}), we have
\begin{equation}\label{4.80}
-(2+\frac{2}{n})\Delta_{\theta_0}v+R_{\theta_0}v=(R_{h}-\alpha f_\phi)v^{1+\frac{2}{n}}+\alpha f_\phi v^{1+\frac{2}{n}}
\geq (R_{h}-\alpha f_\phi)v^{1+\frac{2}{n}}.
\end{equation}
Let $w=\log v-\beta$, where
$\beta$ is a constant chosen such that $\int_{S^{2n+1}}wdV_{\theta_0}=0$. Note that
$$|\nabla_{\theta_0} w|_{\theta_0}^2+\Delta_{\theta_0}w=\frac{\Delta_{\theta_0}v}{v},$$
which together with (\ref{4.80}) implies that
\begin{equation}\label{4.81}
|\nabla_{\theta_0} w|_{\theta_0}^2+\Delta_{\theta_0}w\leq\frac{R_{\theta_0}}{(2+\frac{2}{n})}+\frac{(\alpha f_\phi-R_{h})v^{\frac{2}{n}}}{(2+\frac{2}{n})}
=\frac{n^2}{4}+\frac{(\alpha f_\phi-R_{h})v^{\frac{2}{n}}}{(2+\frac{2}{n})}.
\end{equation}
Integrating (\ref{4.81}) over $S^{2n+1}$, we get
\begin{equation}\label{4.82}
\begin{split}
&\int_{S^{2n+1}}|\nabla_{\theta_0} w|_{\theta_0}^2dV_{\theta_0}\\
&\leq
\frac{n^2}{4}\int_{S^{2n+1}}dV_{\theta_0}+\frac{n}{2n+2}\int_{S^{2n+1}}(\alpha f_\phi-R_{h})v^{\frac{2}{n}}dV_{\theta_0}\\
&\leq\frac{n^2}{4}\int_{S^{2n+1}}dV_{\theta_0}+\frac{n}{2n+2}
\left(\int_{S^{2n+1}}|\alpha f_\phi-R_{h}|^{n+1}v^{2+\frac{2}{n}}dV_{\theta_0}\right)^{\frac{1}{n+1}}
\left(\int_{S^{2n+1}}dV_{\theta_0}\right)^{\frac{n}{n+1}}\leq C_0
\end{split}
\end{equation}
by H\"{o}lder's inequality and Lemma \ref{lem3.2}. Now, for any $p\in\mathbb{Z}^+$
multiplying (\ref{4.81}) by $|w|^{2p}\geq 0$ and integrating over $S^{2n+1}$, we get
\begin{equation}\label{4.83}
\begin{split}
&\int_{S^{2n+1}}(|w|^{2p}|\nabla_{\theta_0} w|_{\theta_0}^2+|w|^{2p}\Delta_{\theta_0}w)dV_{\theta_0}\\
&\leq \frac{n^2}{4}\int_{S^{2n+1}}|w|^{2p}dV_{\theta_0}+\frac{n}{2n+2}\int_{S^{2n+1}}(\alpha f_\phi-R_{h})v^{\frac{2}{n}}|w|^{2p}dV_{\theta_0}\\
&\leq \frac{n^2}{4}\int_{S^{2n+1}}|w|^{2p}dV_{\theta_0}+\frac{n}{2n+2}
\left(\int_{S^{2n+1}}|\alpha f_\phi-R_{h}|^{n+1}v^{2+\frac{2}{n}}dV_{\theta_0}\right)^{\frac{1}{n+1}}
\left(\int_{S^{2n+1}}|w|^{2p(\frac{n+1}{n})}dV_{\theta_0}\right)^{\frac{n}{n+1}}\\
&\leq\frac{n^2}{4}\int_{S^{2n+1}}|w|^{2p}dV_{\theta_0}\\
&\hspace{4mm}+\frac{n}{2n+2}F_{n+1}(t)^{\frac{1}{n+1}}Y(S^{2n+1},\theta_0)^{-1}
\left((2+\frac{2}{n})
\int_{S^{2n+1}}|\nabla_{\theta_0} w^p|^2_{\theta_0}dV_{\theta_0}+R_{\theta_0}\int_{S^{2n+1}}|w|^{2p}dV_{\theta_0}\right)
\end{split}
\end{equation}
where we have used H\"{o}lder's inequality and Lemma \ref{lem2.3}. Since
$$\Delta_{\theta_0}|w|^{2p+1}=(2p+1)|w|^{2p}\Delta_{\theta_0}w+2p(2p+1)|w|^{2p-1}|\nabla_{\theta_0} w|^2_{\theta_0},$$
we can rewrite the left hand side of (\ref{4.83}) as
\begin{equation}\label{4.84}
\begin{split}
&\int_{S^{2n+1}}(|w|^{2p}|\nabla_{\theta_0} w|_{\theta_0}^2+|w|^{2p}\Delta_{\theta_0}w)dV_{\theta_0}\\
&=\int_{S^{2n+1}}|w|^{2p}|\nabla_{\theta_0} w|_{\theta_0}^2dV_{\theta_0}
-2p\int_{S^{2n+1}}|w|^{2p-1}|\nabla_{\theta_0} w|^2_{\theta_0}dV_{\theta_0}\\
&\geq\int_{S^{2n+1}}|w|^{2p}|\nabla_{\theta_0} w|_{\theta_0}^2dV_{\theta_0}
-\frac{1}{2}\int_{S^{2n+1}}|w|^{2p}|\nabla_{\theta_0} w|^2_{\theta_0}dV_{\theta_0}
-2p^2\int_{S^{2n+1}}|w|^{2p-2}|\nabla_{\theta_0} w|^2_{\theta_0}dV_{\theta_0}\\
&=\frac{1}{2(p+1)^2}\int_{S^{2n+1}}|\nabla_{\theta_0} w^{p+1}|^2_{\theta_0}dV_{\theta_0}
-2\int_{S^{2n+1}}|\nabla_{\theta_0} w^p|^2_{\theta_0}dV_{\theta_0}
\end{split}
\end{equation}
where we have used Young's inequality. Combining (\ref{4.83}) and  (\ref{4.84}), we obtain
\begin{equation}\label{4.85}
\begin{split}
\frac{1}{2(p+1)^2}\int_{S^{2n+1}}|\nabla_{\theta_0} w^{p+1}|^2_{\theta_0}dV_{\theta_0}\leq&
\Big(F_{n+1}(t)^{\frac{1}{n+1}}Y(S^{2n+1},\theta_0)^{-1}+2\Big)\int_{S^{2n+1}}|\nabla_{\theta_0} w^p|^2_{\theta_0}dV_{\theta_0}\\
&
+\frac{n^2}{4}\Big(F_{n+1}(t)^{\frac{1}{n+1}}Y(S^{2n+1},\theta_0)^{-1}+1\Big)\int_{S^{2n+1}}|w|^{2p}dV_{\theta_0}.
\end{split}
\end{equation}
By Lemma \ref{lem3.2}, there exists a constant $C_1$ such that $F_{n+1}(t)^{\frac{1}{n+1}}\leq C_1$ for all $t\geq 0$. Hence,
it follows from (\ref{4.85}) that
\begin{equation}\label{4.86}
\|\nabla_{\theta_0} w^{p+1}\|^2_{L^2(S^{2n+1},\theta_0)}\leq (p+1)^2(C_2\|\nabla_{\theta_0} w^{p}\|^2_{L^2(S^{2n+1},\theta_0)}
+C_3\|w^{p}\|^2_{L^2(S^{2n+1},\theta_0)})
\end{equation}
where $C_2=2C_1Y(S^{2n+1},\theta_0)^{-1}+4$ and $C_3=\displaystyle\frac{n^2}{2}(C_1Y(S^{2n+1},\theta_0)^{-1}+1)$.
Recall the following Poincar\'{e}-type inequality (see Theorem 3.20 \cite{Dragomir}):
there exists a constant $C_4$ such that
\begin{equation}\label{4.87}
\|\varphi\|_{L^2(S^{2n+1},\theta_0)}\leq C_4\|\nabla_{\theta_0} \varphi\|_{L^2(S^{2n+1},\theta_0)}
\end{equation}
for all $\varphi\in S_1^2(S^{2n+1},\theta_0)$ satisfying $\int_{S^{2n+1}}\varphi\,dV_{\theta_0}=0$. It follows that with
$A:=\max\{\sqrt{C_0},\sqrt{C_3}\}$ and $B:=\displaystyle\sqrt{\frac{C_2}{C_3}}+2C_4$ for all $p\in\mathbb{Z}^+$ there holds
\begin{equation}\label{4.88}
\|\nabla_{\theta_0} w^{p+1}\|_{L^2(S^{2n+1},\theta_0)}\leq A(p+1)\left(B\|\nabla_{\theta_0} w^{p}\|_{L^2(S^{2n+1},\theta_0)}
+\|w^{p}\|_{L^2(S^{2n+1},\theta_0)}\right).
\end{equation}
To see this, if we denote $\overline{w^p}=\displaystyle\frac{1}{\mbox{Vol}(S^{2n+1},\theta_0)}\int_{S^{2n+1}}w^pdV_{\theta_0}$,
 by (\ref{4.86}) and (\ref{4.87}) we have
\begin{equation*}
\begin{split}
&\|\nabla_{\theta_0} w^{p+1}\|_{L^2(S^{2n+1},\theta_0)}\\&\leq
(p+1)\left(\sqrt{C_2}\|\nabla_{\theta_0} w^{p}\|_{L^2(S^{2n+1},\theta_0)}
+\sqrt{C_3}\|w^{p}\|_{L^2(S^{2n+1},\theta_0)}\right)\\
&\leq
(p+1)\left(\sqrt{C_2}\|\nabla_{\theta_0} w^{p}\|_{L^2(S^{2n+1},\theta_0)}
+\sqrt{C_3}\|w^{p}-\overline{w^p}\|_{L^2(S^{2n+1},\theta_0)}
+\sqrt{C_3}\|\overline{w^p}\|_{L^2(S^{2n+1},\theta_0)}\right)\\
&\leq
(p+1)\left((\sqrt{C_2}+\sqrt{C_3}C_4)\|\nabla_{\theta_0} w^{p}\|_{L^2(S^{2n+1},\theta_0)}
+\sqrt{\frac{C_3}{\mbox{Vol}(S^{2n+1},\theta_0)}}\left|\int_{S^{2n+1}}w^pdV_{\theta_0}\right|\right)\\
&\leq
(p+1)\left((\sqrt{C_2}+\sqrt{C_3}C_4)\|\nabla_{\theta_0} w^{p}\|_{L^2(S^{2n+1},\theta_0)}
+\sqrt{C_3}\|w^{p}\|_{L^2(S^{2n+1},\theta_0)}\right)\\
&=
\sqrt{C_3}(p+1)\left((\sqrt{\frac{C_2}{C_3}}+C_4)\|\nabla_{\theta_0} w^{p}\|_{L^2(S^{2n+1},\theta_0)}
+\|w^{p}\|_{L^2(S^{2n+1},\theta_0)}\right).
\end{split}
\end{equation*}

By induction on $p$ we can now obtain the estimates
\begin{equation}\label{4.89}
\|\nabla_{\theta_0} w^{p}\|_{L^2(S^{2n+1},\theta_0)}\leq A^{p}B^{p-1}p^p\hspace{2mm}\mbox{ and }\hspace{2mm}
\|w^{p}\|_{L^2(S^{2n+1},\theta_0)}\leq A^{p}B^{p}p^p
\end{equation}
for all $p\in\mathbb{Z}^+$. Indeed, for $p=1$, (\ref{4.89}) follows from (\ref{4.82}),  (\ref{4.87}) and the fact that
$A\geq \sqrt{C_0}$ and $B\geq C_4$. Moreover, assuming that (\ref{4.89}) is true for some $p$. Then from (\ref{4.88}) and H\"{o}lder's inequality we deduce
\begin{equation*}
\begin{split}
\|\nabla_{\theta_0} w^{p+1}\|_{L^2(S^{2n+1},\theta_0)}&\leq A(p+1)\left(B\|\nabla_{\theta_0} w^{p}\|_{L^2(S^{2n+1},\theta_0)}
+\|w^{p}\|_{L^2(S^{2n+1},\theta_0)}\right)\\
&\leq A(p+1)\left(B\cdot A^{p}B^{p-1}p^p+A^{p}B^{p}p^p\right)
=A^{p+1}B^p2(p+1)p^p.
\end{split}
\end{equation*}
Now the first inequality in (\ref{4.89}) for $p+1$ follows from this inequality, because
Bernoulli's inequality says that $(1+t)^{\frac{1}{p}}\leq 1+\frac{t}{p}$ for all $t>0$, which implies
that $2\leq \big((p+1)/p\big)^p$. Similarly, by (\ref{4.87}), we have
\begin{equation*}
\|w^{p+1}-\overline{w^{p+1}}\|_{L^2(S^{2n+1},\theta_0)}\leq C_4\|\nabla_{\theta_0}w^{p+1}\|_{L^2(S^{2n+1},\theta_0)},
\end{equation*}
which implies that
\begin{equation*}
\begin{split}
\|w^{p+1}\|_{L^2(S^{2n+1},\theta_0)}&\leq C_4\|\nabla_{\theta_0}w^{p+1}\|_{L^2(S^{2n+1},\theta_0)}+
\|\overline{w^{p+1}}\|_{L^2(S^{2n+1},\theta_0)}\\
&= C_4\|\nabla_{\theta_0}w^{p+1}\|_{L^2(S^{2n+1},\theta_0)}+
\frac{1}{\sqrt{\mbox{Vol}(S^{2n+1},\theta_0)}}\left|\int_{S^{2n+1}}w^{p+1}dV_{\theta_0}\right|\\
&\leq C_4\|\nabla_{\theta_0}w^{p+1}\|_{L^2(S^{2n+1},\theta_0)}+
\|w^{p+1}\|_{L^1(S^{2n+1},\theta_0)}\\
&\leq C_4\|\nabla_{\theta_0}w^{p+1}\|_{L^2(S^{2n+1},\theta_0)}+
\|w^{p}\|_{L^2(S^{2n+1},\theta_0)}\|w\|_{L^2(S^{2n+1},\theta_0)}\\
&\leq C_4A^{p+1}B^{p}(p+1)^{p+1}+A^{p}B^{p}p^p\cdot AB\\
&\leq \frac{1}{2}A^{p+1}B^{p+1}(p+1)^{p+1}+A^{p+1}B^{p+1}p^p\\
&\leq A^{p+1}B^{p+1}(p+1)^{p+1}
\end{split}
\end{equation*}
since $2C_4\leq B$. This completes the induction step.

Now using Stirling's formula to estimate $p!\geq\displaystyle\left(\frac{p}{e}\right)^p$ for any $p\in\mathbb{Z}^+$, for any $p_0>0$ we can estimate
\begin{equation*}
\int_{S^{2n+1}}(e^{p_0|w|}-1)dV_{\theta_0}\leq\sum_{p=1}^\infty\int_{S^{2n+1}}\left(\frac{p_0e|w|}{p}\right)^pdV_{\theta_0}
\leq\sum_{p=1}^\infty(p_0eAB)^p,
\end{equation*}
which is finite whenever $p_0eAB<1$. Choosing $p_0=\displaystyle\frac{1}{2eAB}>0$, we then conclude that
\begin{equation}\label{4.90}
\int_{S^{2n+1}}v^{-p_0}dV_{\theta_0}\cdot \int_{S^{2n+1}}v^{p_0}dV_{\theta_0}\leq\left(\int_{S^{2n+1}}\exp(p_0|w|)dV_{\theta_0}\right)^2\leq C.
\end{equation}
The inequality (\ref{4.90}) says that $v_k$ does not satisfy condition $(\ast)$ with $\epsilon=p_0$ if and only if
$\int_{S^{2n+1}}v^{p_0}_kdV_{\theta_0}\rightarrow 0$ as $k\rightarrow\infty$.
\end{proof}

\begin{cor}\label{cor4.9.5}
The normalized conformal factor $v(t)$ satisfies condition $(\ast)$.
\end{cor}
\begin{proof}
For $0<\delta<2$, let $\epsilon=\displaystyle\frac{(2+2n)\delta}{2+n\delta}>0$. Then by (\ref{4.30}), H\"{o}lder's inequality, Proposition \ref{prop2.1}, and Lemma \ref{lem4.9}, we have
\begin{equation*}
\begin{split}
0<C_0\leq\int_{S^{2n+1}}v(t)^2dV_{\theta_0}&\leq
\left(\int_{S^{2n+1}}v(t)^{\frac{2+2n}{n}}dV_{\theta_0}\right)^{\frac{n(2-\delta)}{2+2n}}
\left(\int_{S^{2n+1}}v(t)^{\delta\cdot\frac{2+2n}{2+n\delta}}dV_{\theta_0}\right)^{\frac{2+n\delta}{2+2n}}\\
&=\mbox{Vol}(S^{2n+1},\theta_0)^{\frac{n(2-\delta)}{2+2n}}
\left(\int_{S^{2n+1}}v(t)^{\epsilon}dV_{\theta_0}\right)^{\frac{2+n\delta}{2+2n}}.
\end{split}
\end{equation*}
Hence, it follows from Lemma \ref{lem4.9.5} that $v(t)$ satisfies condition $(\ast)$ with the $\epsilon>0$ defined above.
\end{proof}

\begin{lem}\label{lem4.10}
There exists a uniform constant $C_3>0$ such that $v(t)\geq C_3>0$ on $S^{2n+1}$ for all $t\geq 0$ and
 any $u_0\in C^\infty_f$.
\end{lem}
\begin{proof}
Note that $v(t)$ satisfies (\ref{4.22}). For $\delta>0$, multiply
(\ref{4.22}) by $v^{-1-2\delta}$, then integrate both sides, we get
\begin{equation}\label{4.31}
\begin{split}
&(2+\frac{2}{n})\int_{S^{2n+1}}|\nabla_{\theta_0}v^{-\delta}|_{\theta_0}^2dV_{\theta_0}\\
&=-\frac{\delta^2}{1+2\delta}\int_{S^{2n+1}}R_hv^{1+\frac{2}{n}-1-2\delta}dV_{\theta_0}
+\frac{\delta^2R_{\theta_0}}{1+2\delta}
\int_{S^{2n+1}}v^{-2\delta}dV_{\theta_0}\\
&=-\frac{\delta^2}{1+2\delta}\int_{S^{2n+1}}(R_h-\alpha f_\phi)v^{\frac{2-2n\delta}{n}}dV_{\theta_0}
-\frac{\delta^2}{1+2\delta}\int_{S^{2n+1}}\alpha f_\phi v^{\frac{2-2n\delta}{n}}dV_{\theta_0}\\
&\hspace{4mm}+\frac{\delta^2R_{\theta_0}}{1+2\delta}
\int_{S^{2n+1}}v^{-2\delta}dV_{\theta_0}\\
&\leq C(\delta,n)\|R_h-f_\phi\|_{L^{\frac{2n+2}{2-2n\delta}}(S^{2n+1},h)}
+\frac{\delta^2R_{\theta_0}}{1+2\delta}
\int_{S^{2n+1}}v^{-2\delta}dV_{\theta_0}
\end{split}
\end{equation}
where we have used the short form $f_\phi$ for $f\circ \phi(t)$.

Let $\lambda_1=n/2$ be the first nonzero eigenvalue of
$-\Delta_{\theta_0}$, then for any $\varphi\in C^\infty(S^{2n+1})$,
Raleigh's inequality gives
$$\int_{S^{2n+1}}\varphi^2\,dV_{\theta_0}\leq\frac{1}{\mbox{Vol}(S^{2n+1},\theta_0)}
\left(\int_{S^{2n+1}}\varphi\,dV_{\theta_0}\right)^2
+\lambda_1^{-1}\int_{S^{2n+1}}|\nabla_{\theta_0}\varphi|_{\theta_0}^2dV_{\theta_0}.$$
Choose $\varphi=v^{-\delta}$, then
\begin{equation}\label{4.32a}
\int_{S^{2n+1}}v^{-2\delta}dV_{\theta_0}\leq\frac{1}{\mbox{Vol}(S^{2n+1},\theta_0)}
\left(\int_{S^{2n+1}}v^{-\delta}dV_{\theta_0}\right)^2
+\lambda_1^{-1}\int_{S^{2n+1}}|\nabla_{\theta_0}v^{-\delta}|_{\theta_0}^2dV_{\theta_0}.
\end{equation}
It follows from Corollary \ref{cor4.9.5} that $v(t)$ satisfies condition $(\ast)$ with some $\epsilon>0$.
By condition $(\ast)$, if $\delta<\epsilon$, we have
\begin{equation*}
\int_{S^{2n+1}}v^{-\delta}dV_{\theta_0}=\int_{\Omega}v^{-\delta}dV_{\theta_0}
+\int_{\Omega^c}v^{-\delta}dV_{\theta_0}\leq C_2+\left(\int_{\Omega^c}v^{-2\delta}dV_{\theta_0}\right)
^{\frac{1}{2}}\left(\int_{\Omega^c}dV_{\theta_0}\right)^{\frac{1}{2}}.
\end{equation*}
This and Young's inequality imply that, for any $\eta>0$, there holds
\begin{equation}\label{4.33}
\begin{split}
\frac{1}{\mbox{Vol}(S^{2n+1},\theta_0)}\left(\int_{S^{2n+1}}v^{-\delta}dV_{\theta_0}\right)^2
&
\leq \frac{1}{\mbox{Vol}(S^{2n+1},\theta_0)}\left[C_2+\left(\int_{\Omega^c}v^{-2\delta}
dV_{\theta_0}\right)^{\frac{1}{2}}\left(\int_{\Omega^c}
dV_{\theta_0}\right)^{\frac{1}{2}}\right]^2\\
&\leq\frac{(1+\eta^{-1})C^2_2}{\mbox{Vol}(S^{2n+1},\theta_0)}
+\frac{(1+\eta)|\Omega^c|}{\mbox{Vol}(S^{2n+1},\theta_0)}
\left(\int_{\Omega^c}v^{-2\delta}dV_{\theta_0}\right).
\end{split}
\end{equation}
Since $\mbox{Vol}(\Omega,\theta_0)>0$, then
$$\frac{\mbox{Vol}(\Omega^c,\theta_0)}{\mbox{Vol}(S^{2n+1},\theta_0)}=1-2\theta
\mbox{ where
}\theta=\frac{\mbox{Vol}(\Omega,\theta_0)}{2\mbox{Vol}(S^{2n+1},\theta_0)}>0.$$
Thus choosing $\eta$ sufficiently small such that
$(1+\eta)(1-2\theta)<1-\theta$, together with (\ref{4.31}),
(\ref{4.32a}) and (\ref{4.33}), we conclude that
\begin{equation*}
\begin{split}
\int_{S^{2n+1}}v^{-2\delta}dV_{\theta_0}&\leq C+(1-\theta)\int_{S^{2n+1}}v^{-2\delta}dV_{\theta_0}
+\frac{n^2\delta^2}{4\lambda_1(1+2\delta)}\int_{S^{2n+1}}v^{-2\delta}dV_{\theta_0}\\
&\hspace{4mm}+C\|R_h-f_\phi\|_{L^{\frac{2n+2}{2-2n\delta}}(S^{2n+1},h)}.
\end{split}
\end{equation*}
By Lemma \ref{lem3.2} and taking $\delta>0$ sufficiently small, we have
\begin{equation}\label{4.32}
\int_{S^{2n+1}}v^{-2\delta}dV_{\theta_0}\leq C.
\end{equation}
Now let $G(Q,\cdot\,)$ be the Green's function for $-\Delta_{\theta_0}$
with singularity at $Q\in S^{2n+1}$. Then from \cite{Sanchez-Calle}, we have
\begin{equation}\label{4.35}
G(Q,\cdot\,)>0\mbox{ and }\|G(Q,\cdot\,)\|_{L^{p}(S^{2n+1},\theta_0)}\leq C(p)\mbox{ for all }p<2.
\end{equation}
Hence, for $\gamma>0$ sufficiently small so that
$2+\gamma<2+\displaystyle\frac{2}{n}$, we obtain by (\ref{4.22}) and
(\ref{4.35}) that
\begin{equation*}
\begin{split}
v^{-\gamma}(Q)&=\frac{1}{\mbox{Vol}(S^{2n+1},\theta_0)}\int_{S^{2n+1}}v^{-\gamma}
dV_{\theta_0}-\int_{S^{2n+1}}G(Q,\cdot\,)\Delta_{\theta_0}(v^{-\gamma})
dV_{\theta_0}\\
&=\frac{1}{\mbox{Vol}(S^{2n+1},\theta_0)}\int_{S^{2n+1}}v^{-\gamma}dV_{\theta_0}
-\frac{n}{2n+2}\int_{S^{2n+1}}G(Q,\cdot\,)\\
&\hspace{4mm}\Big(\gamma R_hv^{\frac{2}{n}-\gamma}-\gamma R_{\theta_0}v^{-\gamma}+
(2+\frac{2}{n})\gamma(1+\gamma)v^{-\gamma-2}|\nabla_{\theta_0}v|^2_{\theta_0}\Big)
dV_{\theta_0}\\
&\leq\frac{1}{\mbox{Vol}(S^{2n+1},\theta_0)}\int_{S^{2n+1}}v^{-\gamma}dV_{\theta_0}
+\frac{n}{2n+2}\gamma R_{\theta_0}\int_{S^{2n+1}}G(Q,\cdot\,)v^{-\gamma}dV_{\theta_0}\\
&\hspace{4mm}-\frac{n}{2n+2}\gamma\int_{S^{2n+1}}G(Q,\cdot\,)(R_h-\alpha f_{\phi})v^{\frac{2}{n}-\gamma}
dV_{\theta_0}\\
&\leq\frac{1}{\mbox{Vol}(S^{2n+1},\theta_0)}\int_{S^{2n+1}}v^{-\gamma}dV_{\theta_0}\\
&\hspace{4mm}+\frac{n}{2n+2}\gamma R_{\theta_0}C
\Big(\frac{n+1}{n}\Big)\|v^{-\gamma}\|_{L^{n+1}(S^{2n+1},\theta_0)}\\
&\hspace{4mm}+\frac{n}{2n+2}\gamma C\Big(\frac{2n+2}{2n-n\gamma}\Big)
\|R_h-\alpha f_{\phi}\|_{L^{\frac{2n+2}{2-n\gamma}}(S^{2n+1},h)}.
\end{split}
\end{equation*}
Finally, choosing $\gamma=\displaystyle\frac{2\delta}{n+1}$ with
$\delta<1$ and using (\ref{4.32}), (\ref{4.35}) and Lemma \ref{lem3.2}, we get the
desired lower bound $v\geq C_3>0$.
\end{proof}

\begin{lem}\label{lem4.11}
Suppose that case \emph{(i)} does not occur in Theorem
\ref{thm4.7}. For any sequence $\{t_k\}$, let $\lambda_1(\theta(t_k))$ be the
first non-zero eigenvalue of the sub-Laplacian with respect to the
contact form $\theta(t_k)=u^{\frac{2}{n}}(t_k)\theta_0$. Then there
exists a subsequence $\{t_j\}$ such that $\lambda_1(\theta(t_j))\geq
\beta_0>0$ for all $j$, where $\beta_0$ is a  positive constant independent of $j$.
\end{lem}
\begin{proof}
Since $\theta(t_k)=u^{\frac{2}{n}}(t_k)\theta_0$ is diffeomorphic to
the normalized contact form $h(t_k)=v_k^{\frac{2}{n}}\theta_0$, we only
need to get the first non-zero eigenvalue estimate for
$h(t_k)=v_k^{\frac{2}{n}}\theta_0$. By Lemma \ref{lem4.3}, for any
sequence $\{t_k\}$, we have
\begin{equation}\label{4.36}
\liminf_{k\rightarrow\infty}\left(\frac{1}{\mbox{Vol}(S^{2n+1},\theta_0)}\int_{B_r(Q)}
|R_{h(t_k)}|^{n+1}dV_{h(t_k)}\right)^{\frac{1}{n+1}}\geq\frac{n(n+1)}{2}
\end{equation}
for any $r>0$. Therefore for each $k$, there exists $r_k>0$ such
that
\begin{equation}\label{4.37}
\left(\frac{1}{\mbox{Vol}(S^{2n+1},\theta_0)}\int_{B_{r_k}(Q)}
|R_{h(t_k)}|^{n+1}dV_{h(t_k)}\right)^{\frac{1}{n+1}}=\frac{n(n+1)}{2},
\end{equation}
since $|R_{h(t_k)}|^{n+1}$ is integrable with respect to
$dV_{h(t_k)}$. It follows from (\ref{4.36}) and (\ref{4.37}) that
$r_k\rightarrow 0$ as $k\rightarrow\infty$.

Now, up to a rotation, we may assume $Q=S=(0,...,0,-1)$, the south pole of $S^{2n+1}$.  Let
$\Psi(z,\tau)=\pi^{-1}(z,\tau)$ for
$(z,\tau)\in\mathbb{H}^n\subset\mathbb{C}^n\times\mathbb{R}$ where
$\pi$ is defined as in (\ref{4.20}). Set
$$w_k(z,\tau)=|\det(d\Psi)|^{\frac{n}{2n+2}}(v_k\circ\Psi)\geq 0$$
to obtain a sequence $w_k:\mathbb{H}^n\rightarrow\mathbb{R}$ which
is in  $S^p_2(\mathbb{H}^n)$ locally and converges to a function
such that
$$(2+\frac{2}{n})\Delta_{\mathbb{H}^n}w_\infty+\alpha(\infty)f(Q)w_\infty^{\frac{n}{2n+2}}=0.$$
Here $\Delta_{\mathbb{H}^n}$ is the sub-Laplacian with respect to
the standard contact form
$\theta_{\mathbb{H}^n}=d\tau+i\sum_{j=1}^n(z_jd\overline{z}_j-\overline{z}_jdz_j)$
on $\mathbb{H}^n$ (see \cite{Dragomir}). By the classification
theorem of Jerison and Lee in \cite{Jerison&Lee1}, up to the scaling
with $\alpha(\infty)f(Q)=n(n+1)/2$,
$$w_\infty(z,\tau)=\lambda^{n}\omega
\circ D_{\lambda}\circ T_{(z',\tau')}(z,\tau)\mbox{ and
}\omega(z,\tau)=\left(\frac{4}{\tau^2+(1+|z|^2)^2}\right)^{\frac{n}{2}},$$
where $D_{\lambda},
T_{(z',\tau')}:\mathbb{H}^n\rightarrow\mathbb{H}^n$ are respectively
the dilation and translation on $\mathbb{H}^n$ defined as in
(\ref{4.21}). Then let us define
$$
\psi=\lambda^{n}\Psi\circ D_{\lambda}\circ T_{(z',\tau')}\circ\pi:
S^{2n+1}\setminus\{S\}\rightarrow S^{2n+1}\setminus\{S\}.
$$
It follows directly from the definition that
$$\widetilde{h}(t_k):=\psi^*\big(h(t_k)\big)\rightarrow\theta_0\hspace{2mm}\mbox{ as }k\rightarrow\infty$$
locally in $S^p_2(S^{2n+1}\setminus\{S\},\theta_0)$  for any
$p<\infty$. Denote by $\widetilde{w}_k^{\frac{2}{n}}\theta_0$ the
contact form $\widetilde{h}(t_k)=\psi^*\big(h(t_k)\big)$. Hence we
have
\begin{equation}\label{4.38}
\widetilde{w}_k\rightarrow 1\hspace{2mm}\mbox{ as
}k\rightarrow\infty\mbox{ locally in
}S^p_2(S^{2n+1}\setminus\{S\},\theta_0).
\end{equation}

Now for any given $0<\rho<1$, let $\zeta_\rho$ be a cut-off function
on $\mathbb{H}^n$ defined by
$$\zeta_\rho(z,\tau)=\log\log\frac{1}{\rho}-\log\log|(z,\tau)|_{\mathbb{H}^n}$$
for
$\rho^{-1/e}<|(z,\tau)|_{\mathbb{H}^n}=\sqrt[4]{|z|^4+\tau^2}<\rho^{-1}$;
and with $\zeta_\rho(z)=0$ in $\mathbb{H}^n\setminus
B_{\rho^{-1}}(0)$, $\zeta_\rho(z)=1$ in $ B_{\rho^{-1/e}}(0)$. Then
we have
\begin{equation}\label{4.39}
\begin{split}
&\int_{S^{2n+1}}|\nabla_{\theta_0}(\zeta_\rho\circ\pi)|_{\theta_0}^{2n+2}dV_{\theta_0}=
\int_{\mathbb{H}^n}|\nabla_{\theta_{\mathbb{H}^n}}\zeta_\rho|_{\theta_{\mathbb{H}^n}}^{2n+2}
dV_{\theta_{\mathbb{H}^n}}\\
&=C\iint_{\{\rho^{-1/e}<\sqrt[4]{|z|^4+\tau^2}<\rho^{-1}\}}
\left(\frac{|z|^3}{(|z|^4+\tau^2)\log(|z|^4+\tau^2)}\right)^{2n+2}dzd\tau\\
&\longrightarrow 0\hspace{2mm}\mbox{ as }\rho\rightarrow 0.
\end{split}
\end{equation}

For any $k$, let $\psi_k$ be the first eigenfunction of the
sub-Laplacian of $\widetilde{h}(t_k)$ with nonzero first eigenvalue
$\lambda_1(\widetilde{h}(t_k))$. We normalize the eigenfunction such
that
\begin{equation}\label{4.40}
\int_{S^{2n+1}}\psi_k dV_{\widetilde{h}(t_k)}=0\hspace{2mm}\mbox{
and }\hspace{2mm}\int_{S^{2n+1}}\psi_k^2 dV_{\widetilde{h}(t_k)}=1.
\end{equation}
We claim that there exists a subsequence $k$ such that
$\lambda_1(\widetilde{h}(t_k))\geq\displaystyle\frac{n}{4}$. Assume
the claim does not hold, we would have
$\lambda_1(\widetilde{h}(t_k))\leq\displaystyle\frac{n}{4}$ for all
sufficiently large $k$. This would imply
\begin{equation}\label{4.41}
\int_{S^{2n+1}}|\nabla_{\widetilde{h}(t_k)}\psi_k|^2_{\widetilde{h}(t_k)}dV_{\widetilde{h}(t_k)}
=\lambda_1(\widetilde{h}(t_k))\leq\frac{n}{4}
\end{equation}
by (\ref{4.38}).  Note that, by Lemma \ref{lem4.10},
$\widetilde{w}_k\geq c_0>0$ for some constant $c_0>0$ depending only
on $C_3$ and $\lambda$. Therefore for each fixed $\rho>0$
sufficiently small, by (\ref{4.40}) we have
\begin{equation*}
\begin{split}\int_{S^{2n+1}}(\psi_k\zeta_\rho\circ\pi)^2dV_{\theta_0}&\leq c_0^{-(2+\frac{2}{n})}
\int_{S^{2n+1}}(\psi_k\zeta_\rho\circ\pi)^2\widetilde{w}_k^{2+\frac{2}{n}}dV_{\theta_0}\\
&\leq c_0^{-(2+\frac{2}{n})}
\int_{S^{2n+1}}\psi_k^2dV_{\widetilde{h}(t_k)}=c_0^{-(2+\frac{2}{n})}
\end{split}
\end{equation*}
since $(\zeta_\rho\circ\pi)^2\leq 1$ for all $z$. We also have
\begin{equation*}
\begin{split}
&\int_{S^{2n+1}}|\nabla_{\theta_0}(\psi_k\zeta_\rho\circ\pi)|^2dV_{\theta_0}
\leq
c_0^{-2}\int_{S^{2n+1}}|\nabla_{\theta_0}(\psi_k\zeta_\rho\circ\pi)|_{\theta_0}^2\widetilde{w}_k^2dV_{\theta_0}\\
&=2c_0^{-2}\left(\int_{S^{2n+1}}|\nabla_{\theta_0}\psi_k|_{\theta_0}^2\widetilde{w}_k^2
(\zeta_\rho\circ\pi)^2dV_{\theta_0}
+\int_{S^{2n+1}}|\nabla_{\theta_0}(\zeta_\rho\circ\pi)|_{\theta_0}^2\widetilde{w}_k^2\psi_k^2dV_{\theta_0}\right).
\end{split}
\end{equation*}
Observe that the first term in  bracket of the right hand side is
bounded since $(\zeta_\rho\circ\pi)^2\leq 1$ and
\begin{equation}\label{4.41a}
\int_{S^{2n+1}}|\nabla_{\theta_0}\psi_k|_{\theta_0}^2\widetilde{w}_k^2dV_{\theta_0}=
\int_{S^{2n+1}}|\nabla_{\widetilde{h}(t_k)}\psi_k|^2_{\widetilde{h}(t_k)}dV_{\widetilde{h}(t_k)}
\leq\frac{n}{4}.
\end{equation} For the second term, we can apply H\"{o}lder's
inequality to get
\begin{equation*}
\begin{split}
&\int_{S^{2n+1}}|\nabla_{\theta_0}(\zeta_\rho\circ\pi)|_{\theta_0}^2\widetilde{w}_k^2\psi_k^2dV_{\theta_0}\\
&\leq
\left(\int_{S^{2n+1}}|\nabla_{\theta_0}(\zeta_\rho\circ\pi)|_{\theta_0}^{2n+2}
dV_{\theta_0}\right)^{\frac{1}{n+1}}
\left(\int_{S^{2n+1}}|\widetilde{w}_k\psi_k|^{2+\frac{2}{n}}
dV_{\theta_0}\right)^{\frac{n}{n+1}}.
\end{split}
\end{equation*}
The first factor is bounded thanks to (\ref{4.39}), while the second
factor can be estimated as follows:
\begin{equation}\label{4.42}
\begin{split}
&Y(S^{2n+1},\theta_0)\left(\int_{S^{2n+1}}|\widetilde{w}_k\psi_k|^{2+\frac{2}{n}}
dV_{\theta_0}\right)^{\frac{n}{n+1}}\\
&=Y(S^{2n+1},\theta_0)\|\widetilde{w}_k\psi_k\|^2_{L^{2+\frac{2}{n}}(S^{2n+1},\theta_0)}=
Y(S^{2n+1},\theta_0)\left(\int_{S^{2n+1}}|\psi_k|^{2+\frac{2}{n}}
dV_{\widetilde{h}(t_k)}\right)^{\frac{n}{n+1}}\\
&\leq(2+\frac{2}{n})\int_{S^{2n+1}}
|\nabla_{\widetilde{h}(t_k)}\psi_k|^2_{\widetilde{h}(t_k)}dV_{\widetilde{h}(t_k)}
+\int_{S^{2n+1}}
R_{\widetilde{h}(t_k)}\psi_k^2dV_{\widetilde{h}(t_k)}\\
&=(2+\frac{2}{n})\int_{S^{2n+1}}
|\nabla_{\widetilde{h}(t_k)}\psi_k|^2_{\widetilde{h}(t_k)}dV_{\widetilde{h}(t_k)}
+\alpha(t_k)\int_{S^{2n+1}}f\psi_k^2dV_{\widetilde{h}(t_k)}\\
&\hspace{4mm}+\int_{S^{2n+1}}
(R_{\widetilde{h}(t_k)}-\alpha(t_k)f)\psi_k^2dV_{\widetilde{h}(t_k)}
\\
&\leq(2+\frac{2}{n})\int_{S^{2n+1}}
|\nabla_{\widetilde{h}(t_k)}\psi_k|^2_{\widetilde{h}(t_k)}dV_{\widetilde{h}(t_k)}
+\alpha(t_k)\int_{S^{2n+1}}f\psi_k^2dV_{\widetilde{h}(t_k)}\\
&\hspace{4mm} +\left(\int_{S^{2n+1}}
|R_{\widetilde{h}(t_k)}-\alpha(t_k)f|^{n+1}dV_{\widetilde{h}(t_k)}\right)^{\frac{1}{n+1}}
\left(\int_{S^{2n+1}}|\psi_k|^{2+\frac{2}{n}}
dV_{\widetilde{h}(t_k)}\right)^{\frac{n}{n+1}}
\end{split}
\end{equation}
where the first inequality follows from Lemma \ref{lem2.3} and the
last inequality follows from H\"{o}lder's inequality. When $k$ is
sufficiently large, $$\left(\int_{S^{2n+1}}
|R_{\widetilde{h}(t_k)}-\alpha(t_k)f|^{n+1}dV_{\widetilde{h}(t_k)}\right)^{\frac{1}{n+1}}
\leq\frac{Y(S^{2n+1},\theta_0)}{2}$$ by Lemma \ref{lem3.2}. Combining this with
(\ref{4.42}), we deduce that
\begin{equation*}
\begin{split}
&\frac{Y(S^{2n+1},\theta_0)}{2}\left(\int_{S^{2n+1}}|\psi_k|^{2+\frac{2}{n}}
dV_{\widetilde{h}(t_k)}\right)^{\frac{n}{n+1}}\\
&\leq
(2+\frac{2}{n})\int_{S^{2n+1}}|\nabla_{\widetilde{h}(t_k)}\psi_k|^2_{\widetilde{h}(t_k)}dV_{\widetilde{h}(t_k)}+
\alpha(t_k)\int_{S^{2n+1}}f\psi_k^2dV_{\widetilde{h}(t_k)}\\
&\leq
(2+\frac{2}{n})\int_{S^{2n+1}}|\nabla_{\widetilde{h}(t_k)}\psi_k|^2_{\widetilde{h}(t_k)}dV_{\widetilde{h}(t_k)}
+\alpha_2\big(\max_{S^{2n+1}}f\big)\int_{S^{2n+1}}\psi_k^2dV_{\widetilde{h}(t_k)}\leq
C
\end{split}
\end{equation*}
by (\ref{4.40}), (\ref{4.41a}), and Lemma \ref{lem2.4}.

Therefore we have shown that $\{\psi_k\zeta_\rho\circ\pi\}$ is a
bound sequence in $S_1^2(S^{2n+1},\theta_0)$. Thus there exists a
subsequence, still denote by $\{\psi_k\zeta_\rho\circ\pi\}$, and a
function $\psi_\rho$ such that
$\psi_k\zeta_\rho\circ\pi\rightarrow\psi_\rho$ as
$k\rightarrow\infty$ weakly in $S^2_1(S^{2n+1},\theta_0)$ which is
also strongly convergent in $L^q(S^{2n+1},\theta_0)$ with
$q<\displaystyle\frac{n}{2n+2}$. Then  $\{\psi_\rho\}$ is bounded in
$S^2_1(S^{2n+1},\theta_0)$ since
$\displaystyle\|\psi_\rho\|_{S^2_1(S^{2n+1},\theta_0)}
\leq\lim_{k\rightarrow\infty}\|\psi_k\zeta_\rho\circ\pi\|_{S^2_1(S^{2n+1},\theta_0)}$.
Therefore there exists a subsequence $\{\psi_{\rho_j}\}$ and
$\psi_0\in S^2_1(S^{2n+1},\theta_0)$ such that
$\psi_{\rho_j}\rightarrow\psi_0$ as $j\rightarrow\infty$ weakly in
$S^2_1(S^{2n+1},\theta_0)$.

Now we claim that if we let $\pi=\Psi^{-1}$, then there hold
\begin{equation}\label{4.43}
\lim_{\rho\rightarrow
0}\lim_{k\rightarrow\infty}\|\psi_k\zeta_\rho\circ\pi-\psi_k\|_{L^2(S^{2n+1},\widetilde{h}(t_k))}=0,
\end{equation}
and
\begin{equation}\label{4.44}
\lim_{j\rightarrow\infty}\lim_{k\rightarrow\infty}
\big(\|\nabla_{\widetilde{h}(t_k)}(\psi_k\zeta_{\rho_j}\circ\pi)\|_{L^2(S^{2n+1},\widetilde{h}(t_k))}
-\|\nabla_{\widetilde{h}(t_k)}\psi_k\|_{L^2(S^{2n+1},\widetilde{h}(t_k))}\big)\leq
0.
\end{equation}
If we assume these, then we have
\begin{equation*}
\begin{split}
\int_{S^{2n+1}}\psi_0^2\,dV_{\theta_0}=\lim_{j\rightarrow\infty}\int_{S^{2n+1}}\psi_{\rho_j}^2dV_{\theta_0}
&=\lim_{j\rightarrow\infty}\lim_{k\rightarrow\infty}\int_{S^{2n+1}}(\psi_k\zeta_{\rho_j}\circ\pi)^2dV_{\theta_0}\\
&=\lim_{k\rightarrow\infty}\int_{S^{2n+1}}\psi_k^2\widetilde{w}_k^{2+\frac{2}{n}}dV_{\theta_0}=1
\end{split}
\end{equation*}
where we have used (\ref{4.38}), (\ref{4.40}) and (\ref{4.43}).
Similarly, we have
\begin{equation*}
\begin{split}
\int_{S^{2n+1}}\psi_0\,dV_{\theta_0}=\lim_{j\rightarrow\infty}\int_{S^{2n+1}}\psi_{\rho_j}dV_{\theta_0}
&=\lim_{j\rightarrow\infty}\lim_{k\rightarrow\infty}\int_{S^{2n+1}}\psi_k\zeta_{\rho_j}\circ\pi\,dV_{\theta_0}\\
&=\lim_{k\rightarrow\infty}\int_{S^{2n+1}}\psi_k\widetilde{w}_k^{2+\frac{2}{n}}dV_{\theta_0}=0.
\end{split}
\end{equation*}
Thus $\psi_0\not\equiv 0$ and Raleigh's inequality gives
\begin{equation}\label{4.45}
\|\nabla_{\theta_0}\psi_0\|_{L^2(S^{2n+1},\theta_0)}\geq\frac{n}{2}.
\end{equation}
Hence, we get the following contradiction:
\begin{equation*}
\begin{split}
\frac{n}{4}&\geq\lim_{k\rightarrow\infty}\lambda_1(\widetilde{h}(t_k))
=\lim_{k\rightarrow\infty}
\int_{S^{2n+1}}|\nabla_{\widetilde{h}(t_k)}\psi_k|^2_{\widetilde{h}(t_k)}dV_{\widetilde{h}(t_k)}\\
&\geq \lim_{j\rightarrow\infty}\lim_{k\rightarrow\infty}
\|\nabla_{\widetilde{h}(t_k)}(\psi_k\zeta_{\rho_j}\circ\pi)\|_{L^2(S^{2n+1},\widetilde{h}(t_k))}
=\lim_{j\rightarrow\infty}\lim_{k\rightarrow\infty}
\|\nabla_{\theta_0}(\psi_k\zeta_{\rho_j}\circ\pi)\|_{L^2(S^{2n+1},\theta_0)}\\
&\geq\lim_{j\rightarrow\infty}
\|\nabla_{\theta_0}\psi_{\rho_j}\|_{L^2(S^{2n+1},\theta_0)}\geq
\|\nabla_{\theta_0}\psi_0\|_{L^2(S^{2n+1},\theta_0)}\geq\frac{n}{2},
\end{split}
\end{equation*}
where the first inequality follows from (\ref{4.41}), the second
inequality follows from (\ref{4.44}), the second equality follows
from (\ref{4.38}), and the last inequality follows from
(\ref{4.45}).

Therefore it remains to show (\ref{4.43}) and (\ref{4.44}). By
H\"{o}lder's inequality and the fact that $\zeta_\rho\circ\pi\leq
1$, we have
\begin{equation*}
\begin{split}
&\int_{S^{2n+1}}|\psi_k(\zeta_\rho\circ\pi-1)|^2dV_{\widetilde{h}(t_k)}
\leq\int_{\{x\in
S^{2n+1}|\zeta_\rho\circ\pi<1\}}|\psi_k|^2dV_{\widetilde{h}(t_k)}\\
&\leq\mbox{Vol}\big(\{x\in
S^{2n+1}|\zeta_\rho\circ\pi<1\},\widetilde{h}(t_k)\big)^{\frac{1}{n+1}}
\left(\int_{S^{2n+1}}|\psi_k|^{2+\frac{2}{n}}dV_{\widetilde{h}(t_k)}\right)^{\frac{n}{n+1}},
\end{split}
\end{equation*}
where the first factor can be estimated by (\ref{4.38}):{\footnote {The
integral can be estimated as follows:
\begin{equation*}
\begin{split}
&\int_{\mathbb{H}^n\setminus
B_{\rho^{-1}}(0)}\frac{dzd\tau}{(\tau^2+(1+|z|^2)^2)^{n+1}}
=\int_{\{\sqrt[4]{\tau^2+|z|^4}\geq\rho^{-1}\}}\frac{dzd\tau}{(\tau^2+(1+|z|^2)^2)^{n+1}}\\
&\leq
2\int_{\{|z|\geq\rho^{-1}\}}\left(\int_{\sqrt{\rho^{-4}-|z|^4}}^\infty\frac{d\tau}{1+\tau^2}\right)
\frac{dz}{(1+|z|^2)^{2n}}=2\int_{\{|z|\geq\rho^{-1}\}}\Big[\tan^{-1}(\tau)\Big]_{0}^\infty
\frac{dz}{(1+|z|^2)^{2n}}\\
&\leq\pi\int_{\{|z|\geq\rho^{-1}\}}\frac{dz}{(1+|z|^2)^{2n}}
=C\int_{\rho^{-1}}^\infty\frac{r^{2n-1}dr}{(1+r^2)^{2n}}\leq
C\int_{\rho^{-1}}^\infty\frac{dr}{r^{2n+1}}=O(\rho^{2n}).
\end{split}
\end{equation*}
}}
\begin{equation*}
\begin{split}
\mbox{Vol}\big(\{x\in
S^{2n+1}|\zeta_\rho\circ\pi<1\},\widetilde{h}(t_k)\big)
&=\mbox{Vol}\big(\{x\in
S^{2n+1}|\zeta_\rho\circ\pi<1\},\widetilde{w}(t_k)^{\frac{2}{n}}\theta_0\big)\\
&\xrightarrow{k\rightarrow\infty} \mbox{Vol}\big(\{x\in
S^{2n+1}|\zeta_\rho\circ\pi<1\},\theta_0\big)\\
&=\int_{\mathbb{H}^n\setminus
B_{\rho^{-1/e}}(0)}\left(\frac{4}{(1+|z|^2)^2+\tau^2}\right)^{n+1}dV_{\theta_{\mathbb{H}^n}}
\xrightarrow{\rho\rightarrow 0} 0
\end{split}
\end{equation*}
and the second factor can be
bounded as follows:
\begin{equation}\label{4.46}
\begin{split}
\left(\int_{S^{2n+1}}|\psi_k|^{2+\frac{2}{n}}dV_{\widetilde{h}(t_k)}\right)^{\frac{n}{n+1}}&\leq
C\left(\int_{S^{2n+1}}|\psi_k|^{2+\frac{2}{n}}dV_{\theta_0}\right)^{\frac{n}{n+1}}\\
&\leq C\|\psi_k\|_{S_1^2(S^{2n+1},\theta_0)}\leq
C\|\psi_k\|_{S_1^2(S^{2n+1},\widetilde{h}(t_k))}\leq C
\end{split}
\end{equation}
by (\ref{4.38}), (\ref{4.40}), (\ref{4.41}) and Lemma \ref{lem2.3}.
This proves (\ref{4.43}). To prove (\ref{4.44}), we apply
H\"{o}lder's inequality to get
\begin{equation}\label{4.47}
\begin{split}
&\|\psi_k\nabla_{\widetilde{h}(t_k)}(\zeta_{\rho}\circ\pi)\|_{L^2(S^{2n+1},\widetilde{h}(t_k))}^2\\
&=\int_{S^{2n+1}}|\nabla_{\widetilde{h}(t_k)}(\zeta_{\rho}\circ\pi)|^2_{\widetilde{h}(t_k)}
\psi_k^2\,dV_{\widetilde{h}(t_k)}\\
&\leq\left(\int_{S^{2n+1}}|\nabla_{\widetilde{h}(t_k)}(\zeta_{\rho}\circ\pi)|^{2n+2}_{\widetilde{h}(t_k)}
dV_{\widetilde{h}(t_k)}\right)^{\frac{1}{n+1}}
\left(\int_{S^{2n+1}}\psi_k^{2+\frac{2}{n}}dV_{\widetilde{h}(t_k)}\right)^{\frac{n}{n+1}},
\end{split}
\end{equation}
where the first factor can be estimated as follows:
\begin{equation}\label{4.48}
\begin{split}
&\int_{S^{2n+1}}|\nabla_{\widetilde{h}(t_k)}(\zeta_{\rho}\circ\pi)|^{2n+2}_{\widetilde{h}(t_k)}
dV_{\widetilde{h}(t_k)}
=\int_{S^{2n+1}}|\nabla_{\theta_0}(\zeta_{\rho}\circ\pi)|^{2n+2}_{\theta_0}
dV_{\theta_0}\xrightarrow{\rho\rightarrow 0} 0
\end{split}
\end{equation}
by (\ref{4.39}), and the second factor is bounded by (\ref{4.46}).
Hence,  as $\rho\rightarrow 0$, we have
\begin{equation*}
\begin{split}
&\|\nabla_{\widetilde{h}(t_k)}(\psi_k\zeta_{\rho}\circ\pi)\|_{L^2(S^{2n+1},\widetilde{h}(t_k))}\\
&\leq\|(\zeta_{\rho}\circ\pi)\nabla_{\widetilde{h}(t_k)}\psi_k\|_{L^2(S^{2n+1},\widetilde{h}(t_k))}+
\|\psi_k\nabla_{\widetilde{h}(t_k)}(\zeta_{\rho}\circ\pi)\|_{L^2(S^{2n+1},\widetilde{h}(t_k))}\\
&\leq
\|\nabla_{\widetilde{h}(t_k)}\psi_k\|_{L^2(S^{2n+1},\widetilde{h}(t_k))}+o(1),
\end{split}
\end{equation*}
by (\ref{4.42}), (\ref{4.48}) and the fact that
$\zeta_{\rho}\circ\pi\leq 1$. This proves (\ref{4.44}).
\end{proof}

Next we have the following estimate, which is used to control the
upper bound of the normalized conformal factor $v$.

\begin{lem}\label{lem4.12}
On $(S^{2n+1},\theta_0)$, if $h=v^{\frac{2}{n}}\theta_0$ with $v>0$
is a contact form satisfying\\
\emph{(i)}
$\displaystyle\int_{S^{2n+1}}v^{2+\frac{2}{n}}dV_{\theta_0}=\int_{S^{2n+1}}u_0^{2+\frac{2}{n}}dV_{\theta_0}
=\mbox{\emph{Vol}}(S^{2n+1},\theta_0)$,\\
\emph{(ii)}
$\displaystyle\int_{S^{2n+1}}|R_h|^pv^{2+\frac{2}{n}}dV_{\theta_0}\leq
a_p\mbox{\emph{Vol}}(S^{2n+1},\theta_0)$ for some $p>n+1$,\\
\emph{(iii)} $0<\Lambda\leq\lambda_1(h)$, where $\lambda_1(h)$ is
the first nonzero eigenvalue of the sub-Laplacian $-\Delta_h$, where
$a_p,
\Lambda$ are positive constants and\\
\emph{(iv)} assuming in addition the condition: There exists some
positive constants $\sigma_0, l_0$ such that
\begin{equation}\label{4.52}
\int_{\{x\in S^{2n+1}|v(x)\geq\sigma_0\}}dV_{\theta_0}\geq l_0.
\end{equation}
Then there exists $\epsilon_0>0$ and a constant $C_0$ depending only
on $n, p, a_p, \Lambda, \sigma_0, l_0$ with
$$\int_{S^{2n+1}}v^{2+\frac{2}{n}+\epsilon_0}dV_{\theta_0}\leq
C_0.$$
\end{lem}
\begin{proof}
Since
$$-(2+\frac{2}{n})\Delta_{\theta_0}v+R_{\theta_0}v=R_hv^{1+\frac{2}{n}}\mbox{ on }S^{2n+1},$$
we multiply it by $v^\beta$ for some $\beta>0$ and integrate it to
get
$$(2+\frac{2}{n})\beta\int_{S^{2n+1}}v^{\beta-1}|\nabla_{\theta_0}v|_{\theta_0}^2dV_{\theta_0}
+R_{\theta_0}\int_{S^{2n+1}}v^{\beta+1}dV_{\theta_0}=\int_{S^{2n+1}}R_hv^{1+\frac{2}{n}+\beta}dV_{\theta_0}.$$
Set $w=v^{\frac{1+\beta}{2}}$, then
\begin{equation}\label{4.53}
(2+\frac{2}{n})\int_{S^{2n+1}}|\nabla_{\theta_0}w|_{\theta_0}^2dV_{\theta_0}=-\frac{(\beta+1)^2}{4\beta}R_{\theta_0}\int_{S^{2n+1}}w^2dV_{\theta_0}
+\frac{(\beta+1)^2}{4\beta}\int_{S^{2n+1}}R_hv^{\frac{2}{n}}w^2dV_{\theta_0}.
\end{equation}
By Lemma \ref{lem2.3}, we have
\begin{equation}\label{4.54}
\begin{split}
&\frac{n(n+1)}{2}\mbox{Vol}(S^{2n+1},\theta_0)^{\frac{1}{n+1}}\left(\int_{S^{2n+1}}w^{2+\frac{2}{n}}dV_{\theta_0}\right)^{\frac{n}{n+1}}\\
&\leq\int_{S^{2n+1}}\left((2+\frac{2}{n})|\nabla_{\theta_0}w|^2_{\theta_0}+R_{\theta_0}w^2\right)dV_{\theta_0}\\
&=\Big(1-\frac{(\beta+1)^2}{4\beta}\Big)R_{\theta_0}\int_{S^{2n+1}}w^2dV_{\theta_0}
+\frac{(\beta+1)^2}{4\beta}\int_{S^{2n+1}}R_hv^{\frac{2}{n}}w^2dV_{\theta_0}
\end{split}
\end{equation}
where we have used (\ref{4.53}). For a sufficiently large number
$b>0$, we get
\begin{equation}\label{4.55}
b^p\int_{\{|R_h|>b\}}v^{2+\frac{2}{n}}dV_{\theta_0}\leq
\int_{S^{2n+1}}|R_h|^pv^{2+\frac{2}{n}}dV_{\theta_0} \leq
a_p\mbox{Vol}(S^{2n+1},\theta_0)
\end{equation}
by assumption (ii). By H\"{o}lder's inequality, we have
\begin{equation}\label{4.56}
\begin{split}
&\int_{\{|R_h|>b\}}|R_h|^pv^{\frac{2}{n}}w^2dV_{\theta_0}\\
&\leq\left(\int_{S^{2n+1}}|R_h|^pv^{2+\frac{2}{n}}dV_{\theta_0}\right)^{\frac{1}{p}}
\left(\int_{\{|R_h|>b\}}v^{2+\frac{2}{n}}dV_{\theta_0}\right)^{\frac{1}{n+1}-\frac{1}{p}}
\left(\int_{S^{2n+1}}w^{2+\frac{2}{n}}dV_{\theta_0}\right)^{\frac{n}{n+1}}\\
&\leq
a_p^{\frac{1}{p}}\left(\frac{a_p}{b^p}\right)^{\frac{1}{n+1}-\frac{1}{p}}
\mbox{Vol}(S^{2n+1},\theta_0)^{\frac{1}{n+1}}
\left(\int_{S^{2n+1}}w^{2+\frac{2}{n}}dV_{\theta_0}\right)^{\frac{n}{n+1}}
\end{split}
\end{equation}
where we have used (\ref{4.55}) and the assumption (ii).

On the other hand, for any $\psi\in S^2_1(S^{2n+1},h)$, Raleigh's
inequality gives
$$\int_{S^{2n+1}}\psi^2\,dV_{h}\leq\frac{1}{\mbox{Vol}(S^{2n+1},\theta_0)}
\left(\int_{S^{2n+1}}\psi\,dV_{h}\right)^2
+\lambda_1(h)^{-1}\int_{S^{2n+1}}|\nabla_{h}\psi|_{h}^2dV_{h}.$$ Now
let $\psi=v^\epsilon$ with $0<\epsilon<1$, then
\begin{equation}\label{4.57}
\int_{S^{2n+1}}v^{2+\frac{2}{n}+2\epsilon}dV_{\theta_0}\leq
\frac{1}{\mbox{Vol}(S^{2n+1},\theta_0)}\left(\int_{S^{2n+1}}v^{2+\frac{2}{n}+\epsilon}dV_{\theta_0}\right)^2
+\Lambda^{-1}\int_{S^{2n+1}}v^2|\nabla_{\theta_0}v^\epsilon|_{\theta_0}^2dV_{\theta_0}
\end{equation}
by assumption (iii). Notice that condition (\ref{4.52}) implies that
there exist $\sigma_0>0$ and $l_0>0$ such that
\begin{equation}\label{4.58}
\int_{E_{\sigma_0}}dV_{\theta_0}\geq l_0\mbox{ i.e.
}\mbox{Vol}(E_{\sigma_0},\theta_0)\geq l_0,
\end{equation}
where $E_{\sigma_0}=\{x\in S^{2n+1}|v(x)\geq\sigma_0\}$. Observe
that
\begin{equation*}
\begin{split}
&\int_{S^{2n+1}}v^{2+\frac{2}{n}+\epsilon}dV_{\theta_0}=\int_{E_{\sigma_0}}v^{2+\frac{2}{n}+\epsilon}dV_{\theta_0}
+\int_{E_{\sigma_0}^c}v^{2+\frac{2}{n}+\epsilon}dV_{\theta_0}\\
&=\int_{E_{\sigma_0}}(v^{2+\frac{2}{n}}-\sigma_0^{2+\frac{2}{n}})v^\epsilon
dV_{\theta_0}+\sigma_0^{2+\frac{2}{n}}\int_{E_{\sigma_0}}v^\epsilon
dV_{\theta_0}
+\int_{E_{\sigma_0}^c}v^{2+\frac{2}{n}+\epsilon}dV_{\theta_0}\\
&\leq\left(\int_{E_{\sigma_0}}(v^{2+\frac{2}{n}}-\sigma_0^{2+\frac{2}{n}})v^{2\epsilon}
dV_{\theta_0}\right)^{\frac{1}{2}}\left(\int_{E_{\sigma_0}}(v^{2+\frac{2}{n}}-\sigma_0^{2+\frac{2}{n}})
dV_{\theta_0}\right)^{\frac{1}{2}}\\
&\hspace{4mm}+\sigma_0^{2+\frac{2}{n}}\left(\int_{E_{\sigma_0}}v^{2+\frac{2}{n}}
dV_{\theta_0}\right)^{\frac{n\epsilon}{2n+2}}\mbox{Vol}(S^{2n+1},\theta_0)^{1-\frac{n\epsilon}{2n+2}}
+\sigma_0^{2+\frac{2}{n}+\epsilon}\int_{E_{\sigma_0}^c}dV_{\theta_0}\\
&\leq\left(\int_{E_{\sigma_0}}v^{2+\frac{2}{n}+2\epsilon}
dV_{\theta_0}\right)^{\frac{1}{2}}\left(\int_{E_{\sigma_0}}(v^{2+\frac{2}{n}}-\sigma_0^{2+\frac{2}{n}})
dV_{\theta_0}\right)^{\frac{1}{2}}+C(\sigma_0)
\end{split}
\end{equation*}
where we have used assumption (i). This together with Young's
inequality implies that for any $0<\eta<1$
\begin{equation}\label{4.59}
\begin{split}
& \left(\int_{S^{2n+1}}v^{2+\frac{2}{n}+\epsilon}
dV_{\theta_0}\right)^{2}\\
&\leq (1+\eta)\left(\int_{E_{\sigma_0}}v^{2+\frac{2}{n}+2\epsilon}
dV_{\theta_0}\right)\left(\int_{E_{\sigma_0}}(v^{2+\frac{2}{n}}-\sigma_0^{2+\frac{2}{n}})
dV_{\theta_0}\right)+(1+\eta^{-1})C(\sigma_0)^2\\
&\leq (1+\eta)\left(\int_{E_{\sigma_0}}v^{2+\frac{2}{n}+2\epsilon}
dV_{\theta_0}\right)\left(\mbox{Vol}(S^{2n+1},\theta_0)
-\sigma_0^{2+\frac{2}{n}}\mbox{Vol}(E_{\sigma_0},\theta_0)\right)+(1+\eta^{-1})C(\sigma_0)^2
\end{split}
\end{equation}
by assumption (i). By choosing $\sigma_0$ small enough, we can
assume
$\sigma_0^{2+\frac{2}{n}}\mbox{Vol}(E_{\sigma_0},\theta_0)<\mbox{Vol}(S^{2n+1},\theta_0)$.
Since
$\sigma_0^{2+\frac{2}{n}}\mbox{Vol}(E_{\sigma_0},\theta_0)\geq\sigma_0^{2+\frac{2}{n}}
l_0>0$, then we may choose sufficiently small $\eta>0$ such that
\begin{equation}\label{4.60}
(1+\eta)\left(\mbox{Vol}(S^{2n+1},\theta_0)-\sigma_0^{2+\frac{2}{n}}l_0\right)\leq
(1-\delta)\mbox{Vol}(S^{2n+1},\theta_0)
\end{equation}
for some positive constant $\delta=\delta(\sigma_0,l_0)$. Then
\begin{equation*}
\begin{split}
\int_{S^{2n+1}}v^{2+\frac{2}{n}+2\epsilon}dV_{\theta_0}
&\leq
\frac{1}{\mbox{Vol}(S^{2n+1},\theta_0)}\left(\int_{S^{2n+1}}v^{2+\frac{2}{n}+\epsilon}dV_{\theta_0}\right)^2
+\Lambda^{-1}\int_{S^{2n+1}}v^2|\nabla_{\theta_0}v^\epsilon|_{\theta_0}^2dV_{\theta_0}\\
&\leq\frac{(1+\eta)\Big(\mbox{Vol}(S^{2n+1},\theta_0)
-\sigma_0^{2+\frac{2}{n}}\mbox{Vol}(E_{\sigma_0},\theta_0)\Big)}{\mbox{Vol}(S^{2n+1},\theta_0)}
\int_{S^{2n+1}}v^{2+\frac{2}{n}+2\epsilon}dV_{\theta_0}
\\
&\hspace{4mm}+\frac{(1+\eta^{-1})C(\sigma_0)^2}{\mbox{Vol}(S^{2n+1},\theta_0)}
+\Lambda^{-1}\int_{S^{2n+1}}v^2|\nabla_{\theta_0}v^\epsilon|_{\theta_0}^2dV_{\theta_0}\\
&\leq
(1-\delta)\int_{S^{2n+1}}v^{2+\frac{2}{n}+2\epsilon}dV_{\theta_0}
+\frac{(1+\eta^{-1})C(\sigma_0)^2}{\mbox{Vol}(S^{2n+1},\theta_0)}
+\Lambda^{-1}\int_{S^{2n+1}}v^2|\nabla_{\theta_0}v^\epsilon|_{\theta_0}^2dV_{\theta_0}
\end{split}
\end{equation*}
where we have used (\ref{4.57}) in the first inequality,
(\ref{4.59}) in the second inequality, and (\ref{4.60}) in the last
inequality. This implies that
\begin{equation}\label{4.61}
\delta\int_{S^{2n+1}}v^{2+\frac{2}{n}+2\epsilon}dV_{\theta_0}\leq
C(\sigma_0,l_0)+\Lambda^{-1}\int_{S^{2n+1}}v^2|\nabla_{\theta_0}v^\epsilon|_{\theta_0}^2dV_{\theta_0}.
\end{equation}

Choose $\beta=1+2\epsilon$, then
$w=v^{\frac{1+\beta}{2}}=v^{1+\epsilon}$. Then by (\ref{4.61})
\begin{equation}\label{4.62}
\begin{split}
\int_{S^{2n+1}}v^{2+\frac{2}{n}+2\epsilon}dV_{\theta_0}&\leq
C\delta^{-1}+(\delta\Lambda)^{-1}\int_{S^{2n+1}}v^2|\nabla_{\theta_0}v^\epsilon|_{\theta_0}^2dV_{\theta_0}\\
&=C\delta^{-1}+(\delta\Lambda)^{-1}\frac{\epsilon^2}{(1+\epsilon)^2}
\int_{S^{2n+1}}|\nabla_{\theta_0}w|_{\theta_0}^2dV_{\theta_0}\\
&\leq
C\delta^{-1}+(\delta\Lambda)^{-1}(2+\frac{2}{n})^{-1}\frac{\epsilon^2}{1+2\epsilon}
\int_{S^{2n+1}}R_hv^{\frac{2}{n}}w^2dV_{\theta_0}
\end{split}
\end{equation}
where the last inequality follows from (\ref{4.53}) with
$\beta=1+2\epsilon$. If we set
$I=\displaystyle\int_{S^{2n+1}}R_hv^{\frac{2}{n}}w^2dV_{\theta_0}$,
then
\begin{equation*}
\begin{split}
I&=\int_{\{|R_h|>b\}}|R_h|v^{\frac{2}{n}}w^2dV_{\theta_0}+\int_{\{|R_h|\leq
b\}}|R_h|v^{\frac{2}{n}}w^2dV_{\theta_0}\\
&\leq
a_p^{\frac{1}{p}}\left(\frac{a_p}{b^p}\right)^{\frac{1}{n+1}-\frac{1}{p}}
\mbox{Vol}(S^{2n+1},\theta_0)^{\frac{1}{n+1}}
\left(\int_{S^{2n+1}}w^{2+\frac{2}{n}}dV_{\theta_0}\right)^{\frac{n}{n+1}}
+b\int_{S^{2n+1}}v^{\frac{2}{n}}w^2dV_{\theta_0}\\
&\leq
a_p^{\frac{1}{p}}\left(\frac{a_p}{b^p}\right)^{\frac{1}{n+1}-\frac{1}{p}}
\mbox{Vol}(S^{2n+1},\theta_0)^{\frac{1}{n+1}}
\left(\int_{S^{2n+1}}w^{2+\frac{2}{n}}dV_{\theta_0}\right)^{\frac{n}{n+1}}\\
&\hspace{4mm}+bC\delta^{-1}+b(\delta\Lambda)^{-1}(2+\frac{2}{n})^{-1}\frac{\epsilon^2}{1+2\epsilon}
I
\end{split}
\end{equation*}
where the first inequality follows from (\ref{4.56}), and the last
inequality follows from (\ref{4.62}). This implies that
\begin{equation*}
\begin{split}
&\left[1-b(\delta\Lambda)^{-1}(2+\frac{2}{n})^{-1}\frac{\epsilon^2}{1+2\epsilon}\right]I\\
&\leq
a_p^{\frac{1}{p}}\left(\frac{a_p}{b^p}\right)^{\frac{1}{n+1}-\frac{1}{p}}
\mbox{Vol}(S^{2n+1},\theta_0)^{\frac{1}{n+1}}
\left(\int_{S^{2n+1}}w^{2+\frac{2}{n}}dV_{\theta_0}\right)^{\frac{n}{n+1}}+bC\delta^{-1}.
\end{split}
\end{equation*}
First, choose $b$ large enough such that
$$a_p^{\frac{1}{p}}\left(\frac{a_p}{b^p}\right)^{\frac{1}{n+1}-\frac{1}{p}}
<\frac{n(n+1)}{16}.$$
Next, choose $\epsilon>0$ small enough such that
$$1-b(\delta\Lambda)^{-1}(2+\frac{2}{n})^{-1}\frac{\epsilon^2}{1+2\epsilon}\geq\frac{1}{2}.$$
Thus
\begin{equation}\label{4.63}
I\leq \frac{n(n+1)}{8}\mbox{Vol}(S^{2n+1},\theta_0)^{\frac{1}{n+1}}
\left(\int_{S^{2n+1}}w^{2+\frac{2}{n}}dV_{\theta_0}\right)^{\frac{n}{n+1}}+C.
\end{equation}
Now by our choice of $\beta=1+2\epsilon$ with $\epsilon<1$, we
conclude that $\displaystyle\frac{(1+\beta)^2}{4\beta}\leq 2$. This
together with (\ref{4.54}) and (\ref{4.63}) implies that
\begin{equation*}
\begin{split}
&\frac{n(n+1)}{2}\mbox{Vol}(S^{2n+1},\theta_0)^{\frac{1}{n+1}}\left(\int_{S^{2n+1}}w^{2+\frac{2}{n}}dV_{\theta_0}\right)^{\frac{n}{n+1}}
\leq
C\int_{S^{2n+1}}w^2dV_{\theta_0} +2I\\
&\leq C\int_{S^{2n+1}}w^2dV_{\theta_0} +\frac{n(n+1)}{4}\mbox{Vol}(S^{2n+1},\theta_0)^{\frac{1}{n+1}}
\left(\int_{S^{2n+1}}w^{2+\frac{2}{n}}dV_{\theta_0}\right)^{\frac{n}{n+1}}+C,
\end{split}
\end{equation*}
which gives
\begin{equation}\label{4.64}
\frac{n(n+1)}{4}\mbox{Vol}(S^{2n+1},\theta_0)^{\frac{1}{n+1}}
\left(\int_{S^{2n+1}}w^{2+\frac{2}{n}}dV_{\theta_0}\right)^{\frac{n}{n+1}}\leq
C\int_{S^{2n+1}}w^2dV_{\theta_0}+C.
\end{equation}
On the other hand, if $\epsilon\leq\frac{1}{n}$, then
\begin{equation}\label{4.64a}
\begin{split}
\int_{S^{2n+1}}w^2dV_{\theta_0}&=\int_{S^{2n+1}}v^{2+2\epsilon}dV_{\theta_0}\\
&\leq\left(\int_{S^{2n+1}}v^{2+\frac{2}{n}}dV_{\theta_0}\right)^{\frac{n(1+\epsilon)}{1+n}}
\mbox{Vol}(S^{2n+1},\theta_0)^{1-\frac{n(1+\epsilon)}{1+n}}\leq C
\end{split}
\end{equation}
by H\"{o}lder's inequality and assumption (i). It follows from (\ref{4.64}) and (\ref{4.64a}) that
\begin{equation*}
\left(\int_{S^{2n+1}}v^{2+\frac{2}{n}+\epsilon_0}dV_{\theta_0}\right)^{\frac{n}{n+1}}
=\left(\int_{S^{2n+1}}w^{2+\frac{2}{n}}dV_{\theta_0}\right)^{\frac{n}{n+1}}\leq
C
\end{equation*}
where $\epsilon_0=\displaystyle(2+\frac{2}{n})\epsilon>0$. This
proves the assertion.
\end{proof}

\begin{lem}\label{lem4.13}
Let $\{\theta_k=\theta(t_k)\}$ be any sequence for the flow with
initial data $u_0\in C^\infty_f$. If $f$ cannot be realized as a
Wester scalar curvature in its conformal class, then there exists a
subsequence $\{t_j\}$ such that
$$v(t_j)\rightarrow 1\hspace{2mm}\mbox{ in
}C^{1,\lambda}_P(S^{2n+1})$$ for some $\lambda\in (0,1)$. Here
$C^{1,\lambda}_P(S^{2n+1})$ is the parabolic H\"{o}rmander
H\"{o}lder spaces defined as in $(\ref{2.24})$.
\end{lem}
\begin{proof}
For simplicity, set $\phi_k=\phi(t_k)$ to be the CR diffeomorphism
such that the normalized contact form $h_k$ is given by
$h_k=h(t_k)=\phi_k^*\big(\theta(t_k)\big)=v(t_k)^{\frac{2}{n}}\theta_0$.
By Lemma \ref{lem2.4} and Lemma \ref{lem3.2}, we have
\begin{equation}\label{4.65}
\begin{split}
\|R_{h_k}\|_{L^p(S^{2n+1},h_k)}&\leq
\|\alpha(t_k)f\circ\phi_k-R_{h_k}\|_{L^p(S^{2n+1},h_k)}+\|\alpha(t_k)f\circ\phi_k\|_{L^p(S^{2n+1},h_k)}\\
&=\|\alpha(t_k)f-R_{\theta_k}\|_{L^p(S^{2n+1},\theta_k)}+\|\alpha(t_k)f\|_{L^p(S^{2n+1},\theta_k)}\leq
C
\end{split}
\end{equation}
for any $p\geq 1$. Note that condition (i) in Lemma \ref{lem4.12} is
satisfied because of (\ref{4.30}) and Proposition \ref{prop2.1}, condition (ii) is
fulfilled by (\ref{4.65}), and condition (iii) is satisfied thanks
to Lemma \ref{lem4.11}. Finally by Lemma \ref{lem4.10}, if we choose
$\sigma_0=C_3$ and $l_0=\mbox{Vol}(S^{2n+1},\theta_0)$, condition
(iv) in Lemma \ref{lem4.12} is fulfilled. Hence, we can apply Lemma
\ref{lem4.12} to $v(t_k)$ to show that there exists $\epsilon_0>0$
and $C_0>0$ such that
\begin{equation}\label{4.66}
\int_{S^{2n+1}}v(t_k)^{2+\frac{2}{n}+\epsilon_0}dV_{\theta_0}\leq
C_0.
\end{equation}

For simplicity, set $v(t_k)=v_k$. Starting with
$q_0=\displaystyle2+\frac{2}{n}+\epsilon_0$,
$p_0=\displaystyle\frac{q_0}{\frac{1}{2n+2}q_0+\frac{n+1}{n}}>\frac{2n+2}{n+2}$,
$r_0=\displaystyle\frac{q_0-(2+\frac{2}{n})}{(1+\frac{2}{n})p_0-(2+\frac{2}{n})}>1$,
we have
\begin{equation*}
\begin{split}
&\int_{S^{2n+1}}|R_{h_k}v_k^{1+\frac{2}{n}}|^{p_0}dV_{\theta_0}
=\int_{S^{2n+1}}|R_{h_k}|^{p_0}v_k^{(1+\frac{2}{n})p_0}dV_{\theta_0}\\
&\leq
\left(\int_{S^{2n+1}}|R_{h_k}|^{\frac{p_0r_0}{r_0-1}}v_k^{2+\frac{2}{n}}dV_{\theta_0}\right)^{\frac{r_0-1}{r_0}}
\left(\int_{S^{2n+1}}v_k^{q_0}dV_{\theta_0}\right)^{\frac{1}{r_0}}\leq
C
\end{split}
\end{equation*}
by H\"{o}lder's inequality, (\ref{4.65}) and (\ref{4.66}). Hence,
\begin{equation}\label{4.67}
(2+\frac{2}{n})\Delta_{\theta_0}v=R_{\theta_0}v-R_hv^{1+\frac{2}{n}}\in
L^{p_0}(S^{2n+1},\theta_0).
\end{equation}
Hence, $v_k\in S^{p_0}_2(S^{2n+1},\theta_0)$ by Theorem 3.16 in
\cite{Dragomir} (see also \cite{Folland&Stein}). By Folland-Stein
embedding theorem (see \cite{Folland&Stein} or Theorem 3.13 in
\cite{Dragomir}), we obtain
$$v_k\in S^{p_0}_2(S^{2n+1},\theta_0)\hookrightarrow
L^{\frac{p_0(n+1)}{n+1-p_0}}(S^{2n+1},\theta_0)=L^{q_1}(S^{2n+1},\theta_0)$$
with $q_1=\displaystyle\frac{p_0(n+1)}{n+1-p_0}$. Then we set
$p_1=\displaystyle\frac{q_1}{\frac{1}{2n+2}q_1+\frac{n+1}{n}}$,
$r_1=\displaystyle\frac{q_1-(2+\frac{2}{n})}{(1+\frac{2}{n})p_1-(2+\frac{2}{n})}$.
In general, if we know $q_{l-1}$, $p_{l-1}$, $r_{l-1}$, we can
define inductively
\begin{equation}\label{4.68}
\begin{split}
&q_l=\frac{p_{l-1}(n+1)}{n+1-p_{l-1}}=\frac{2n(n+1)q_{l-1}}{2(n+1)^2-nq_{l-1}}
,\hspace{2mm}
p_l=\frac{q_l}{\frac{1}{2n+2}q_l+\frac{n+1}{n}},\\
&
r_l=\frac{q_l-(2+\frac{2}{n})}{(1+\frac{2}{n})p_l-(2+\frac{2}{n})}>1,
\end{split}
\end{equation}
where $l\in\mathbb{Z}^+$.  Note  that if $\displaystyle
2+\frac{2}{n}<q_l<\frac{2(n+1)^2}{n}$, then by (\ref{4.68})
\begin{equation*}
q_{l+1}-q_{l}=\left(\frac{q_l-(2+\frac{2}{n})}{\frac{2(n+1)^2}{n}-q_l}\right)q_l
\geq\left(\frac{\epsilon_0}{n(2+\frac{2}{n})-\epsilon_0}\right)q_l
\end{equation*}
since $q_l\geq q_0\geq\displaystyle 2+\frac{2}{n}+\epsilon_0$. Thus
there exists $l_0\in\mathbb{N}$ with
$q_{l_0}>\displaystyle\frac{2(n+1)^2}{n}$ and
$q_0<q_1<\cdots<q_{l_0-1}<\displaystyle\frac{2(n+1)^2}{n}<q_{l_0}$
such that
$$p_{l_0}=\frac{q_{l_0}}{\frac{1}{2n+2}q_{l_0}+\frac{n+1}{n}}>n+1.$$
Therefore  by using the similar argument to get (\ref{4.67}), we can
prove that
$$
(2+\frac{2}{n})\Delta_{\theta_0}v=R_{\theta_0}v-R_hv^{1+\frac{2}{n}}\in
L^{p_{l_0}}(S^{2n+1},\theta_0).$$ Hence, by using Theorem 3.17 in
\cite{Dragomir} (see also \cite{Folland&Stein}), we can conclude
that
$$\|v_k\|_{C^{\sigma}_P(S^{2n+1})}\leq C\hspace{2mm}\mbox{ with }\sigma=2-\frac{2n+2}{p_{l_0}},$$
Consequently, from (\ref{4.67}) of $v_j$ and Lemma \ref{lem3.2}, we
conclude that
$$v_j\rightarrow v_\infty\hspace{2mm}\mbox{ in }C^{1,\lambda}_P(S^{2n+1}),
\mbox{ for any }\lambda\in (0,1),\mbox{ as }j\rightarrow\infty.$$

Recall now that $\widehat{P(t_k)}$ is the center of mass of the contact form
$\theta_k=\theta(t_k)$, and by the remark right after Lemma
\ref{lem4.3}, (\ref{4.23}) holds and hence $\widehat{P(t_k)}\rightarrow Q$ as
$k\rightarrow\infty$ if the sequence is not bounded in
$S_2^p(S^{2n+1},\theta_0)$ for any $p>n+1$. For the normalized
conformal CR diffeomorphisms
$\phi_k=\phi(t_k)=\phi_{p(t_k),r(t_k)}$ where $p(t_k)\in\mathbb{H}^n$ and $r(t_k)>0$, if there exists a
subsequence $\{t_j\}$ such that $r_j=r(t_j)\rightarrow r_0<\infty$,
then $\phi_j \rightarrow\phi_{Q,r_0}$ as $j\rightarrow\infty$. Since
$v_j$ is bounded from above and below by positive constants, and so
is $\det(d\phi(t_j))\rightarrow\det(d\phi_{Q,r_0})$, we conclude
that $\{u_j=u(t_j)\}$ is uniformly bounded  from above and below by
positive constants. Hence, there exists a convergent subsequence
with $u_\infty$ as limit. Then by assumption
$R_{\theta(t_k)}\rightarrow R_\infty$ in $L^p(S^{2n+1},\theta_0)$
with $p>n+1$ and on the other hand
$R_{\theta(t_k)}-\alpha(t_k)f\rightarrow 0$ in
$L^p(S^{2n+1},\theta_0)$ with $p>n+1$, hence up to a subsequence
$R_\infty=c f$. Thus the Webster scalar curvature of the contact
form $u_\infty^{\frac{2}{n}}\theta_0$ is equal to $f$ up to a
constant, which contradicts our assumption that $f$ cannot be
realized as the Webster scalar curvature in the conformal class.
Therefore $r(t_k)\rightarrow\infty$ as $k\rightarrow\infty$.
Furthermore, there holds
\begin{equation}\label{4.69}
\begin{split}
&\phi_{p(t_k),r(t_k)}\rightarrow\phi_{Q,\infty}\equiv Q\hspace{2mm}\mbox{ uniformly for }
x\in S^{2n+1}\setminus B_{\delta}(Q)\\
&\mbox{ with any sufficiently small }\delta>0.
\end{split}
\end{equation}
Thus, from  (\ref{4.23}), (\ref{4.69}) and our assumption on $R_\infty$,
letting $k\rightarrow\infty$, we see that
\begin{equation*}
\begin{split}
&\|R_{h_k}-R_\infty(Q)\|_{L^p(S^{2n+1},h_k)}\\
&\leq \|R_{\theta_k}\circ\phi_k-R_\infty\circ\phi_k\|_{L^p(S^{2n+1},h_k)}
+\|R_\infty\circ\phi_k-R_\infty(Q)\|_{L^p(S^{2n+1},h_k)}\\
&=\|R_{\theta_k}-R_\infty\|_{L^p(S^{2n+1},\theta_k)}
+\|R_\infty-R_\infty(Q)\|_{L^p(S^{2n+1},\theta_k)}\rightarrow 0,
\end{split}
\end{equation*}
for any $p>n+1$. Hence, we have shown that $v_\infty$ weakly solves
$$-(2+\frac{2}{n})\Delta_{\theta_0}v_\infty+R_{\theta_0}v_\infty=R_\infty(Q)v_\infty^{1+\frac{2}{n}}
\hspace{2mm}\mbox{ on }S^{2n+1}.$$ Since $v_k$ satisfies
normalization (\ref{4.19}) and $\mbox{Vol}(S^{2n+1},h_k)=
\displaystyle\int_{S^{2n+1}}u_0^{2+\frac{2}{n}}dV_{\theta_0}=\mbox{Vol}(S^{2n+1},\theta_0)$,
$v_\infty$ must be constant  (see the proof of Theorem 3.1 in \cite{Frank&Lieb} or see \cite{Jerison&Lee1}).. Therefore,
we asset that $v_\infty\equiv 1$, which indicates (\ref{4.24}). This
proves Lemma \ref{lem4.13}
\end{proof}

\begin{lem}\label{lem4.14}
Let $f:S^{2n+1}\rightarrow\mathbb{R}$ be a smooth positive
non-degenerate Morse function satisfying the simple bubble condition
(sbc):
$$\frac{\max_{S^{2n+1}}f}{\min_{S^{2n+1}}f}<2^{\frac{1}{n}}.$$
Suppose that $f$ cannot be realized as the the Webster scalar
curvature of any contact form conformal to $\theta_0$. Let $u(t)$ be
a smooth solution of $(\ref{2.3})$ with initial data $u_0\in
C_f^\infty$. Then there exists a family of CR diffeomorphism
$\phi(t)$ on $S^{2n+1}$ with the normalized contact form
$h(t)=v(t)^{\frac{2}{n}}\theta_0=\phi(t)^*\big(\theta(t)\big)$ such
that as $t\rightarrow\infty$
$$v(t)\rightarrow 1,\hspace{2mm}h(t)\rightarrow\theta_0\hspace{2mm}\mbox{ in }C^{1,\gamma}_P(S^{2n+1})$$
for any $\gamma\in(0,1)$, and $\phi(t)-\widehat{P(t)}\rightarrow 0$ in
$L^2(S^{2n+1},\theta_0)$. Moreover, as $t\rightarrow\infty$, we have
$$\|f\circ\phi(t)-f(\widehat{P(t)})\|_{L^2(S^{2n+1},\theta_0)}\rightarrow 0\hspace{2mm}\mbox{ and
}\hspace{2mm}\alpha(t)f(\widehat{P(t)})\rightarrow R_{\theta_0}.$$
Here $\widehat{P(t)}$ is defined as in $(\ref{cm})$.
\end{lem}
\begin{proof}
We prove it by contradiction. Suppose, for a fixed $\gamma\in
(0,1)$, there exists a sequence $t_l\rightarrow\infty$ such that
\begin{equation}\label{4.70}
\lim_{l\rightarrow\infty}\Big(\|v(t_l)-1\|_{C^{1,\gamma}(S^{2n+1})}
+\|\phi(t_l)-\widehat{P(t_l)}\|_{L^2(S^{2n+1},\theta_0)}\Big)>0.
\end{equation}
Since Lemma \ref{lem4.13} implies that
\begin{equation}\label{4.71}
\lim_{l\rightarrow\infty}\|v(t_l)-1\|_{C^{1,\gamma}(S^{2n+1})}=0,
\end{equation}
by (\ref{4.70})
\begin{equation*}
\lim_{l\rightarrow\infty}
\|\phi(t_l)-\widehat{P(t_l)}\|^2_{L^2(S^{2n+1},\theta_0)}\equiv
C_0\mbox{Vol}(S^{2n+1},\theta_0)
\end{equation*}
for some constant $C_0>0$. Observe that
$|\phi(t)-\widehat{P(t)}|^2=2(1-\langle\phi(t),\widehat{P(t)}\rangle)$. Thus we have
\begin{equation}\label{4.72}
\begin{split}
C_0\mbox{Vol}(S^{2n+1},\theta_0)&=\lim_{l\rightarrow\infty}
\|\phi(t_l)-\widehat{P(t_l)}\|_{L^2(S^{2n+1},\theta_0)}^2\\
&=\lim_{l\rightarrow\infty}\int_{S^{2n+1}}2|1-\langle\phi(t_l),\widehat{P(t)}\rangle|dV_{\theta_0}\\
&=\lim_{l\rightarrow\infty}\int_{S^{2n+1}}2|1-\langle\phi(t_l),\widehat{P(t)}\rangle|dV_{h_l}\\
&\hspace{4mm}+\lim_{l\rightarrow\infty}\int_{S^{2n+1}}2|1-\langle\phi(t_l),\widehat{P(t)}\rangle|
(dV_{\theta_0}-dV_{h_l}).
\end{split}
\end{equation}
Clearly, the second limit is zero as it can be seen from
(\ref{4.71}). On the other hand, we have
\begin{equation}\label{4.73}
\begin{split}
\int_{S^{2n+1}}2|1-\langle\phi(t_l),\widehat{P(t)}\rangle|dV_{h_l}
&=2\Big(\mbox{Vol}(S^{2n+1},h_l)-\Big\langle\int_{S^{2n+1}}\phi(t_l)dV_{h_l},\widehat{P(t)}\Big\rangle\Big)\\
&=2\Big(\mbox{Vol}(S^{2n+1},\theta_0)-\Big\langle\int_{S^{2n+1}}x
u(t_l)^{2+\frac{2}{n}}dV_{\theta_0},\widehat{P(t)}\Big\rangle\Big)\\
&=2\mbox{Vol}(S^{2n+1},\theta_0)(1-\|P(t_l)\|^2).
\end{split}
\end{equation}
Combining (\ref{4.72}) and (\ref{4.73}), we obtain that
$$1-\|P(t_l)\|^2\geq\frac{C_0}{2}>0$$
for $l$ sufficiently large.

Since $\|P(t_l)\|^2\leq 1-\displaystyle \frac{C_0}{2}<1$, there
exists a subsequence, still denote as $\{t_l\}$, such that
$P(t_l)\rightarrow P_0$ as $l\rightarrow\infty$ with $\|P_0\|<1$.
Then $\phi(t_l)\rightarrow \phi_{\frac{P_0}{\|P_0\|},r_0}$ as
$l\rightarrow\infty$ with $r_0$ being finite. Hence, we can conclude
as in the proof of Lemma \ref{lem4.13} that $u(t_l)$ is uniformly
bounded from above and below by positive constants. Then it implies
that $\{u_l\}$ has a convergent subsequence, and its limit
$u_\infty$ is strictly positive and smooth such that
$u_\infty^{\frac{2}{n}}\theta_0$ has Webster scalar curvature equal
to $f$, which contradicts the assumption  that $f$ cannot be
realized as the Webster scalar curvature in the conformal class.

Moreover, since $v\rightarrow 1$ in $C^{1,\gamma}(S^{2n+1})$, we
have
$$\int_{S^{2n+1}}R_hdV_h
=\int_{S^{2n+1}}\left((2+\frac{2}{n})|\nabla_{\theta_0}v|^2_{\theta_0}+R_{\theta_0}v^2\right)dV_{\theta_0}
\rightarrow R_{\theta_0}\mbox{Vol}(S^{2n+1},\theta_0)$$ as
$t\rightarrow\infty$. Thus, we have
\begin{equation*}
\begin{split}
\lim_{t\rightarrow\infty}\Big(R_{\theta_0}-\alpha(t)f(\widehat{P(t)})\Big)\cdot\mbox{Vol}(S^{2n+1},\theta_0)
&=\lim_{t\rightarrow\infty}\left(\int_{S^{2n+1}}R_hdV_h
-\alpha(t)f(\widehat{P(t)})
\cdot\mbox{Vol}(S^{2n+1},\theta_0)\right)\\
&=\lim_{t\rightarrow\infty}\alpha(t)\int_{S^{2n+1}}\Big(f\circ\phi(t)
-f(\widehat{P(t)})\Big)dV_h=0.
\end{split}
\end{equation*}
Thus the proof of Lemma \ref{lem4.14} is complete.
\end{proof}

\section{Spectral decomposition}\label{section6}

As before, we denote
$$F_p(t)=\int_{S^{2n+1}}|R_\theta-\alpha f|^pdV_\theta\hspace{2mm}\mbox{ and }\hspace{2mm}G_p(t)
=\int_{S^{2n+1}}|\nabla_\theta(R_\theta-\alpha f)|_\theta^pdV_\theta$$
for $p\geq 1$.

\subsection{Upper bound of change rate of $F_2(t)$}

\begin{lem}\label{lem6.1}
With error $o(1)\rightarrow 0$ as $t\rightarrow\infty$, there holds
$$\frac{d}{dt}F_2(t)\leq (n+1+o(1))(nF_2(t)-2G_2(t))+o(1)F_2(t).$$
\end{lem}
\begin{proof}
By (\ref{3.2}), we have
\begin{equation}\label{6.1}
\begin{split}
&\frac{d}{dt}\left(\frac{1}{2}\int_{S^{2n+1}}(\alpha
f-R_\theta)^{2}dV_{\theta}\right)+\frac{1}{n+1}\left(\frac{\alpha'}{\alpha}\right)^2E(u)\\
&=\frac{n}{n+1}\frac{\alpha'}{\alpha}\int_{S^{2n+1}}(\alpha
f-R_\theta)R_\theta dV_{\theta}+\int_{S^{2n+1}}(\alpha
f-R_\theta)^2R_\theta
dV_{\theta}\\
&\hspace{4mm}-(n+1)\int_{S^{2n+1}}|\nabla_\theta(\alpha
f-R_\theta)|^2_\theta
dV_{\theta}+\frac{n+1}{2}\int_{S^{2n+1}}(\alpha
f-R_\theta)^3dV_{\theta}\\
&=\frac{n}{n+1}\frac{\alpha'}{\alpha}\int_{S^{2n+1}}\alpha f(\alpha
f-R_\theta)dV_{\theta}
-\frac{n}{n+1}\frac{\alpha'}{\alpha}\int_{S^{2n+1}}(\alpha
f-R_\theta)^2dV_{\theta}\\
&\hspace{4mm}+\int_{S^{2n+1}}\alpha f(\alpha f-R_\theta)^2
dV_{\theta}+\frac{n-1}{2}\int_{S^{2n+1}}(\alpha
f-R_\theta)^3dV_{\theta}-(n+1)\int_{S^{2n+1}}|\nabla_\theta(\alpha
f-R_\theta)|^2_\theta dV_{\theta}\\
&=\frac{n}{(n+1)E(u)}\left[-n\int_{S^{2n+1}}(\alpha
f-R_\theta)^2dV_{\theta}-\int_{S^{2n+1}}\alpha f (\alpha
f-R_{\theta})dV_{\theta}\right]\int_{S^{2n+1}}\alpha
f(\alpha f-R_\theta)dV_{\theta} \\
&\hspace{4mm}-\frac{n}{n+1}\frac{\alpha'}{\alpha}\int_{S^{2n+1}}(\alpha
f-R_\theta)^2dV_{\theta}+\int_{S^{2n+1}}\alpha f(\alpha
f-R_\theta)^2 dV_{\theta}\\
&\hspace{4mm}+\frac{n-1}{2}\int_{S^{2n+1}}(\alpha
f-R_\theta)^3dV_{\theta}-(n+1)\int_{S^{2n+1}}|\nabla_\theta(\alpha
f-R_\theta)|^2_\theta dV_{\theta}
\end{split}
\end{equation}
where the last equality follows from (\ref{2.22}). Note that the
first term on the right hand side of (\ref{6.1}) is less than or
equals to
\begin{equation*}
\begin{split}
&-\frac{n^2}{(n+1)E(u)}\left(\int_{S^{2n+1}}(\alpha
f-R_\theta)^2dV_{\theta}\right)\left(\int_{S^{2n+1}}\alpha f (\alpha
f-R_{\theta})dV_{\theta}\right)\\
&\leq\frac{n^2}{(n+1)R_{\theta_0}\mbox{Vol}(S^{2n+1},\theta_0)}F_2(t)\left(\alpha_2\big(\max_{S^{2n+1}}
f\big)F_2(t)^{\frac{1}{2}}\mbox{Vol}(S^{2n+1},\theta_0)^{\frac{1}{2}}\right)=o(1)F_2(t)
\end{split}
\end{equation*}
by (\ref{2.12a}), H\"{o}lder's inequality, Proposition \ref{prop2.1}, Lemma \ref{lem2.4}, and Lemma \ref{lem3.2}. For the second term on
the right hand side of (\ref{6.1}), we can combine it with the
second term on the left hand side to get
\begin{equation*}
\begin{split}
&\frac{1}{n+1}\left(\frac{\alpha'}{\alpha}\right)^2E(u)+\frac{n}{n+1}\frac{\alpha'}{\alpha}\int_{S^{2n+1}}(\alpha
f-R_\theta)^2dV_{\theta}\\
&=\frac{1}{n+1}\left[\frac{\alpha'}{\alpha}\sqrt{E(u)}+\frac{n\int_{S^{2n+1}}(\alpha
f-R_\theta)^2dV_{\theta}}{2\sqrt{E(u)}}\right]^2
-\frac{n^2}{4(n+1)E(u)}\left(\int_{S^{2n+1}}(\alpha
f-R_\theta)^2dV_{\theta}\right)^2\\
&\geq
-\frac{n^2}{4(n+1)R_{\theta_0}\mbox{Vol}(S^{2n+1},\theta_0)}\left(\int_{S^{2n+1}}(\alpha
f-R_\theta)^2dV_{\theta}\right)^2
\end{split}
\end{equation*}
by (\ref{2.12a}) and Proposition \ref{prop2.1}. Note also that
\begin{equation*}
\begin{split}
\left(\int_{S^{2n+1}}(\alpha f-R_\theta)^2dV_{\theta}\right)^2&\leq
\mbox{Vol}(S^{2n+1},\theta_0)^{\frac{2}{3}}\left(\int_{S^{2n+1}}|\alpha
f-R_\theta|^3dV_{\theta}\right)^{\frac{4}{3}}\\
&=C\int_{S^{2n+1}}|\alpha f-R_\theta|^3dV_{\theta}
\end{split}
\end{equation*}
by H\"{o}lder's inequality, Proposition \ref{prop2.1} and Lemma
\ref{lem3.2}. Combining all these, we can deduce from (\ref{6.1})
that
\begin{equation}\label{6.2}
\frac{1}{2}\frac{d}{dt}F_2(t)\leq \int_{S^{2n+1}}\alpha f(\alpha
f-R_\theta)^2
dV_{\theta}-(n+1)G_2(t)+C\int_{S^{2n+1}}|\alpha
f-R_\theta|^3dV_{\theta}+o(1)F_2(t).
\end{equation}
Observe that
\begin{equation}\label{6.3}
\begin{split}
&\int_{S^{2n+1}}|\alpha f-R_\theta|^3dV_{\theta}\\
&\leq \left(\int_{S^{2n+1}}|\alpha
f-R_\theta|^{n+1}dV_{\theta}\right)^{\frac{1}{n+1}}\left(\int_{S^{2n+1}}|\alpha
f-R_\theta|^{2+\frac{2}{n}}dV_{\theta}\right)^{\frac{n}{n+1}}\\
&\leq
F_{n+1}(t)^{\frac{1}{n+1}}\cdot\frac{1}{Y(S^{2n+1},\theta_0)}\left(\int_{S^{2n+1}}(2+\frac{2}{n})|\nabla_\theta(\alpha
f-R_\theta)|_\theta^2+R_\theta|\alpha
f-R_\theta|^2dV_{\theta}\right)\\
&=o(1)\left(\int_{S^{2n+1}}|\nabla_\theta(\alpha
f-R_\theta)|_\theta^2+|\alpha f-R_\theta|^2+|\alpha
f-R_\theta|^3dV_{\theta}\right)
\end{split}
\end{equation}
by H\"{o}lder's inequality, Lemma \ref{lem2.3}, \ref{lem2.4}, and
\ref{lem3.2}, which implies that
\begin{equation}\label{6.4}
\int_{S^{2n+1}}|\alpha f-R_\theta|^3dV_{\theta} \leq
o(1)(G_2(t)+F_2(t)).
\end{equation}
Finally, we rewrite the first term in (\ref{6.2}) as
\begin{equation}\label{6.5}
\begin{split}
\int_{S^{2n+1}}\alpha f(\alpha f-R_\theta)^2
dV_{\theta}&=\int_{S^{2n+1}}\alpha (f_\phi-f(\widehat{P(t)})(\alpha
f_\phi-R_h)^2 dV_{h}\\
&\hspace{4mm}+\Big(\alpha
f(\widehat{P(t)})-\frac{n(n+1)}{2}\Big)\int_{S^{2n+1}}(\alpha
f_\phi-R_h)^2 dV_{h}\\
&\hspace{4mm}+\frac{n(n+1)}{2}\int_{S^{2n+1}}(\alpha f_\phi-R_h)^2
dV_{h}.
\end{split}
\end{equation}
By Lemma \ref{lem4.14}, $\alpha f(\widehat{P(t)})\rightarrow
R_{\theta_0}=\displaystyle\frac{n(n+1)}{2}$ as $t\rightarrow\infty$,
we obtain that
\begin{equation}\label{6.6}
\Big(\alpha
f(\widehat{P(t)})-\frac{n(n+1)}{2}\Big)\int_{S^{2n+1}}(\alpha
f_\phi-R_h)^2 dV_{h}=o(1)F_2(t).
\end{equation} To get control of the first term on
the right hand side of (\ref{6.5}), first we notice that
\begin{equation*}
\int_{S^{2n+1}}|f_\phi-f(\widehat{P(t)})|^{n+1} dV_{h}
\leq C\int_{S^{2n+1}}|f_\phi-f(\widehat{P(t)})|^{n+1} dV_{\theta_0}=o(1)
\end{equation*}
by Lemma \ref{lem4.14}.  Hence, by (\ref{6.3}),
H\"{o}lder's inequality, Lemma \ref{lem2.3} and \ref{lem2.4}, we have
\begin{equation}\label{6.7}
\begin{split}
&\left|\int_{S^{2n+1}}\alpha (f_\phi-f(\widehat{P(t)})(\alpha f_\phi-R_h)^2 dV_{h}\right|\\
&\leq\alpha_2\left(\int_{S^{2n+1}} |f_\phi-f(\widehat{P(t)})|^{n+1}
dV_{h}\right)^{\frac{1}{n+1}}\left(\int_{S^{2n+1}} |\alpha
f_\phi-R_h|^{2+\frac{2}{n}} dV_{h}\right)^{\frac{n}{n+1}}\\
&\leq o(1)\left(\int_{S^{2n+1}}(2+\frac{2}{n})|\nabla_\theta(\alpha
f-R_\theta)|_\theta^2+R_\theta|\alpha
f-R_\theta|^2dV_{\theta}\right)\\
&\leq o(1)(G_2(t)+F_2(t)).
\end{split}
\end{equation}
Substituting  (\ref{6.3})-(\ref{6.7}) into (\ref{6.2}), we deduce
\begin{equation*}
\begin{split}
\frac{1}{2}\frac{d}{dt}F_2(t)\leq
\frac{n(n+1)}{2}F_2(t)-(n+1)G_2(t)+o(1)(G_2(t)+F_2(t)),
\end{split}
\end{equation*}
as required.
\end{proof}

\subsection{The spectral decomposition}

Let $\{\varphi_i\}$ be an $L^2(S^{2n+1},\theta_0)$-orthonormal basis
of eigenfunctions of $-\Delta_{\theta_0}$, satisfying
$-\Delta_{\theta_0}\varphi_i=\lambda_i\varphi_i$ with eigenvalues
$0=\lambda_0<\lambda_1=\cdots=\lambda_{2n+2}=\displaystyle\frac{n}{2}<\lambda_{2n+3}\leq\cdots$.
Now in terms of the orthonormal basis $\{\varphi_i^\theta\}$,
$\{\varphi_i^h\}$ of the eigenfunctions of $-\Delta_\theta$,
$-\Delta_h$ with the corresponding eigenvalues $\lambda_i^\theta$,
$\lambda_i^h$ respectively, we expand
\begin{equation}\label{6.20}
\alpha
f-R_\theta=\sum_{i=0}^\infty\beta_\theta^i\varphi_i^\theta\hspace{2mm}\mbox{
and }\hspace{2mm}\alpha
f_\phi-R_h=\sum_{i=0}^\infty\beta_h^i\varphi_i^h,
\end{equation}
with coefficients
\begin{equation}\label{6.21}
\beta^i_h=\int_{S^{2n+1}}(\alpha f_\phi-R_h)\varphi_i^h
\,dV_h=\int_{S^{2n+1}}(\alpha f-R_\theta)\varphi_i^\theta
\,dV_\theta=\beta_\theta^i
\end{equation}
for all $i\in\mathbb{N}$. First notice that we always have
\begin{equation}\label{6.19}
\beta^0_\theta=0
\end{equation} in view of (\ref{2.2}). It is well known that
$\varphi_i^h=\varphi_i^\theta\circ\phi$, which implies (\ref{6.21})
and $\lambda_i^\theta=\lambda_i^h$ for all $i\in\mathbb{N}$.

\begin{lem}\label{lem6.3}
 As $t\rightarrow\infty$, we have
$\lambda_i^\theta=\lambda_i^h\rightarrow\lambda_i$ and we can choose
$\varphi_i$ such that $\varphi_i^h\rightarrow\varphi_i$ in
$L^2(S^{2n+1},\theta_0)$ for all $i\in\mathbb{N}$.
\end{lem}
\begin{proof}
First, for $i=0$, for any time $t$, the eigenvalue
$\lambda_0^h=\lambda_0=0$ and $\varphi_0^h=\varphi_0=1$. Thus the
statement is true for $i=0$. Now assume that the statement is true
for all $i\leq k$ for a fixed integer $k\geq 0$. We should show that
it is true for $i=k+1$ for a suitable choice of $\varphi_{k+1}$. As
starting point, we show that $\lambda_{k+1}^h\rightarrow\lambda_{k+1}$. The
argument in the proof of Lemma \ref{lem4.11} shows that
$\displaystyle\liminf_{t\rightarrow\infty}\lambda_{k+1}^h\geq\lambda_{k+1}$.
Hence we only need to show that
$\displaystyle\limsup_{t\rightarrow\infty}\lambda_{k+1}^h\leq\lambda_{k+1}$.
To do this, we pick $\varphi_{k+1}$ the eigenfunction for
$-\Delta_{\theta_0}$ which is perpendicular to all $\varphi_j$ for
$0\leq j\leq k$ with the smaller eigenvalue. Since by assumption
$\varphi_j^h\rightarrow\varphi_j$ as $t\rightarrow\infty$ and
$\displaystyle\int_{S^{2n+1}}\varphi_{k+1}\varphi_j\,dV_{\theta_0}=0$
for all $0\leq j\leq k$, we conclude that
\begin{equation}\label{6.22}
\int_{S^{2n+1}}\varphi_{k+1}\varphi_j^h\,dV_{\theta_0}=o(1)\hspace{2mm}\mbox{
as }t\rightarrow\infty.
\end{equation}
Now consider the test function
\begin{equation}\label{6.23}
\Psi^h=\varphi_{k+1}-\sum_{j=0}^k\left(\int_{S^{2n+1}}\varphi_{k+1}\varphi_j^h\,dV_h\right)\varphi_j^h.
\end{equation}
Since $\{\varphi_i^h\}_{i\in\mathbb{N}}$ are orthonormal, one has
\begin{equation}\label{6.24}
\begin{split}
\int_{S^{2n+1}}\Psi^h\varphi_l^h\,dV_h
&= \int_{S^{2n+1}}\varphi_{k+1}\varphi_l^h\,dV_h
-\sum_{j=0}^k\left(\int_{S^{2n+1}}\varphi_{k+1}\varphi_j^h\,dV_h\right)\int_{S^{2n+1}}\varphi_j^h\varphi_l^h\,dV_h\\
&=\int_{S^{2n+1}}\varphi_{k+1}\varphi_l^h\,dV_h
-\sum_{j=0}^k\left(\int_{S^{2n+1}}\varphi_{k+1}\varphi_j^h\,dV_h\right)\delta_{jl}=0
\hspace{2mm}\mbox{ for all }0\leq l\leq k.
\end{split}
\end{equation}
Thus by characterization of the $(k+1)$-th eigenvalue, one obtains
\begin{equation}\label{6.25}
\lambda_{k+1}^h\leq\frac{\int_{S^{2n+1}}|\nabla_h\Psi^h|_h^2dV_h}{\int_{S^{2n+1}}(\Psi^h)^2dV_h}.
\end{equation}
Note that
\begin{equation*}
\begin{split}
\int_{S^{2n+1}}(\Psi^h)^2dV_h &=\int_{S^{2n+1}}\Psi^h\varphi_{k+1}\,dV_h
-\sum_{j=0}^k\left(\int_{S^{2n+1}}\varphi_{k+1}\varphi_j^h\,dV_h\right)\int_{S^{2n+1}}\Psi^h\varphi_j^h\,dV_h\\
&=\int_{S^{2n+1}}\Psi^h\varphi_{k+1}\,dV_h\\
&=\int_{S^{2n+1}}\varphi_{k+1}^2\,dV_h
-\sum_{j=0}^k\left(\int_{S^{2n+1}}\varphi_{k+1}\varphi_j^h\,dV_h\right)^2=1+o(1)\mbox{
as }t\rightarrow\infty,
\end{split}
\end{equation*}
where the second equality follows from (\ref{6.24}), and the last
equality follows from (\ref{4.24}) and (\ref{6.22}). Note also that
\begin{equation*}
\begin{split}
\int_{S^{2n+1}}|\nabla_h\Psi^h|_h^2dV_h
&=\int_{S^{2n+1}}|\nabla_h\varphi_{k+1}|_h^2dV_h
-2\sum_{j=0}^k\left(\int_{S^{2n+1}}\varphi_{k+1}\varphi_j^h\,dV_h\right)
\int_{S^{2n+1}}\langle\nabla_h\varphi_{k+1},\nabla_h\varphi_j^h\rangle_hdV_h\\
&\hspace{2mm}+
\sum_{j,l=0}^k\left(\int_{S^{2n+1}}\varphi_{k+1}\varphi_j^h\,dV_h\right)
\left(\int_{S^{2n+1}}\varphi_{k+1}\varphi_l^h\,dV_h\right)
\int_{S^{2n+1}}\langle\nabla_h\varphi_j^h,\nabla_h\varphi_l^h\rangle_hdV_h\\
&=\lambda_{k+1}+o(1)\mbox{ as }t\rightarrow\infty,
\end{split}
\end{equation*}
where we have used (\ref{4.24}) and (\ref{6.22}). Now if we set
$\varphi_{k+1}^h=\displaystyle\Psi^h\Big/\left(\int_{S^{2n+1}}(\Psi^h)^2dV_h\right)^{\frac{1}{2}}$,
then the required estimate follows easily from (\ref{6.25}) and
these two estimates.

Since $\{\varphi_{k+1}^h\}$ is a bounded sequence in
$S_2^2(S^{2n+1},\theta_0)$, there exists a subsequence of time $t_l$
such that $\{\varphi_{k+1}^h(t_l)\}$ weakly converges to
$\varphi_{k+1}^0$. Hence it will converge to $\varphi_{k+1}^0$
strongly in $S_1^2(S^{2n+1},\theta_0)$. In particular, we have
$$\int_{S^{2n+1}}(\varphi_{k+1}^0)^2dV_{\theta_0}=1\hspace{2mm}\mbox{ and }\hspace{2mm}
\int_{S^{2n+1}}|\nabla_{\theta_0}\varphi_{k+1}^0|_{\theta_0}^2dV_{\theta_0}=\lambda_{k+1}.$$
Then $\displaystyle\lim_{k\rightarrow\infty}
\int_{S^{2n+1}}|\nabla_{\theta_0}\varphi_{k+1}^h|_{\theta_0}^2dV_{\theta_0}=\lambda_{k+1}$,
and $\varphi_{k+1}^0$ is in the eigenfunction of the eigenvalue
$\lambda_{k+1}$ and is orthogonal to all eigenfunctions $\varphi_j$
with $0\leq j\leq k$. Redefine $\varphi_{k+1}$ to be
$\varphi_{k+1}^0$ to finish the induction argument.
\end{proof}

\subsection{Convergence of CR Yamabe flow}

Now we consider the case when $f$ is a positive constant function, i.e. $f\equiv c$ for some positive constant $c>0$. Then
by (\ref{2.2}), we have
\begin{equation*}
\alpha f=\frac{\alpha\int_{S^{2n+1}}f dV_{\theta}}{\int_{S^{2n+1}} dV_{\theta}}=\frac{\int_{S^{2n+1}}R_\theta dV_{\theta}}{\int_{S^{2n+1}}dV_\theta}
=\overline{R}_\theta
\end{equation*}
where  $\overline{R}_{\theta}$ is the average of the Webster scalar curvature $R_{\theta}$ of $\theta$.
Therefore, the Webster scalar curvature flow in (\ref{2.1}) reduced to the CR Yamabe flow
\begin{equation*}
\frac{\partial }{\partial t}\theta=-(R_{\theta}- \overline{R}_{\theta})\theta,
\hspace{4mm} \theta\big|_{t=0}=u_0^{\frac{2}{n}}\theta_0.
\end{equation*}
Our aim is to prove the following result, which recovers the result of the author in \cite{Ho2}:
\begin{theorem}\label{thm6.3}
Suppose that $u_0$ satisfies $(\ref{2.01})$ and $f$ is a positive constant function. Then the flow $(\ref{2.1})$
converges to a contact form $\theta_\infty=u_\infty^{\frac{2}{n}}\theta_0$ of constant
Webster scalar curvature.
\end{theorem}
\begin{proof}
First we show the exponential decay of $F_2$. We define
$$b=(b^1,\cdots,b^{2n+2})=\int_{S^{2n+1}}(x,\overline{x})(\overline{R}_h-R_h)dV_h$$
where $x=(x_1,...,x_{n+1})\in S^{2n+1}\subset\mathbb{C}^{n+1}$ and
$\overline{x}=(\overline{x}_1,...,\overline{x}_{n+1})$.
That is,
\begin{equation}\label{5.0}
b^i=\int_{S^{2n+1}}x_i(\overline{R}_h-R_h)dV_h\mbox{ and }b^{n+1+i}=\int_{S^{2n+1}}\overline{x}_i(\overline{R}_h-R_h)dV_h\mbox{ for }1\leq i\leq n+1.
\end{equation}
J. H. Cheng proved the following
Kazdan-Warner type condition in \cite{Cheng}:
\begin{equation}\label{5.1}
\int_{S^{2n+1}}\langle\nabla_{\theta_0}x_i,
\nabla_{\theta_0}R_h\rangle_{\theta_0}dV_h=0\mbox{ and
}\int_{S^{2n+1}}\langle\nabla_{\theta_0}\overline{x}_i,
\nabla_{\theta_0}R_h\rangle_{\theta_0}dV_h=0
\end{equation}
for $1\leq i\leq n+1$.
Since $$-\Delta_{\theta_0}x_i=\frac{n}{2}x_i\hspace{2mm}\mbox{ and }\hspace{2mm}-\Delta_{\theta_0}\overline{x}_i=\frac{n}{2}\overline{x}_i\mbox{ for }1\leq i\leq n+1,$$ by (\ref{5.1}) and integration by parts we find that
\begin{equation}\label{5.2}
\begin{split}
b^i&=\int_{S^{2n+1}}(\overline{R}_h-R_h)x_i\,dV_h=-\frac{2}{n}\int_{S^{2n+1}}(\overline{R}_h-R_h)\Delta_{\theta_0}x_i\,dV_h\\
&=\frac{4n+4}{n^2}\int_{S^{2n+1}}(\overline{R}_h-R_h)\langle\nabla_{\theta_0}x_i,\nabla_{\theta_0}v\rangle_{\theta_0}v^{1+\frac{2}{n}}dV_{\theta_0}\mbox{ for }1\leq i\leq n+1.
\end{split}
\end{equation}
Similarly, we have
\begin{equation}\label{5.3}
b^{n+1+i}
=\frac{4n+4}{n^2}\int_{S^{2n+1}}(\overline{R}_h-R_h)\langle\nabla_{\theta_0}\overline{x}_i,\nabla_{\theta_0}v\rangle_{\theta_0}v^{1+\frac{2}{n}}dV_{\theta_0}\mbox{ for }1\leq i\leq n+1.
\end{equation}
It follows from (\ref{5.2}), (\ref{5.3}) and Lemma \ref{lem4.14} that
with error $o(1)\rightarrow 0$ as $t\rightarrow\infty$
\begin{equation}\label{5.4}
|b^i|\leq CF_2(t)^{\frac{1}{2}}\|v-1\|_{S_1^2(S^{2n+1},\theta_0)}=o(1)F_2(t)^{\frac{1}{2}}\mbox{ for }1\leq i\leq 2n+2.
\end{equation}
We can take $\varphi_i=x_i/\sqrt{n+1}$ and $\varphi_{n+i}=\overline{x}_i/\sqrt{n+1}$ for $1\leq i\leq n+1$. Hence, by (\ref{6.21}), (\ref{5.0}), and Lemma \ref{lem6.3}, we obtain
\begin{equation*}
\beta^i_h=\int_{S^{2n+1}}\varphi^h_i(\overline{R}_h-R_h)dV_h=\int_{S^{2n+1}}\varphi_i(\overline{R}_h-R_h)dV_h+o(1)
=\frac{b_i}{\sqrt{n+1}}+o(1)\mbox{ for }1\leq i\leq 2n+2,
\end{equation*}
which implies by (\ref{6.21}) and (\ref{5.4}) that
\begin{equation}\label{5.5}
|\beta^i_\theta|=|\beta^i_h|\leq o(1)F_2(t)^{\frac{1}{2}}+o(1)\mbox{ for }1\leq i\leq 2n+2.
\end{equation}

On the other hand, by (\ref{6.20}), (\ref{6.21}) and (\ref{6.19}), we have
\begin{equation}\label{5.6}
F_2(t)=\int_{S^{2n+1}}(\overline{R}_\theta-R_\theta)^2dV_\theta
=\sum_{i,j=1}^\infty\beta_\theta^i\beta_\theta^j\int_{S^{2n+1}}\varphi_i^\theta\varphi_j^\theta\,dV_\theta\\
=\sum_{i=1}^\infty|\beta_\theta^i|^2
\end{equation}
and
\begin{equation}\label{5.7}
\begin{split}
G_2(t)=\int_{S^{2n+1}}|\nabla_\theta(\overline{R}_\theta-R_\theta)|_{\theta}^2dV_\theta&=-\int_{S^{2n+1}}(\overline{R}_\theta-R_\theta)\Delta_\theta (\overline{R}_\theta-R_\theta)dV_\theta\\
&=\sum_{i,j=1}^\infty\beta_\theta^i\beta_\theta^j\int_{S^{2n+1}}\varphi_i^\theta
\big(-\Delta_\theta\varphi_j^\theta\big)\,dV_\theta\\
&=\sum_{i,j=1}^\infty\beta_\theta^i\beta_\theta^j\lambda_j^\theta
\int_{S^{2n+1}}\varphi_i^\theta\varphi_j^\theta\,dV_\theta=\sum_{i=1}^\infty\lambda_i^\theta|\beta_\theta^i|^2.
\end{split}
\end{equation}
Combining (\ref{5.5}), (\ref{5.6}), (\ref{5.7}) and Lemma \ref{lem6.3}, we obtain
\begin{equation}\label{5.8}
G_2(t)=\sum_{i=1}^\infty\lambda_i^\theta|\beta_\theta^i|^2\geq (\lambda_{2n+3}^\theta+o(1))\sum_{i=1}^\infty|\beta_\theta^i|^2
= (\lambda_{2n+3}+o(1))F_2(t).
\end{equation}
Combining (\ref{5.8}) and Lemma \ref{lem6.1} and taking advantage of the spectral gap $\lambda_{2n+3}>n/2$,
for sufficiently large $t$ we infer the
estimate
$$\frac{d}{dt}F_2(t)\leq-\delta F_2(t)$$
for some uniform constant $\delta>0$, which implies that
\begin{equation}\label{5.9}
F_2(t)\leq Ce^{-\delta t}
\end{equation}
for all $t\geq 0$ with some uniform constant $C$.

Note that (\ref{5.9}) rules out the concentration of volume. Indeed, let $Q$ be the
unique concentration point described in Theorem \ref{thm4.7}, and
$B_{r_0}(Q)=B_{r_0}(Q,\theta_0)$. For any $r_0>0$, we have
\begin{equation}\label{5.10}
\begin{split}
\left|\frac{d}{dt}\mbox{Vol}(B_{r_0}(Q),\theta)\right|
&=\left|\frac{d}{dt}\Big(\int_{B_{r_0}(Q)}dV_\theta\Big)\right|=(n+1)\left|\int_{B_{r_0}(Q)}(\alpha
f-R_\theta)dV_\theta\right|\\
&\leq(n+1)\mbox{Vol}(S^{2n+1},\theta)^{\frac{1}{2}}\left(\int_{B_{r_0}(Q)}(\alpha
f-R_\theta)^2dV_\theta\right)^{\frac{1}{2}}\\
&\leq
(n+1)\mbox{Vol}(S^{2n+1},\theta_0)^{\frac{1}{2}}F_2(t)^{\frac{1}{2}}\leq C
e^{-\frac{\delta}{2}t}\hspace{2mm}\mbox{
for all }t\geq 0.
\end{split}
\end{equation}
Here we have used (\ref{2.3}), (\ref{5.9}) and H\"{o}lder's inequality. Thus, by integrating (\ref{5.10})
 from $T$ to $t$, we obtain
\begin{equation}\label{6.35}
\mbox{Vol}(B_{r_0}(Q),\theta(t))\leq\mbox{Vol}(B_{r_0}(Q),\theta(T))
+\frac{C}{\delta}e^{-\frac{\delta}{2}T}
\end{equation}
for all $t\geq T$. Now, first by choosing $T$ large enough and then by choosing $r_0>0$ small enough, we obtain
\begin{equation*}
\mbox{Vol}(B_{r_0}(Q),\theta(t))\leq
\frac{1}{2}\mbox{Vol}(S^{2n+1},\theta_0)\mbox{ for all }t\geq T,
\end{equation*}
thanks to (\ref{6.35}). In particular, the concentration in the sense of (\ref{4.23}) cannot occur.
By Theorem \ref{thm4.7}, case (i) occurs. Therefore, Lemma \ref{lem4.8} implies that
$u(t)$ converges to $u_\infty$ in $S_2^p(S^{2n+1},\theta_0)$ for all $p>n+1$, and hence in $C^\infty(S^{2n+1})$, as $t\rightarrow\infty$ such that the Webster scalar curvature of
$\theta_\infty=u_\infty^{\frac{2}{n}}\theta_0$ is constant.
\end{proof}

\section{Appendix}

We have the following lemma which is an improved version of Lemma \ref{lem2.3}:

\begin{lem}\label{lemA}
For any $\epsilon>0$, there exists a constant $C_\epsilon$ such that
for any $0\leq u\in S^2_1(S^{2n+1},\theta_0)$ satisfying
\begin{equation}\label{8.1}
\int_{S^{2n+1}}x_iu^{2+\frac{2}{n}}dV_{\theta_0}=\int_{S^{2n+1}}\overline{x}_iu^{2+\frac{2}{n}}dV_{\theta_0}=0\hspace{2mm}\mbox{
for all }i=1,2...,n+1,
\end{equation}
where $x=(x_1,...,x_{n+1})\in S^{2n+1}\subset\mathbb{C}^{n+1}$,
 we have
\begin{equation}\label{8.1a}
Y(S^{2n+1},\theta_0)\left(\int_{S^{2n+1}}u^{2+\frac{2}{n}}dV_{\theta_0}\right)^{\frac{n}{n+1}}
\leq \left(2^{-\frac{1}{n+1}}\cdot\frac{2n+2}{n}+\epsilon\right)
\int_{S^{2n+1}}|\nabla_{\theta_0} u|^2_{\theta_0}
dV_{\theta_0}+C_\epsilon\int_{S^{2n+1}}u^2dV_{\theta_0}.
\end{equation}
\end{lem}
\begin{proof}
The proof is based on the argument of Aubin in \cite{Aubin1}
(see also \cite{Aubin2}).

Dividing both sides by $R_{\theta_0}=n(n+1)/2$, (\ref{8.1a}) is equivalent to
\begin{equation}\label{8.1b}
\mbox{Vol}(S^{2n+1},\theta_0)^{\frac{1}{n+1}}\left(\int_{S^{2n+1}}u^{2+\frac{2}{n}}dV_{\theta_0}\right)^{\frac{n}{n+1}}
\leq \left(2^{-\frac{1}{n+1}}\cdot\frac{4}{n^2}+\epsilon\right)
\int_{S^{2n+1}}|\nabla_{\theta_0} u|^2_{\theta_0}
dV_{\theta_0}+C_\epsilon\int_{S^{2n+1}}u^2dV_{\theta_0}.
\end{equation}
Let $\Lambda$ be the vector space of functions on $S^{2n+1}$ spanned by
$x_i$ and $\overline{x}_i$ where $i=1,2...,n+1$. Let $0<\eta<1/2$ be a real number, which we are going to choose very small. There exists a family of functions $\xi_i\in\Lambda$ ($i=1,2,..., k$) such that
\begin{equation}\label{8.2}
1+\eta<\sum_{i=1}^k|\xi_i|^{\frac{n}{n+1}}<1+2\eta\hspace{2mm}\mbox{ with }\hspace{2mm}|\xi_i|<2^{-\frac{2n+2}{n}}.
\end{equation}
Consider $h_i$, $C^1$ functions, such that everywhere $h_i\,\xi_i\geq 0$ and such that
\begin{equation}\label{8.3}
\big||h_i|^2-|\xi_i|^{\frac{n}{n+1}}\big|<(\eta/k)^{\frac{2n+2}{n}}.
\end{equation}
 Then we have
\begin{equation}\label{8.4}
\begin{split}
&\sum_{i=1}^k|h_i|^2<\sum_{i=1}^k|\xi_i|^{\frac{n}{n+1}}+k(\eta/k)^{\frac{n}{2n+2}}<1+2\eta+\eta=1+3\eta\hspace{2mm}\mbox{ and }\\
&\sum_{i=1}^k|h_i|^2>\sum_{i=1}^k|\xi_i|^{\frac{n}{n+1}}-k(\eta/k)^{\frac{n}{2n+2}}>1+\eta-\eta=1.
\end{split}
\end{equation}
By mean-value theorem, we have $|F(a)-F(b)|\leq
\max_{x\in[a,b]}|F'(x)|(a-b)$ for $a>b$. Apply this with
$F(x)=x^{\frac{n+1}{n}}$ and note that for $x$ lies between
$|h_i|^{2}$ and $|\xi_i|^{\frac{n}{n+1}}$ satisfying
$|x|=|h_i|^{2}+\big||h_i|^2-|\xi_i|^{\frac{n}{n+1}}\big|<|h_i|^{2}+(\eta/k)^{\frac{2n+2}{n}}$
by (\ref{8.3}), we have
\begin{equation}\label{8.5}
\begin{split}
\big||h_i|^{\frac{2n+2}{n}}-|\xi_i|\big|&\leq\frac{n+1}{n}
\left[|\xi_i|^{\frac{n}{n+1}}+(\eta/k)^{\frac{2n+2}{n}}\right]^{\frac{n+1}{n}-1}\big||h_i|^2-|\xi_i|^{\frac{n}{n+1}}\big|
\\
&\leq\frac{n+1}{n}\big||h_i|^2-|\xi_i|^{\frac{n}{n+1}}\big|\\
&\leq\frac{n+1}{n}(\eta/k)^{\frac{2n+2}{n}}:=\epsilon_0^{\frac{2n+2}{n}}
\end{split}
\end{equation}
where the second inequality follows from
\begin{equation*}
|\xi_i|^{\frac{n}{n+1}}+(\eta/k)^{\frac{2n+2}{n}}<2^{-2}+(1/2k)^{\frac{2n+2}{n}}<1
\end{equation*}
by (\ref{8.2}) and the assumption that $\eta<1/2$. Since
 $u\geq 0$, we have
\begin{equation}\label{8.6}
\begin{split}
\|u\|_{L^{\frac{2n+2}{n}}(S^{2n+1},\theta_0)}^2&=\|u^2\|_{L^{\frac{n+1}{n}}(S^{2n+1},\theta_0)}\\
&\leq\Big\|\sum_{i=1}^ku^2|h_i|^2\Big\|_{L^{\frac{n+1}{n}}(S^{2n+1},\theta_0)}
\leq\sum_{i=1}^k\|uh_i\|_{L^{\frac{n+1}{n}}(S^{2n+1},\theta_0)}^2.
\end{split}
\end{equation}

For a function $f$, define $f^+(x)=\sup\{f(x),0\}$ and
$f^-(x)=\sup\{-f(x),0\}$. By assumption (\ref{8.1}), we have
\begin{equation}\label{8.7}
\int_{S^{2n+1}}\xi_i^+u^{\frac{2+2n}{n}}dV_{\theta_0}=\int_{S^{2n+1}}\xi_i^-u^{\frac{2+2n}{n}}dV_{\theta_0}.
\end{equation}
If
\begin{equation}\label{8.8}
\big\|h_i^+|\nabla_{\theta_0}
u|_{\theta_0}\big\|_{L^2(S^{2n+1},\theta_0)}\geq
\big\|h_i^-|\nabla_{\theta_0}
u|_{\theta_0}\big\|_{L^2(S^{2n+1},\theta_0)},
\end{equation}
we obtain
\begin{equation}\label{8.9}
\begin{split}
&\|u h_i\|_{L^{\frac{2n+2}{n}}(S^{2n+1},\theta_0)}^2\\
&=\left[\int_{S^{2n+1}}(h_i^+u)^{\frac{2+2n}{n}}dV_{\theta_0}
+\int_{S^{2n+1}}(h_i^-u)^{\frac{2+2n}{n}}dV_{\theta_0}\right]^{\frac{n}{n+1}}\\
&\leq
\left[\int_{S^{2n+1}}\big(\xi_i^++\epsilon_0^{\frac{2+2n}{n}}\big)u^{\frac{2+2n}{n}}dV_{\theta_0}
+\int_{S^{2n+1}}(h_i^-u)^{\frac{2+2n}{n}}dV_{\theta_0}\right]^{\frac{n}{n+1}}\\
&=
\left[\int_{S^{2n+1}}\big(\xi_i^-+\epsilon_0^{\frac{2+2n}{n}}\big)u^{\frac{2+2n}{n}}dV_{\theta_0}
+\int_{S^{2n+1}}(h_i^-u)^{\frac{2+2n}{n}}dV_{\theta_0}\right]^{\frac{n}{n+1}}\\
&\leq
2^{\frac{n}{n+1}}\left[\int_{S^{2n+1}}\big((h_i^-)^{\frac{2+2n}{n}}+\epsilon_0^{\frac{2+2n}{n}}\big)u^{\frac{2+2n}{n}}dV_{\theta_0}
\right]^{\frac{n}{n+1}}\\
&\leq2^{\frac{n}{n+1}}\|u(h_i^-+\epsilon_0)\|_{L^{\frac{2n+2}{n}}(S^{2n+1},\theta_0)}^2\\
&\leq\frac{2^{\frac{n}{n+1}}}{\mbox{Vol}(S^{2n+1},\theta_0)^{\frac{1}{n+1}}}\left(\frac{4}{n^2}\|\nabla_{\theta_0}\big(u(h_i^-+\epsilon_0)\big)\|_{L^{2}(S^{2n+1},\theta_0)}^2
+\|u(h_i^-+\epsilon_0)\|_{L^{2}(S^{2n+1},\theta_0)}^2\right)
\end{split}
\end{equation}
where the second and third inequality follows from (\ref{8.5}), the
second equality follows from (\ref{8.7}), and the last inequality
follows from Lemma \ref{lem2.3}. Let $H=\displaystyle\sup_{1\leq
i\leq k}\sup_{S^{2n+1}}|\nabla_{\theta_0}h_i|_{\theta_0}$. Then
there exist constants $\mu$ and $\nu$ such that
\begin{equation}\label{8.10}
\begin{split}
\|\nabla_{\theta_0}\big(u(h_i^-+\epsilon_0)\big)\|_{L^{2}(S^{2n+1},\theta_0)}^2
&\leq\int_{S^{2n+1}}(h_i^-+\epsilon_0)^2|\nabla_{\theta_0}u|_{\theta_0}^2dV_{\theta_0}
+\nu
H^2\|u\|_{L^{2}(S^{2n+1},\theta_0)}^2\\
&\hspace{2mm} +\mu H
\|u\|_{L^{2}(S^{2n+1},\theta_0)}\|\nabla_{\theta_0}u\|_{L^{2}(S^{2n+1},\theta_0)}.
\end{split}
\end{equation}
Since $(h_i^-+\epsilon_0)^2\leq (h_i^-)^2+2(h_i^-+\epsilon_0)\epsilon_0<(h_i^-)^2+2\epsilon_0$, then by (\ref{8.6}), (\ref{8.8}),(\ref{8.9}) and (\ref{8.10})
we have
\begin{equation}\label{8.11}
\begin{split}
&\mbox{Vol}(S^{2n+1},\theta_0)^{\frac{1}{n+1}}\|u\|_{L^{\frac{2n+2}{n}}(S^{2n+1},\theta_0)}^2\\
&\leq\frac{4\cdot2^{\frac{n}{n+1}}}{n^2}\left(
\sum_{i=1}^k\int_{S^{2n+1}}(h_i^-+\epsilon_0)^2|\nabla_{\theta_0}u|_{\theta_0}^2dV_{\theta_0}
+\nu k
H^2\|u\|_{L^{2}(S^{2n+1},\theta_0)}^2\right.\\
&\hspace{2mm} \left.\vphantom{\int_{S^{2n+1}}}+\mu kH
\|u\|_{L^{2}(S^{2n+1},\theta_0)}\|\nabla_{\theta_0}u\|_{L^{2}(S^{2n+1},\theta_0)}\right)+2^{\frac{n}{n+1}}
\sum_{i=1}^k\|u(h_i^-+\epsilon_0)\|_{L^{2}(S^{2n+1},\theta_0)}^2\\
&\leq\frac{4\cdot2^{\frac{n}{n+1}}}{n^2}\left(
\frac{1}{2}\sum_{i=1}^k\int_{S^{2n+1}}h_i^2|\nabla_{\theta_0}u|_{\theta_0}^2dV_{\theta_0}+2k\epsilon_0\|\nabla_{\theta_0}u\|_{L^{2}(S^{2n+1},\theta_0)}^2
\right.\\
&\hspace{2mm} \left.\vphantom{\int_{S^{2n+1}}}+\nu k
H^2\|u\|_{L^{2}(S^{2n+1},\theta_0)}^2+\mu kH
\|u\|_{L^{2}(S^{2n+1},\theta_0)}\|\nabla_{\theta_0}u\|_{L^{2}(S^{2n+1},\theta_0)}\right)
+2^{\frac{n}{n+1}}
k\|u\|_{L^{2}(S^{2n+1},\theta_0)}^2.
\end{split}
\end{equation}
Note that if $\big\|h_i^+|\nabla_{\theta_0}
u|_{\theta_0}\big\|_{L^2(S^{2n+1},\theta_0)}\leq
\big\|h_i^-|\nabla_{\theta_0}
u|_{\theta_0}\big\|_{L^2(S^{2n+1},\theta_0)}$, by the same argument
above we also have (\ref{8.11}).

Young's inequality asserts that for any $\epsilon_1>0$, there exists
a constant $C(\epsilon_1)$ such that
\begin{equation}\label{8.12}
\|\nabla_{\theta_0}u\|_{L^{2}(S^{2n+1},\theta_0)}\|u\|_{L^{2}(S^{2n+1},\theta_0)}\leq
\epsilon_1\|\nabla_{\theta_0}u\|_{L^{2}(S^{2n+1},\theta_0)}^2+C(\epsilon_1)\|u\|_{L^{2}(S^{2n+1},\theta_0)}^2.
\end{equation}
Using (\ref{8.4}) and  (\ref{8.12}), (\ref{8.11}) can be written as
\begin{equation*}
\begin{split}
\mbox{Vol}(S^{2n+1},\theta_0)^{\frac{1}{n+1}}\|u\|_{L^{\frac{2n+2}{n}}(S^{2n+1},\theta_0)}^2
&\leq\frac{4}{n^2}\left[2^{-\frac{1}{n+1}}(1+3\eta)+2^{\frac{2n+1}{n+1}}k\epsilon_0
+\mu
kH\epsilon_1\right]\|\nabla_{\theta_0}u\|_{L^{2}(S^{2n+1},\theta_0)}^2\\
&\hspace{2mm}+2^{\frac{n}{n+1}}k\left[\frac{4}{n^2}\Big(\nu H^2+\mu
H C(\epsilon_1)\Big) +1\right]\|u\|_{L^{2}(S^{2n+1},\theta_0)}^2.
\end{split}
\end{equation*}
Recall by (\ref{8.5}),
$\epsilon_0^{\frac{2n+2}{n}}=\frac{n+1}{n}(\eta/k)^{\frac{2n+2}{n}}$.
Taking $\eta$ and $\epsilon_1$ small enough, we prove (\ref{8.1b}) and hence (\ref{8.1a}).
\end{proof}

\bibliographystyle{amsplain}

\end{document}